\documentclass[11pt]{amsart}
\usepackage[margin=1.2in]{geometry}
\usepackage{amsmath,amssymb,amsthm,mathtools,array,booktabs}
\usepackage{enumitem}
\usepackage{microtype}
\usepackage[T1]{fontenc}
\usepackage{lmodern,comment,placeins}
\usepackage{hyperref}
\usepackage{blkarray}

\usepackage{tikz}
\usetikzlibrary{arrows.meta,calc}
\usepackage{graphicx}
\usepackage{tabularx}
\newcolumntype{Y}{>{\centering\arraybackslash}X}

\usepackage{xcolor}
\definecolor{edgeA}{RGB}{0,114,178}    % blue
\definecolor{edgeB}{RGB}{0,100,0}   % dark green
\definecolor{edgeC}{RGB}{204,0,204}    % magenta
\definecolor{edgeD}{RGB}{139,0,0}   % dark red
\definecolor{titlecol}{RGB}{30,30,30}
\definecolor{tablegray}{gray}{0.82}

\usepackage{algorithm}
\usepackage{algpseudocode}

\theoremstyle{plain}
\newtheorem{theorem}{Theorem}[section]
\newtheorem{utheorem}{\textrm{\textbf{Theorem}}}

\newtheorem{lemma}[theorem]{Lemma}
\newtheorem{proposition}[theorem]{Proposition}

\theoremstyle{definition}
\newtheorem{defn}[theorem]{Definition}

\newtheorem{remark}[theorem]{Remark}
\newtheorem{quest}[theorem]{Question}

\newcommand{\R}{\mathbb{R}}
\newcommand{\Lob}{\mathcal{L}}
\newcommand{\ct}{\mathcal{T}}
\newcommand{\ctv}{\mathcal{C}}
\newcommand{\ip}[2]{\left[#1,#2\right]}
\newcommand{\Span}{\operatorname{span}}
\newcommand{\rank}{\operatorname{rank}}
\newcommand{\dist}{\operatorname{dist}}
\newcommand{\inn}{\operatorname{In}}
\newcommand{\diam}{\operatorname{diam}}
\newcommand{\diag}{\operatorname{diag}}

\newcommand{\A}{\operatorname{\mathcal A}}
\newcommand{\lgeq}{\succcurlyeq}

\DeclareMathOperator{\ch}{cosh}
\DeclareMathOperator{\sh}{sinh}
\DeclareMathOperator{\sech}{sech}
\DeclareMathOperator{\ach}{arcosh}
\DeclareMathOperator{\bt}{{\boldsymbol{\tau}}}
\DeclareMathOperator{\bp}{{\boldsymbol{\Pi}}}

\makeatletter
\renewcommand{\l@section}{\@tocline{1}{0pt}{0em}{}{}}
\renewcommand{\l@subsection}{\@tocline{2}{0pt}{1.5em}{}{}}
\renewcommand{\l@subsubsection}{\@tocline{3}{0pt}{3em}{}{}}
\makeatother
\title{Hyperbolic distance matrix completion}
\author[Mihai Putinar]{Mihai Putinar}
\address[M.~Putinar]{University of California at Santa Barbara, CA, USA}
\email{\tt mputinar@math.ucsb.edu}

\author[Prateek Kumar Vishwakarma]{Prateek Kumar Vishwakarma}
\address[P.K.~Vishwakarma]{Universit\'e Laval, Qu\'ebec, QC, Canada}
\email{\tt prateekv@alum.iisc.ac.in}
\date{\today}



\begin{document}

\begin{abstract}
A completion theory for hyperbolic distance data is developed at the interface of matrix analysis, graph theory, and hyperbolic geometry. Krein’s characterization of the metric space embeddability in Lobachevsky space leads to a natural anchoring procedure that transforms the indefinite data into a positive semidefinite kernel. In analogy with positive semidefinite and Euclidean distance matrix completion, chordality of the specification graph is shown to be the necessary and sufficient condition for local Lorentz-Gram data to admit global completion. Existence is complemented by explicit constructions. For trees, we obtain geodesic-rectification and product-distance completions; for chordal graphs, the latter extends to matrix-valued transfers along clique-trees. The resulting canonical completion is characterized by sparsity of its inverse and by a maximum-absolute-determinant principle. Its metric distortion exhibits a sharp dichotomy governed by clique separator size. Applications to exact recovery from sparse hyperbolic measurements and to hierarchical and phylogenetic data are developed.
\end{abstract}

\maketitle
\tableofcontents

\section{Introduction}

A partial matrix is one in which only some of the entries are prescribed, and a completion problem asks whether the unspecified entries can be chosen so that the resulting matrix belongs to a given structured class. For positive semidefinite matrices the answer is governed by the combinatorics of the graph recording the prescribed positions: Grone, Johnson, S\'a, and Wolkowicz proved that the local conditions on the fully specified principal submatrices suffice for global completion precisely when that graph is chordal \cite{Grone-Johnson-Sa-Wolkowicz}. Bakonyi and Johnson established the corresponding theorem for Euclidean distance matrices \cite{Bakonyi-Johnson}. Johnson's survey \cite{JohnsonSurvey1990} records the classical theory through 1990, and \cite{Liberti-SIAM} its distance geometry descendants. Spherical, Euclidean, and hyperbolic distance data are the three constant curvature instances of a single Gram matrix formalism, and the two theorems just quoted settle the first two. The present article treats the third, where chordality again turns out to be the exact criterion.

The Lorentzian foundations of hyperbolic geometry belong to classical geometric algebra and the theory of spaces with an indefinite metric \cite{ArtinGeometricAlgebra,IK-1,IK-2,VinbergGeometryII}, while Lorentz-Gram matrices appeared explicitly in Loewner's work on the critical exponent of the Green function \cite{LoewnerCriticalExponent}. The entrywise preservers of Lorentz-Gram matrices, and the corresponding distance preservers of Lobachevsky space, were recently classified by Belton, Guillot, Khare, and Putinar \cite{Belton-Guillot-Khare-Putinar}. Isometric embeddings of a metric space in Lobatchevsky space appeared in the pioneering work of Krein \cite{Krein1948} where a classification of the respective screw lines was discovered. Hyperbolic embeddings have also been developed for trees and hierarchical data \cite{Sarkar,Sala}, graph metrics, social networks, and geographic routing \cite{Verbeek,Kleinberg}, as well as for the completion and denoising of hyperbolic distance matrices \cite{TabaghiDokmanic2020}. Their applicability is now visible across several areas: hyperbolic geometry has been used to uncover structure in financial networks \cite{KRN}, to represent and update phylogenetic trees \cite{Jiang,Wilson}, and to organize image and video data \cite{Kh-IEEE,HOVER}.

\subsection{Lorentz-Gram completion and its historical antecedent}

The data considered in this article are partial: the entries associated with certain pairs are prescribed, and the remaining entries are unknown. Our problem is to determine when the unspecified entries can be chosen so that the resulting matrix is the Lorentz-Gram matrix of a configuration in Lobachevsky space. The prescribed entries are therefore retained, and the principal questions are the existence, structure, and \textit{explicit} construction of such completions.

The passage from partial to global Lorentz-Gram data has a classical structured antecedent in the work of Krein on indefinite Toeplitz forms and screw lines in Lobachevsky space. Krein's theory of Lobachevsky screw lines, announced \cite{Krein1948} in $1948$ and subsequently developed with Iohvidov, associates to a stationary metric function a Toeplitz kernel having one positive square. This is an early offspring of the emerging at that time theory of spaces with an indefinite metric, later on bearing the names of Pontryagin and Krein. In the first part, Iohvidov and Krein \cite{IK-1} studied complete indefinite inner-product spaces with $\kappa$ positive squares and derived Pontryagin's invariant-subspace theorem for self-adjoint operators \cite{Pontryagin1944} from its unitary analogue. In the second part \cite{IK-2}, they studied in depth the structure of finite Toeplitz forms. Particular attention is given to their continuation while preserving the prescribed number of positive squares, obtaining both uniqueness results in the degenerate case and infinitely many continuations in the nondegenerate case. A determinantal identity referring to overlapping principal minors, due to Sylvester and Jacobi is essential in the proof of the continuation of these specific indefinite Toeplitz forms (see footnote on page 343 in \cite{IK-2}). The continuous time analog is then recovered from these finite differences results and yields the extension of every finite screw arc to an infinite screw line. Thus, indefinite Lorentz-Gram completion already appears in this classical theory under the strong constraint of translation invariance. The present work replaces this Toeplitz pattern by arbitrary graph-specified data, where the corresponding local-to-global principle is no longer stationarity, but chordality. Krein's classification of Lobachevsky screw lines has also found a different recent application: it plays a key role in the classification of entrywise distance preservers for Lobachevsky space \cite{Belton-Guillot-Khare-Putinar}.

Our exact extension framework should also be distinguished from the recovery of a low-rank matrix from a sampled subset of its entries, for which Cand\'es and Tao \cite{CandesTao} established convex recovery guarantees under suitable sampling and incoherence assumptions. Here the specification graph is given, and the required structure is Lorentz-Gram realizability rather than a prescribed low rank.

The present realizability framework also differs from the optimization-based approaches commonly studied in the theory of hyperbolic embeddings. Much of that literature seeks embeddings with small distortion, prescribed dimension, or good performance with respect to a data-dependent objective \cite{Sarkar,Sala,Verbeek,Kh-IEEE,HOVER}; completion and denoising from possibly noisy observations have also been approached through semidefinite optimization \cite{TabaghiDokmanic2020}. 

Our starting point is the exact extension of the prescribed data, rather than an optimization criterion. In the nonsingular chordal setting, we further characterize the canonical completion through sparsity of its inverse and show that it uniquely maximizes the absolute determinant over all Lorentz-Gram completions. We subsequently study the metric distortion of the explicit completions constructed below.

\subsection{Three geometries, one anchoring} The place of Lorentz-Gram completion within the classical completion theory becomes clearest when the spherical, Euclidean, and Lobachevsky geometries are viewed side by side. The three constant curvature model spaces encode their distance data through Gram matrices associated with ambient bilinear forms. In each case, the admissibility of the distance data can be expressed in terms of the inertia of an appropriate matrix. What changes from one geometry to another is the signature of the ambient form and the quadric on which the representing vectors lie.

\medskip
\textit{\textbf{1.} Ambient realizations and inertia.}
\medskip

\noindent Throughout this comparison, $d_{ij}$ denotes the intrinsic geodesic distance between the points indexed by $i$ and $j$ in the geometry under consideration. On the unit sphere $S^d\subseteq\R ^{d+1}$, the ambient form $\langle \cdot, \cdot \rangle$ is positive definite and
\[
g_{ij}=\langle u_i,u_j\rangle=\cos d_{ij}.
\]
Consequently, the matrix $G=(g_{ij})$ is positive semidefinite and has diagonal entries equal to $1$. Whereas, on Lobachevsky space $\Lob^d\subseteq\R ^{1,d}$, the ambient form $\ip{\cdot}{\cdot}$ is Lorentzian and
\[
g_{ij}=[u_i,u_j]=\cosh d_{ij}.
\]
The resulting matrix $G=(g_{ij})$ has one positive square and satisfies $g_{ij}\geq1=g_{ii}$. Except in the degenerate case where all the points coincide, it is not positive semidefinite. Euclidean space admits a related description through the light cone. On $\R \oplus\R ^d\oplus\R $, consider the form
\[
[(\alpha,x,\beta),(\alpha',x',\beta')]_0
:=
\frac12\bigl(\alpha\beta'+\alpha'\beta\bigr)
-\langle x,x'\rangle.
\]
The map $x\mapsto\widehat x:=(1,x,\|x\|^2)$ takes $\R ^d$ into the light cone, since $[\widehat x,\widehat x]_0=0$; and $2[\widehat x_i,\widehat x_j]_0=d_{ij}^2$. Accordingly, Euclidean distance data are encoded by the bordered Cayley--Menger matrix
\[
\operatorname{CM}(D)
:=
\begin{pmatrix}
0&\mathbf1^t\\
\mathbf1&(d_{ij}^2)
\end{pmatrix},
\]
whose admissibility is likewise expressed by the condition that it have one positive square.

\medskip
\textit{\textbf{2.} The uniform anchoring formula.}
\medskip

\noindent The role of anchoring is the same in all three geometries. Fix an index $k$. On the sphere, orthogonally removing the component in the direction of $u_k$ gives
\[
\bigl(g_{ij}-g_{ik}g_{jk}\bigr)_{i,j}
=
\bigl(
\langle u_i-g_{ik}u_k,\,
u_j-g_{jk}u_k\rangle
\bigr)_{i,j}
\lgeq0.
\]
For points $x_i$ in Euclidean space, the corresponding anchored Euclidean Gram matrix, commonly called the Schoenberg matrix or a modified Cayley--Menger matrix, is
\[
\left(
\frac{d_{ik}^2+d_{jk}^2-d_{ij}^2}{2}
\right)_{i,j}
=
\bigl(
\langle x_i-x_k,x_j-x_k\rangle
\bigr)_{i,j}
\lgeq0.
\]
Finally, in Lobachevsky space the orthogonal complement of $u_k$ is negative definite, and hence
\[
\bigl(g_{ik}g_{jk}-g_{ij}\bigr)_{i,j}
=
\bigl(
-[u_i-g_{ik}u_k,\,
u_j-g_{jk}u_k]
\bigr)_{i,j}
\lgeq0.
\]
Thus, in all three geometries, anchoring converts the distance data into a positive semidefinite matrix. Their common origin becomes visible once the curvature is made explicit: the three transforms are instances of a single curvature-dependent formula. More precisely, let $\kappa\in\{1,0,-1\}$ denote the curvature, and define the ``generalized cosine'' $c_\kappa$ and $a_{ij}$ by
\[
c_\kappa(t):=
\begin{cases}
\cos t,  & \kappa=1,\\
1,       & \kappa=0,\\
\cosh t, & \kappa=-1,
\end{cases}
\qquad\qquad
a_{ij}:=
\begin{cases}
\frac{1}{\kappa}\big({c_\kappa(d_{ij})-1}\big), & \kappa\neq0,\\[1mm]
-\frac12d_{ij}^2,                  & \kappa=0.
\end{cases}
\]
With this normalization, $a_{ii}=0$, and $-2a_{ij}$ is the positive quadratic quantity associated with the difference of the corresponding ambient representatives:
\[
-2a_{ij}
=
\begin{cases}
\langle u_i-u_j,u_i-u_j\rangle, &\quad \kappa=1,\\
\langle x_i-x_j,x_i-x_j\rangle
, &\quad \kappa=0,\\
-[u_i-u_j,u_i-u_j], &\quad \kappa=-1.
\end{cases}
\]
Indeed, in the spherical case $\langle u_i,u_j\rangle=\cos d_{ij}$, while in the Lobachevsky case $[u_i,u_j]=\cosh d_{ij}$, so the three quantities above are $2(1-\cos d_{ij})$,~$d_{ij}^2$,~$2(\cosh d_{ij}-1)$, respectively. Anchoring at $k$ now gives the single formula
\begin{equation}\label{eq:uniform-anchor}
T_k(i,j)
:=
a_{ij}-a_{ik}-a_{jk}-\kappa a_{ik}a_{jk}.
\end{equation}
For $\kappa\neq0$, this can equivalently be written as $T_k(i,j)=\frac{1}{\kappa}\big({g_{ij}-g_{ik}g_{jk}}\big)$ where $g_{ij}:=c_\kappa(d_{ij})$; whereas at $\kappa=0$ it becomes $T_k(i,j)= \frac12\bigl(d_{ik}^2+d_{jk}^2-d_{ij}^2\bigr)$. Thus,
\[
T_k(i,j)
=
\begin{cases}
\cos d_{ij}-\cos d_{ik}\cos d_{jk},
    & \kappa=1,\\
\frac12\bigl(d_{ik}^2+d_{jk}^2-d_{ij}^2\bigr),
    & \kappa=0,\\
\cosh d_{ik}\cosh d_{jk}-\cosh d_{ij},
    & \kappa=-1,
\end{cases}
\qquad
\mbox{and}\qquad T_k\lgeq0.
\]

\medskip
\textit{\textbf{3.} Curvature, Schur complements, and dimension.}
\medskip

\noindent Formula~\eqref{eq:uniform-anchor} also displays the effect of curvature. At $\kappa=0$, the quadratic term disappears and one obtains the additive Euclidean transform. For nonzero curvature, the term $-\kappa a_{ik}a_{jk}$ is the corresponding curvature correction; in the Lobachevsky case it introduces the product term that foreshadows the product-distance and clique-product formulas appearing later in the paper. There is also a common Schur complement interpretation. When $\kappa\neq0$, the matrix $T_k$ is the Schur complement of the $1\times1$ anchor block $(g_{kk})=(1)$, normalized by $\kappa$. At $\kappa=0$, the anchor vector is null, and the corresponding $1\times1$ block is singular. One instead uses the block
\[
\begin{pmatrix}
0&1\\
1&0
\end{pmatrix}
\]
formed by the border and the $k$-th index of the Cayley--Menger matrix. The Schur complement of this block has $(i,j)$-entry $d_{ij}^2-d_{ik}^2-d_{jk}^2=-2T_k(i,j)$. The Euclidean construction is therefore the limiting form of the same anchoring procedure.

The sign of the curvature determines the sign of the orthogonal complement of the anchor. It is positive definite on the sphere and negative definite in Lobachevsky space, accounting for the opposite orders of the terms in the two curved cases. Moreover,
\[
\diag T_k
=
\begin{cases}
(\sin^2d_{ik})_i,  & \kappa=1,\\
(d_{ik}^2)_i,      & \kappa=0,\\
(\sinh^2d_{ik})_i, & \kappa=-1,
\end{cases}
\]
and $\rank T_k$ gives the minimal dimension of the corresponding realization.

\medskip
\textit{\textbf{4.} The three completion theories.}
\medskip

\noindent This comparison places the present paper within the classical completion theory. For spherical distance data, completion reduces to positive semidefinite completion and is governed by the theorem of Grone, Johnson, S\'a, and Wolkowicz \cite{Grone-Johnson-Sa-Wolkowicz}. For Euclidean distance data, the corresponding theory was developed by Bakonyi and Johnson \cite{Bakonyi-Johnson}. The present paper develops the analogous completion theory for Lobachevsky distance data. The reversal from
\[
|g_{ij}|\leq1=g_{ii}
\quad\text{on the sphere}
\qquad\text{to}\qquad
g_{ij}\geq1=g_{ii}
\quad\text{in Lobachevsky space}
\]
is precisely what replaces positive semidefiniteness of the full Gram matrix by the condition that it have one positive square. Table~\ref{tab:three-geometries} summarizes these correspondences.

{\begin{table}[H]
\centering
\footnotesize
\renewcommand{\arraystretch}{1.18}
\setlength{\tabcolsep}{3pt}
\begin{tabularx}{\textwidth}{
    @{}
    >{\raggedright\arraybackslash}p{2.65cm}
    YYY
    @{}
}
\toprule
&
\textbf{Sphere $S^d$}
&
\textbf{Euclidean $\R ^d$}
&
\textbf{Lobachevsky $\Lob^d$}
\\
\midrule

Curvature $\kappa$
&
$1$
&
$0$
&
$-1$
\\

\addlinespace
Ambient form and signature
&
$\langle\cdot,\cdot\rangle$ on $\R ^{d+1}$,
$(d+1,0)$
&
$[\cdot,\cdot]_0$ on $\R ^{d+2}$,
$(1,d+1)$
&
$[\cdot,\cdot]$ on $\R ^{d+1}$,
$(1,d)$
\\

\addlinespace
Model set
&
$\langle u,u\rangle=1$
&
$[\widehat x,\widehat x]_0=0$,
$\widehat x=(1,x,\|x\|^2)$
&
$[u,u]=1$, $u_0>0$
\\

\addlinespace
\addlinespace
$g_{ij}=c_\kappa(d_{ij})$
&
$\cos d_{ij}$
&
$1$
&
$\cosh d_{ij}$
\\

\addlinespace
$a_{ij}$
&
$\cos d_{ij}-1$
&
$-\dfrac12d_{ij}^2$
&
$1-\cosh d_{ij}$
\\

\addlinespace
$-2a_{ij}$,
the ambient chordal square distance
&
$\|u_i-u_j\|^2$
&
$d_{ij}^2$
&
$-[u_i-u_j,u_i-u_j]$
\\

\addlinespace
Admissibility criterion
&
$G \lgeq 0$,
$g_{ii}=1$
&
$\left(
\begin{smallmatrix}
0&\mathbf1^{t}\\
\mathbf1&(d_{ij}^2)
\end{smallmatrix}
\right)$
has one positive square
&
$G$ has one positive square,
$g_{ii}=1$, $g_{ij}\geq1$
\\

\addlinespace
Position of the diagonal
&
$|g_{ij}|\leq1=g_{ii}$
&
$a_{ij}\leq0=a_{ii}$
&
$g_{ij}\geq1=g_{ii}$
\\

\addlinespace
Anchor block at $k$
&
$(g_{kk})=(1)$
&
$\left(
\begin{smallmatrix}
0&1\\
1&0
\end{smallmatrix}
\right)$
using the border and $k$
&
$(g_{kk})=(1)$
\\

\addlinespace
$T_k(i,j)$ from
\eqref{eq:uniform-anchor}
&
$g_{ij}-g_{ik}g_{jk}$
&
$\dfrac12
 \bigl(d_{ik}^2+d_{jk}^2-d_{ij}^2\bigr)$
&
$g_{ik}g_{jk}-g_{ij}$
\\

\addlinespace
Gram form of $T_k$
&
$\langle
u_i-g_{ik}u_k,\,
u_j-g_{jk}u_k
\rangle$
&
$\langle
x_i-x_k,\,
x_j-x_k
\rangle$
&
$-[
u_i-g_{ik}u_k,\,
u_j-g_{jk}u_k
]$
\\

\addlinespace
Diagonal of $T_k$
&
$\sin^2d_{ik}$
&
$d_{ik}^2$
&
$\sinh^2d_{ik}$
\\

\addlinespace
Minimal embedding dimension
&
$\rank T_k$
&
$\rank T_k$
&
$\rank T_k$
\\

\addlinespace
Completion theory
&
Grone--Johnson--S\'a--Wolkowicz
\cite{Grone-Johnson-Sa-Wolkowicz}
&
Bakonyi--Johnson
\cite{Bakonyi-Johnson}
&
{{Present article}}
\\

\bottomrule
\end{tabularx}

\caption{Spherical, Euclidean, and Lobachevsky distance data as Gram matrices. In each column, $k$ is a fixed anchor index and $T_k$ is the positive semidefinite matrix obtained from \eqref{eq:uniform-anchor}. The Euclidean column is understood as the value at $\kappa=0$, where the anchor vector is null.}
\label{tab:three-geometries}
\end{table}}
\FloatBarrier

Regardless of mentioning that the completion of structured matrices is an old subject which has surprisingly reappeared in an array of contexts: moment problems, interpolation of bounded analytic functions, control theory, prediction of time series, operator dilations, network analysis, the organization of statistical data, to name only a few.

\section{Main results}

We now formulate the Lorentz-Gram completion problem and state the principal results of the paper. The discussion proceeds mainly in four stages.

\begin{enumerate}
\item
We first characterize the fully specified kernels (Definition~\ref{defn:LG-completion}) that arise from families of points in Lobachevsky space.

\item
We then determine exactly which specification graphs (Definition~\ref{defn:LG-completion}) guarantee completion: they are precisely the chordal graphs.

\item
Next, we construct explicit completions for trees and, more generally, for chordal graphs, and realize them algorithmically through orthogonal innovation.

\item
Finally, we study the distortion of the hyperbolic embeddings arising from these Lorentz-Gram completions and obtain two complementary results: the product-distance completion for trees is asymptotically an isometry, whereas the clique-product completion for chordal graphs exhibits a sharp dichotomy according to separator size.
\end{enumerate}

We will focus on the hyperboloid model of Lobachevsky space.

\begin{defn}[Hyperboloid model of Lobachevsky space]
\label{defn:lob-space}
Let $H$ be a real Hilbert space, and equip $\R\oplus H$ with the Lorentz form $[(x_0,{\bf x}_1),(y_0,{\bf y}_1)]:=x_0y_0-\langle {\bf x}_1,{\bf y}_1\rangle_H$. This form has positive index one, and its associated quadratic form is $[(x_0,{\bf x}_1),(x_0,{\bf x}_1)]=x_0^2-\|{\bf x}_1\|^2$. The hyperboloid model of Lobachevsky space is the positive sheet of the unit hyperboloid,
\[
\Lob(H)
:=
\bigl\{(x_0,{\bf x}_1)\in\R\oplus H:
x_0^2-\|{\bf x}_1\|^2=1,\ x_0>0\bigr\}.
\]
It is equipped with the hyperbolic distance $d_{\Lob}(u,v):=\ach [u,v]$ for all vectors $u,v\in\Lob(H)$. Equivalently we have, $[u,v]=\cosh d_{\Lob}(u,v)$.
\end{defn}

We can now formulate the completion problem studied in this paper.

\begin{defn}[Specification graph and Lorentz-Gram completion]\label{defn:LG-completion}
Let $V$ be a nonempty possibly infinite set.
\begin{enumerate}
\item A partial symmetric kernel on $V$ is a possibly infinite matrix $A=(a_{ij})_{i,j\in V}$ in which only some entries are prescribed, subject to the following conditions: every diagonal entry $a_{ii}$ is prescribed; whenever $a_{ij}$ is prescribed, so is $a_{ji}$; and the prescribed entries satisfy $a_{ij}=a_{ji}$. 
    
\item The specification graph of $A$ is the simple undirected graph $G:=G(A):=(V,E)$ whose vertices index the rows and columns of $A$, with $\{i,j\}\in E$, for $i\neq j$, precisely when $a_{ij}$ is prescribed. In particular, if $G$ is the complete graph, then $A$ is called fully specified.
    
\item A completion of $A$ is a fully specified symmetric kernel $B=(b_{ij})_{i,j\in V}$ such that $b_{ij}=a_{ij}$ for every prescribed entry $a_{ij}$. A completion $B=(b_{ij})_{i,j\in V}$ of $A$ is called a Lorentz-Gram completion if there exist a real Hilbert space $H$ and vectors $u_i\in\Lob(H)$, $i\in V$, such that $b_{ij}=[u_i,u_j]$ for all $i,j\in V$.
\end{enumerate}
If $V$ is finite, the corresponding kernel is called a matrix.
\end{defn}

The completion problem naturally separates into two questions. The first is to recognize the fully specified kernels that occur as Lorentz-Gram kernels; the second is to determine when prescribed partial data can be extended to a kernel of this kind. We begin with the first question. Fixing one point as an anchor converts the Lorentzian realization condition into positive semidefiniteness of an associated kernel, while the condition of having one positive square gives an equivalent characterization that is independent of the choice of anchor.

\subsection{Lorentz-Gram kernel recognition}

We formally define positive semidefinite kernel, and kernel with one positive square.

\begin{defn}\label{defn:pos-ker} Let $I$ be a nonempty set.

\begin{enumerate}
\item A real symmetric kernel $P=(p_{ij})_{i,j\in I}$ is called positive semidefinite if, for every finite subset $F\subseteq I$, the matrix $(p_{ij})_{i,j\in F}$ is positive semidefinite.

\item A real symmetric kernel $K=(k_{ij})_{i,j\in I}$ is said to have at most one positive square if every finite restriction $(k_{ij})_{i,j\in F}$ has at most one positive eigenvalue, counted with multiplicity. It is said to have one positive square if, in addition, at least one finite restriction has a positive eigenvalue. 
\end{enumerate}
When $k_{ii}=1$ for every $i\in I$, the latter condition is automatic.
\end{defn}

The recognition of kernels arising from an indefinite inner-product space is classical in the theory of Krein and Pontryagin spaces; in particular, kernels with one positive square admit realizations in a space of positive index one. In the present normalized setting, this general principle takes a particularly concrete form. Namely, fixing one index produces an anchored positive semidefinite kernel, while the condition of having one positive square gives an equivalent intrinsic criterion, independent of the chosen anchor. The following formulation is a convenient specialization of these classical principles to Lorentz-Gram kernels. For the corresponding finite-dimensional characterization of Lorentz-Gram matrices by the presence of exactly one positive eigenvalue, see also \cite[Theorem~2.2]{Belton-Guillot-Khare-Putinar}.

\begin{utheorem}[Lorentz-Gram kernel]
\label{T:LG-ker-char}
Let $I$ be a nonempty set, and let $K:=(k_{ij})_{i,j\in I}$ be a real symmetric kernel satisfying $k_{ij}\geq k_{ii}=1$ for all $i,j\in I$. For each $i_0\in I$, define the anchored kernel $P_{i_0}=\bigl(p^{(i_0)}_{ij}\bigr)_{i,j\in I}$ by the following, for all $i,j\in I$:
\[
p^{(i_0)}_{ij}
:=
k_{ii_0}k_{ji_0}-k_{ij}.
\]
Then the following statements are equivalent:
\begin{enumerate}
\item
There exist a real Hilbert space $H$ and a family $(u_i)_{i\in I}\subseteq\Lob(H)$ such that $k_{ij}=[u_i,u_j]$ for all $i,j\in I$. In this case, $K$ is called a Lorentz-Gram kernel, or a Lorentz-Gram matrix if $I$ is finite.

\item
The kernel $P_{i_0}$ is positive semidefinite for some $i_0\in I$.

\item
The kernel $P_{i_0}$ is positive semidefinite for every $i_0\in I$.

\item
The kernel $K$ has one positive square.
\end{enumerate}
Whenever these conditions hold, $p^{(i_0)}_{i_0,j}=0$ for every $j\in I$. Moreover, the minimal dimension of a Hilbert space $H$ in a realization in~{\rm(1)} is the supremum of $\rank \bigl(p^{(i_0)}_{ij}\bigr)_{i,j\in F}$ where $F\subseteq I$ runs over all finite sets.
\end{utheorem}

Theorem~\ref{T:LG-ker-char} turns the geometric definition of a Lorentz-Gram kernel into an algebraic criterion: a normalized symmetric kernel with entries at least one is Lorentz-Gram precisely when it has one positive square, or, equivalently, when any of its anchored kernels is positive semidefinite. The rank of the anchored kernel also gives the minimal dimension of a realization. Related finite-dimensional questions concerning sufficient conditions for a symmetric matrix to have exactly one positive eigenvalue, including conditions involving entrywise functions, are studied by Al-Saafin and Garloff~\cite{Garloff}. Having thus settled the recognition problem for fully specified kernels, we now turn to the completion problem for partial Lorentz-Gram data.

\subsection{Chordality as the completion criterion}

Let $V$ be a finite set, and let $A=(a_{ij})_{i,j\in V}$ be a partial symmetric matrix with specification graph $G=(V,E)$. A necessary condition for $A$ to admit a Lorentz-Gram completion is that every fully specified clique principal submatrix of $A$ be a Lorentz-Gram matrix. Since principal submatrices of Lorentz-Gram matrices are again Lorentz-Gram, it is enough to impose this condition on the maximal cliques of $G$.

\begin{defn}[Induced subgraph, clique, chordal graph, clique principal submatrix]
Let $V$ be a finite nonempty set, and let $G=(V,E)$ be a finite simple graph.
\begin{enumerate}
\item
For a nonempty subset $V'\subseteq V$, the subgraph of $G$ induced by $V'$ is the graph with vertex set $V'$ whose edges are precisely the edges of $G$ having both endpoints in $V'$.

\item
A nonempty subset $C\subseteq V$ is called a clique of $G$ if the subgraph induced by $C$ is complete. A clique is maximal if it is not properly contained in any other clique of $G$.

\item
The graph $G$ is called chordal if it contains no induced cycle of length at least four.

\item
Let $A$ be a partial symmetric matrix with specification graph $G$, and let $V'\subseteq V$ be nonempty. The principal submatrix of $A$ indexed by $V'$ is called a clique principal submatrix if the subgraph of $G$ induced by $V'$ is a clique. In this case, the principal submatrix is fully specified. If $V'$ is a maximal clique of $G$, it is called a maximal clique principal submatrix.
\end{enumerate}
\end{defn}

Chordality is the fundamental graph-theoretic condition in several classical matrix and distance completion problems: it governs positive definite completion through the theorem of Grone, Johnson, S\'a, and Wolkowicz~\cite{Grone-Johnson-Sa-Wolkowicz}, and Euclidean distance completion through the work of Bakonyi and Johnson~\cite{Bakonyi-Johnson}. Related inheritance principles for positive completions of partial operator matrices were developed by Bakonyi and Constantinescu~\cite{Bakonyi-Constantinescu}. The following result shows that the same class of chordal graphs governs the Lorentz-Gram completion.

\begin{utheorem}[Lorentz-Gram completion and chordal graphs]
\label{T:chordal-characterization}
For a finite simple graph $G=(V,E)$, the following statements are equivalent.
\begin{enumerate}
\item
The graph $G$ is chordal.

\item
Every partial symmetric matrix $A=(a_{ij})_{i,j\in V}$ with specification graph $G$ all of whose fully specified clique principal submatrices are Lorentz-Gram matrices admits a Lorentz-Gram completion.
\end{enumerate}
Consequently, if $G$ is chordal, then a partial symmetric matrix with specification graph $G$ admits a Lorentz-Gram completion if and only if each of its maximal clique principal submatrices is a Lorentz-Gram matrix.
\end{utheorem}

Thus, for chordal specification graphs, the local Lorentz-Gram conditions on the maximal cliques are sufficient for global completion. For nonchordal graphs this fails: there are partial matrices for which every fully specified clique principal submatrix is Lorentz-Gram, while no Lorentz-Gram completion exists. Chordality therefore marks exactly the passage from local clique compatibility to global realizability. We next turn from existence to explicit completion schemes.

\subsection{Three explicit completion schemes}
\label{SS:explicit-constructions}

Theorem~\ref{T:chordal-characterization} determines when the clique conditions guarantee the existence of a Lorentz-Gram completion, but does not specify how such a completion should be chosen or realized. For both structural questions and possible applications, it is therefore important to identify completions that admit explicit realizations. Such realizations make it possible to recover concrete configurations in Lobachevsky space and to analyze the geometry and distortion of the resulting embeddings. They also reveal connections with the product-distance matrices of Bapat and Sivasubramanian \cite{BapatSivasubramanian2012}, the simplicial elimination theory of Dirac \cite{Dirac1961} and Fulkerson and Gross \cite{FulkersonGross1965}, and the clique-tree theory of Buneman \cite{Buneman1974Rigid} and Gavril \cite{Gavril1974}; see also Blair and Peyton \cite{BlairPeyton1993}. The canonical chordal construction also connects with the inverse formula of Johnson and Lundquist \cite{Johnson-Lundquist}. Its inverse-sparsity and variational aspects are compared in Remark~\ref{rem:covariance-moment-comparision} with Dempster's covariance selection \cite{Dempster1972} and with the inverse-zero theory for truncated moment matrices developed by Helton, Lasserre, and Putinar \cite{Helton-Lasserre-Putinar}.

We describe three constructions. For trees, we obtain two generally different completions: one places the entire tree on a single hyperbolic geodesic, while the other produces the product-distance completion by introducing a new orthogonal direction at each vertex. The latter construction extends from trees to chordal graphs, where scalar products along vertex paths are replaced by matrix products across clique separators.

\subsubsection{Signed paths on a geodesic}\label{SS:signed-path}

The first construction places all vertices of the tree on a common hyperbolic geodesic, with the sign assigned to each edge recording the direction of the corresponding displacement.

In Subsections~\ref{SS:signed-path} and~\ref{SS:prod-dist}, as well as in the later sections containing the corresponding proofs, we also allow infinite graphs. Throughout these sections, a tree is a connected simple graph containing no cycles, and its vertex set need not be finite unless stated otherwise. Thus, every two vertices of a tree are joined by a unique finite path.

\begin{defn}[Signed-path coordinates]\label{defn:signed-path-coordinates}
Let $V$ be a nonempty possibly infinite set, and let $A=(a_{ij})_{i,j\in V}$ be a partial symmetric kernel such that $a_{ii}=1$ for every $i\in V$ and $a_{ij}\geq1$ whenever $\{i,j\}\in E$. Suppose its specification graph $G=(V,E)$ is a tree. Fix an anchor vertex $0\in V$, and regard $G$ as rooted at $0$. For each $i\in V$, let $P(0,i)$ denote the unique path from $0$ to $i$, viewed also as its set of edges. For every edge $e=\{i,j\}\in E$, we define $\lambda_e:=\ach(a_{ij})$. An edge-signing is a map $\varepsilon:E\to\{\pm1\}$; we write $\varepsilon_e:=\varepsilon(e)$. The corresponding signed-path coordinates are defined by $\bt^\varepsilon: V \to \R$:
\[
\bt_0^\varepsilon:=\bt^\varepsilon(0):=0,
\qquad
\bt_i^\varepsilon:=\bt^\varepsilon(i)
:=
\sum_{e\in P(0,i)}\varepsilon_e\lambda_e,
\qquad i\neq0.
\]
When the edge-signing is fixed, we suppress the superscript
$\varepsilon$.
\end{defn}

\begin{utheorem}[Signed-path Lorentz-Gram completion]
\label{T:tree-completion}
Retain the hypotheses of Definition~\ref{defn:signed-path-coordinates}. Choose an edge-signing $\varepsilon:E\to\{\pm1\}$, and let $(\bt_i)_{i\in V}$ be the corresponding signed-path coordinates. Then the symmetric kernel $B=(b_{ij})_{i,j\in V}$, defined by
\[
b_{ij}:=\ch(\bt_i-\bt_j)
\]
is a Lorentz-Gram completion of $A$. More precisely, $B$ is the Lorentz-Gram kernel of the points $u_i:=\bigl(\ch\bt_i,\sh\bt_i\bigr)\in\Lob(\R)$, $i\in V$. Its anchor row and column are given by $b_{i0}=b_{0i}=\ch(\bt_i)$. Moreover, the anchor transform $\Pi_0(B):=\bigl(b_{i0}b_{j0}-b_{ij}\bigr)_{i,j\in V}$ is positive semidefinite and has rank at most one (Theorem~\ref{T:LG-ker-char}). In fact, $b_{i0}b_{j0}-b_{ij}=\sh(\bt_i)\sh(\bt_j)$.
\end{utheorem}

The geometric meaning of the edge-signing, the resulting dependence of the completion on the edge-signing, and an algorithm for constructing the realization are discussed in Section~\ref{S:geodesicrectification}.

\subsubsection{Products along tree paths}\label{SS:prod-dist}

A second tree completion is given by the product-distance matrix of Bapat and Sivasubramanian \cite{BapatSivasubramanian2012}, which is related to the exponential distance matrices studied by Bapat, Lal, and Pati \cite{BapatLalPati2006}. The $(i,j)$-entry of this matrix is the product of the prescribed entries along the unique path from $i$ to $j$. Consequently, its entrywise logarithm yields the additive tree pseudometric with edge lengths $\log a_e$, leading to further connections with the classical four-point condition, as noted later.

\begin{utheorem}[Product-distance Lorentz-Gram completion]
\label{T:tree-path-product-completion}
Let $V$ be a nonempty possibly infinite set, and let $A=(a_{ij})_{i,j\in V}$ be a partial symmetric matrix whose specification graph is a tree $G=(V,E)$. Assume that $a_{ii}=1$ for all $i\in V$ and $a_{ij}\geq1$ for all $\{i,j\}\in E$. For an edge $e=\{i,j\}$, write $a_e:=a_{ij}$. For $i,j\in V$, let $P(i,j)$ denote the unique path from $i$ to $j$, viewed as its set of edges, and define
\begin{align*}
\label{eq:canonical-tree-completion}
b_{ij}:=\prod_{e\in P(i,j)}a_e,
\end{align*}
where the empty product is equal to $1$. Then $B=(b_{ij})_{i,j\in V}$ is a Lorentz-Gram completion of $A$. Moreover, $\log [B]:=(\log b_{ij})_{i,j\in V}$ is exactly the additive tree pseudometric with edge lengths $\log a_e$, and it is a metric when each $a_e>1$.
\end{utheorem}

In summary, the aforementioned product-distance completion preserves the prescribed hyperbolic lengths $\ach(a_e)$ on the edges, while its logarithm recovers the additive tree geometry associated with the transformed edge lengths $\log a_e$.

\begin{remark}[Connection with classical tree metrics]
\label{R:product-distance-tree-metrics}
As noted above, the entrywise logarithm of the completion in Theorem~\ref{T:tree-path-product-completion} recovers an additive tree pseudometric:
\[
\delta(i,j):=\log b_{ij}
=\sum_{e\in P(i,j)}w_e,
\qquad
w_e:=\log a_e\geq0.
\]
If every $a_e>1$, then all $w_e>0$, so $\delta$ is a metric. These weights differ from the prescribed hyperbolic edge lengths $\lambda_e=\operatorname{arcosh}(a_e)$: explicitly, $w_e=\log\cosh\lambda_e$. The tree structure implies that $\delta$ satisfies the \emph{four-point condition}:
\[
\delta(i,j)+\delta(k,\ell)
\leq
\max\bigl\{
\delta(i,k)+\delta(j,\ell),
\delta(i,\ell)+\delta(j,k)
\bigr\}
\]
for all $i,j,k,\ell\in V$, including when some edge weights vanish. Equivalently, among
\[
b_{ij}b_{k\ell},
\qquad
b_{ik}b_{j\ell},
\qquad
b_{i\ell}b_{jk},
\]
the two largest quantities are equal. Thus the classical four-point condition has a direct multiplicative expression in the entries of the completion.

Conversely, Buneman's theorem~\cite[Theorem~2]{Buneman} states that a metric on a finite set satisfies the four-point condition if and only if it is induced by a positively weighted tree containing that set among its vertices; additional vertices may be necessary. In particular, every finite tree metric $d$ gives a Lorentz-Gram matrix $(e^{d(i,j)})_{i,j}$: apply Theorem~\ref{T:tree-path-product-completion} to a realizing tree with edge entries $a_e=e^{w_e}$, and restrict the resulting matrix to the original index set.

An earlier discrete counterpart is due to Zarecki\u{\i}~\cite{Zareckii}: an integer-valued metric on a finite set is realizable by distances between distinguished vertices of a finite unweighted tree precisely when it satisfies the four-point condition and every sum $d(i,j)+d(j,k)+d(k,i)$ is even. His proof also gives an explicit reconstruction algorithm.
\end{remark}

The completion in Theorem~\ref{T:tree-path-product-completion} depends only on the underlying tree, and not on the choice of a root or the construction procedure, unlike the completion in Theorem~\ref{T:tree-completion}. Moreover, if $V$ is finite and every prescribed edge entry is strictly greater than one, the completion is nonsingular, and its inverse is supported on the diagonal and the edges of the tree. Its determinant, inertia, and inverse, as well as the comparison between its induced hyperbolic distance and the corresponding logarithmic tree distance, are all explicit.

\begin{proposition}[Properties of $B$ in Theorem~\ref{T:tree-path-product-completion}]
\label{P:tree-product-distance-structure}
The following statements hold about the Lorentz-Gram completion $B=(b_{ij})_{i,j\in V}$ in Theorem~\ref{T:tree-path-product-completion} whenever $V$ is finite.

\begin{enumerate}
\item
The matrix $B$ is congruent to $(1)\oplus\bigoplus_{e\in E}(1-a_e^2)$, and its determinant is given by $\det B=\prod_{e\in E}(1-a_e^2)$. Moreover, if $E_{>}:=\{e\in E:a_e>1\}$ and $E_{=}:=\{e\in E:a_e=1\}$, then its inertia is given by $\inn(B)= \bigl(1,|E_{>}|,|E_{=}|\bigr)$.

\item
Therefore, $B$ is nonsingular if and only if $a_e>1$ for every $e\in E$. In this case, its inverse is supported on the diagonal and the edges of $G$, and is given by
\[
(B^{-1})_{ij}
=
\begin{cases}
\displaystyle
1+\sum_{\ell\sim i} {a_{i\ell}^2}{(1-a_{i\ell}^2)^{-1}},
& i=j,\\
\displaystyle
-{a_{ij}}{(1-a_{ij}^2)^{-1}},
& \{i,j\}\in E,\\
0,
& \text{otherwise}.
\end{cases}
\]

\item The induced hyperbolic pseudometric lies within $\log 2$ of the logarithmic tree pseudometric. More precisely, the hyperbolic distance $d_B(i,j):=\ach (b_{ij})$ determined by $B$ satisfies
\[
\delta(i,j)
\leq
d_B(i,j)
\leq
\delta(i,j)+\log 2,
\]
where $\delta(i,j):=\log b_{ij}$ is the additive tree pseudometric induced by the edge weights $\log a_e$; it is a metric when $a_e>1$ for every $e\in E$.
\end{enumerate}
\end{proposition}

The orthogonal innovation realization of this completion, together with an algorithm for constructing the realizing points, is developed in Subsection~\ref{SS:orthogonal-innovation}.

\subsubsection{Transfers along clique-tree paths}

Let $G$ now be a simple, finite, connected, and chordal graph. Its maximal cliques can be organized into a clique-tree, whose edges record the clique separators through which the maximal cliques meet \cite{Buneman1974Rigid,Gavril1974,BlairPeyton1993}. The unique path between two maximal cliques then plays the role of the unique vertex path in a tree. Accordingly, the scalar path products of Theorem~\ref{T:tree-path-product-completion} are replaced by matrix-valued products through the successive clique-tree separators.

\begin{defn}[Clique-trees and separators]
Let $G=(V,E)$ be a finite connected chordal graph, and let $\mathcal C(G)$ denote the family of its maximal cliques. A clique-tree of $G$ is a tree $\ct$ whose vertex set is $\mathcal C(G)$ and such that, for every $v\in V$, the maximal cliques containing $v$ form a connected subtree of $\ct$. For each edge $CD$ of $\ct$, the set $C\cap D$ is called the clique-tree separator, or simply the separator, associated with $CD$.
\end{defn}

Every finite connected chordal graph admits a clique-tree. Deleting an edge $CD$ from $\ct $ separates $\ct$ into two components. The unions of the cliques in these two components meet precisely in $C\cap D$, and there are no edges of $G$ between their respective complements of $C\cap D$. The clique-tree characterization of chordal graphs and its further basic properties are recalled in Theorem~\ref{T:clique-trees}.

\begin{utheorem}[Clique-product Lorentz-Gram completion]
\label{T:clique-product-completion}
Let $A=(a_{ij})_{i,j\in V}$ be a partial symmetric matrix whose specification graph $G=(V,E)$ is finite, connected, and chordal, and let $\ct$ be a clique-tree of $G$. Assume that $A_{C\times C}$ is a nonsingular Lorentz-Gram matrix for every maximal clique $C$ of $G$. Fix $i,j\in V$, and choose maximal cliques $C_0$ and $C_m$ containing $i$ and $j$, respectively. If $C_0=C_m$, set $b_{ij}:=a_{ij}$. Otherwise, let $C_0,C_1,\ldots,C_m$ be the unique path from $C_0$ to $C_m$ in $\ct$, and let $S_r:=C_{r-1}\cap C_r$ be the clique separators for $r=1,\ldots,m$. Define
\[
b_{ij}:=
A_{\{i\}\times S_1}A_{S_1\times S_1}^{-1}
A_{S_1\times S_2}A_{S_2\times S_2}^{-1}
\cdots
A_{S_{m-1}\times S_m}A_{S_m\times S_m}^{-1}
A_{S_m\times\{j\}},
\]
with the evident shortened form when $m=1$. Then $b_{ij}$ is independent of the choices of $C_0$ and $C_m$, and $B=(b_{ij})_{i,j\in V}$ is a nonsingular Lorentz-Gram completion of $A$.
\end{utheorem}

The clique-product completion has canonical and structural properties, summarized next. Central to these is the well-known inverse formula of Johnson and Lundquist \cite{Johnson-Lundquist}, expressed in terms of maximal cliques and clique separators, and the geometric realization through the simplicial elimination characterization of chordality.

\begin{utheorem}[Uniqueness and orthogonal innovation realization]
\label{T:chordal-canonicity}
Retain the hypotheses of Theorem~\ref{T:clique-product-completion}, and let $n:=|V|$. Fix a clique-tree $\ct$ of $G$, and denote by $B_{\ct}$ the corresponding clique-product completion. Then:

\begin{enumerate}
\item The inverse of $B_{\ct}$ is given by the Johnson--Lunquist clique separator formula \cite{Johnson-Lundquist}:
\begin{align}\label{eq:john-lund-2}
B_{\ct}^{-1}
&=
\sum_{C\in\mathcal C(G)}
\iota_C\!\left(A_{C\times C}^{-1}\right)
-
\sum_{CD\in E(\ct)}
\iota_{C\cap D}\!\left(
A_{(C\cap D)\times(C\cap D)}^{-1}
\right).
\end{align}
Here, for $I\subseteq V$ and a matrix $X$ indexed by $I$, let $\iota_I(X)$ denote its extension to a $V\times V$ matrix by zeros outside $I\times I$. In particular, the inverse of $B_{\ct}$ is supported on the diagonal and the edges of $G$; equivalently, $(B_{\ct}^{-1})_{ij}=0$ whenever $i\neq j$ and $\{i,j\}\notin E$. Furthermore, the determinant of $B_{\ct}$ is given by

\begin{align}\label{eq:det-clique-tree-comp}
\det B_{\ct}=\bigg{[}{\displaystyle\prod_{C\in\mathcal C(G)}\det A_{C\times C}}\bigg{]}~ 
\bigg{[}{\displaystyle\prod_{CD\in E(\ct)}\det A_{(C\cap D)\times(C\cap D)}}\bigg{]}^{-1}.
\end{align}
In both formulas, a clique separator $C\cap D$ is counted once for each clique-tree edge $CD$ on which it occurs.

\item
The matrix $B_{\ct}$ is the unique nonsingular Lorentz-Gram completion $C$ of $A$ whose inverse is supported on the diagonal and the edges of $G$. Consequently, $B_{\ct}$ is independent of the choice of $\ct$.

\item
For every simplicial elimination ordering of $G$ (Theorem~\ref{T:chordal-seo}), Algorithm~\ref{alg:chordal-orthogonal-innovation} is well defined and produces vectors $(u_v)_{v\in V}\subseteq\Lob(\R^{n-1})$ with Lorentz-Gram matrix $B$. In particular, the resulting Lorentz-Gram matrix is independent of the chosen simplicial elimination ordering employed in Algorithm~\ref{alg:chordal-orthogonal-innovation}.
\end{enumerate}
\end{utheorem}

The inverse-sparsity characterization in Theorem~\ref{T:chordal-canonicity} has a natural variational interpretation. For nonsingular symmetric completions, stationarity of $X\mapsto\log|\det X|$ under variations of the unspecified entries is equivalent to vanishing of the corresponding entries of $X^{-1}$. Together with Theorem~\ref{T:chordal-canonicity}, this identifies the canonical clique-product completion as the unique stationary nonsingular Lorentz-Gram completion. Moreover, the negative-definite Schur complements arising in the clique attachments (later in the proofs) show that this completion uniquely maximizes the absolute determinant among all Lorentz-Gram completions.

\begin{utheorem}[Variational and maximum-absolute-determinant characterization]
\label{T:LG-variational-characterization}
Retain the hypotheses of Theorems~\ref{T:clique-product-completion} and \ref{T:chordal-canonicity}, and let $B$ be the canonical clique-product completion of $A$. Let $n:=|V|$ and define the following:
\[
\A(A)
:=
\left\{
X=X^t\in\R^{V\times V}:
x_{ij}=a_{ij}
\text{ whenever }i=j\text{ or }\{i,j\}\in E
\right\}.
\]
Suppose $\Phi(X):=\log|\det X|$. Then, for $X\in\A(A) \cap GL_n$, the following are equivalent:
\begin{enumerate}
\item
$X$ is a stationary point of $\Phi$ on $\A(A) \cap GL_n$; explicitly,
\[
\left.\frac{d}{dt}\Phi(X+tH)\right|_{t=0}=0
\]
for every symmetric $H=(h_{ij})_{i,j\in V}$ with $h_{ij}=0$ whenever $i=j$ or $\{i,j\}\in E$.

\item
The inverse $X^{-1}$ vanishes at every unspecified position:
\[
(X^{-1})_{ij}=0
\qquad
\text{whenever }i\neq j\text{ and }\{i,j\}\notin E.
\]
\end{enumerate}
Consequently, $B$ is the unique nonsingular Lorentz-Gram completion of $A$ that is stationary for $\Phi$ on $\A(A) \cap GL_n$. Moreover, $B$ uniquely maximizes the absolute determinant among all Lorentz-Gram completions of $A$: for every such completion $X$, $|\det X|\leq|\det B|$, with equality if and only if $X=B$.
\end{utheorem}

\begin{remark}[Inverse sparsity, maximum entropy, and moment matrices]\label{rem:covariance-moment-comparision}

We conclude this subsection by highlighting connections involving the clique-product completion and its characterization through inverse sparsity.

\begin{enumerate}
\item 
Theorem~\ref{T:LG-variational-characterization} connects the canonical Lorentz-Gram completion with the classical covariance-selection problem of Dempster \cite{Dempster1972}. If $\Sigma\succ0$ is the covariance matrix of a Gaussian vector in $\R^n$, its differential entropy is the following {\cite[pp 161]{Dempster1972}}: 
\[
h(\Sigma)
=
\frac12\log\det\Sigma+\frac n2\log(2\pi e).
\]
Thus, among positive-definite completions with prescribed diagonal and selected off-diagonal entries, maximizing Gaussian entropy is equivalent to maximizing the determinant. Whenever a positive-definite completion exists, the maximizer is unique and is characterized by the vanishing of its inverse at the unspecified positions {\cite[Theorem~2]{Grone-Johnson-Sa-Wolkowicz}}. The first-order identity underlying this characterization is more general:
\[
D\Phi(X)[H]=\operatorname{Tr}(X^{-1}H),
\qquad
\Phi(X)=\log|\det X|.
\]
For nonsingular symmetric completions, it makes stationarity under variations of the unspecified entries equivalent to inverse sparsity, without requiring positive definiteness or chordality. Ferrante and Pavon developed this variational viewpoint for general nonsingular matrix completions \cite{FerrantePavon2011}; see also \cite{FerrantePavon2013}

\item A particularly relevant parallel appears in the work of Helton, Lasserre, and Putinar \cite{Helton-Lasserre-Putinar}. For positive definite truncated moment matrices, they showed that prescribed zeros in the inverse are governed by full and conditional triangularity properties of the associated orthogonal polynomials \cite[Theorems~5.1 and~6.1]{Helton-Lasserre-Putinar}. This parallels the role played here by simplicial elimination and the orthogonal-innovation construction (in Algorithm~\ref{alg:chordal-orthogonal-innovation}) in producing the inverse-sparsity pattern of the canonical completion. They also observed that, when a positive definite completion is constrained to preserve moment-matrix structure, entropy maximization need not force the inverse to vanish at the unspecified positions \cite[p.~1454]{Helton-Lasserre-Putinar}. Thus, their work exhibits both the structural significance of inverse zeros and the possible failure, under additional constraints, of their classical connection with entropy maximization.
\end{enumerate}

A nonsingular Lorentz-Gram matrix of order $n\geq2$ is indefinite, so it is not itself a covariance matrix and the Gaussian-entropy interpretation does not apply directly. Nevertheless, Theorem~\ref{T:LG-variational-characterization} shows that, under the present nonsingular chordal hypotheses, the first-order inverse-sparsity condition still selects the canonical clique-product completion. Moreover, unlike the constrained moment-matrix phenomenon highlighted by Helton, Lasserre, and Putinar \cite[p.~1454]{Helton-Lasserre-Putinar}, the Lorentz-Gram constraint does not destroy the analogous connection with determinant extremality: the canonical completion uniquely maximizes the absolute determinant among all Lorentz-Gram completions. This conclusion follows from negative-definite Schur complements and Fischer's determinant inequality; it is not a claim of global maximization over the larger set $\A(A)\cap GL_n$. The absolute value is essential, since every nonsingular Lorentz-Gram matrix $X$ of order $n$ satisfies $\det X=(-1)^{n-1} |\det X|$.
\end{remark}

The clique-tree construction of Theorem~\ref{T:clique-product-completion} is developed in Subsection~\ref{SS:clique-tree-completion}, while the inverse formula, canonical uniqueness, and orthogonal-innovation realization in Theorem~\ref{T:chordal-canonicity} together with the proof of Theorem~\ref{T:LG-variational-characterization} are established in Subsection~\ref{SS:chordal-canonicity}.

\subsection{Metric distortion of completions}
\label{SS:metric-distortion}

The problem of embedding graph metrics into hyperbolic space with controlled distortion has been studied from several perspectives. Sarkar constructed embeddings of weighted trees into the hyperbolic plane that, after a suitable uniform scaling of the edge lengths, preserve those lengths and have multiplicative distortion arbitrarily close to one~\cite{Sarkar}. More generally, Verbeek and Suri related the distortion of hyperbolic embeddings of graph shortest-path metrics to combinatorial properties of the graph, particularly quasi-cyclicity \cite{Verbeek}.

Our starting point is different: we prescribe only the edge distances and seek a Lorentz-Gram completion of the resulting partial matrix. Such a completion need not yield an embedding; for instance, in the geodesic rectification of Theorem~\ref{T:tree-completion}, distinct vertices may acquire the same signed-path coordinate. By contrast, for finite trees with strictly positive prescribed edge lengths, the product-distance completion of Theorem~\ref{T:tree-path-product-completion} is nonsingular, as is the clique-product completion under the hypotheses of Theorem~\ref{T:clique-product-completion}. Their realizations therefore yield embeddings. We study the distortion of these embeddings relative to the shortest-path metric induced by the given edge weights.

\begin{defn}[Embedding realization, shortest-path metric, distortion]
\label{D:distortion-Lorentz-Gram-realization}
Let $V$ be a nonempty finite set.
\begin{enumerate}
\item 
Let $B=(b_{ij})_{i,j\in V}$ be a nonsingular Lorentz-Gram matrix. 
An embedding realization of $B$ is a map $f:V\to\Lob(H)$, for some real Hilbert space $H$, such that for all $i,j\in V$,
\[
[f(i),f(j)]=b_{ij}.
\]
By Lemma~\ref{lem:sign-inn}, every such map is injective. Moreover, by Lemma~\ref{L:LG-realization-isometry}, any two embedding realizations of $B$ are related by a Lorentz isometry on their spans. In particular, they determine the same pairwise hyperbolic distances.

\item
Let $G=(V,E)$ be a connected graph, and assign to each edge $e\in E$ a nonnegative number $\lambda_e$. The corresponding graph shortest-path metric for all $i,j\in V$ is given by,
\[
d_G(i,j):=\min_{P:i\leadsto j}\sum_{e\in P}\lambda_e,
\]
where the minimum is taken over all paths $P$ in $G$ joining $i$ to $j$, and the empty path has length zero.

\item
Following Verbeek and Suri \cite{Verbeek}, a hyperbolic embedding of $(V,d_G)$ is an injective map $f: V\to\Lob(H)$, where $H$ is a real Hilbert space. Its multiplicative distortion $\dist(f)$ is the least $\mathcal D\geq1$ such that for all $i,j\in V$, we have
\[
d_{\Lob}\bigl(f(i),f(j)\bigr)\leq d_G(i,j) \leq
\mathcal D\,d_{\Lob}\bigl(f(i),f(j)\bigr).
\]
Every embedding realization of a Lorentz-Gram completion preserves the prescribed edge lengths and is therefore nonexpansive with respect to $d_G$, i.e., by the hyperbolic triangle inequality, $d_{\Lob}(f(i),f(j))\leq d_G(i,j)$. Therefore, $\dist(f_A) =\displaystyle\max_{i\neq j} \tfrac{d_G(i,j)}{d_{\Lob}(f_A(i),f_A(j))}$.
\end{enumerate}
\end{defn}

We obtain two complementary conclusions about the embedding behaviour of our completions. For trees, uniformly scaling the prescribed edge lengths yields a family of product-distance embeddings whose multiplicative distortion converges to one. For chordal graphs, the clique-product embeddings exhibit a sharp separator-size dichotomy: singleton clique-tree separators give a sharp and uniform distortion bound, whereas the presence of a separator of size at least two allows the distortion to become arbitrarily large.

\subsubsection{Asymptotically isometric tree embeddings}

We first consider the product-distance completion for trees. Uniformly scaling the prescribed edge lengths makes the resulting hyperbolic embedding asymptotically isometric.

\begin{utheorem}[Limiting tree embeddings]
\label{T:asymptotically-isometric-tree-embeddings}

Let $T=(V,E)$ be a finite tree with $|V|\geq 2$ equipped with positive edge weights $(\lambda_e)_{e\in E}$, and let $d_T$ be its shortest-path metric. Put $\lambda_{\min}:=\min_{e\in E}\lambda_e$. For $\tau>0$, define the scaled partial matrix $A^{(\tau)}$ by $a_{ii}^{(\tau)}:=1$ and $a_e^{(\tau)}:=\cosh(\tau\lambda_e)$ for $e\in E$. Let $B^{(\tau)}$ be the product-distance completion of $A^{(\tau)}$ given by Theorem~\ref{T:tree-path-product-completion}, and let $f_\tau$ be any embedding realization of this fixed completion. Then $f_\tau$ is a metric embedding, and if $\lambda_{\min}\tau>{\log2}$, for every $i,j\in V$, we have
\[
d_{\Lob}\bigl(f_\tau(i),f_\tau(j)\bigr)
\leq
\tau d_T(i,j)
\leq
\left(
1-\frac{\log2}{\tau\lambda_{\min}}
\right)^{-1}
d_{\Lob}\bigl(f_\tau(i),f_\tau(j)\bigr).
\]
Consequently, the distortion $\dist (f_\tau)$ converges to $1$ as $\tau\to\infty$.
\end{utheorem}

Thus, the product-distance completion is determined by orthogonal innovation (shown later in Algorithm~\ref{alg:tree-orthogonal-innovation}), and uniform rescaling makes the resulting embeddings asymptotically isometric. The estimate depends only on $\tau\lambda_{\min}$ and is uniform over all finite trees. The construction realizes an $n$-vertex tree in $\Lob(\R^{n-1})$, whereas Sarkar's construction takes values in the hyperbolic plane \cite{Sarkar}.

\subsubsection{Separator-size dichotomy}

We next characterize when the canonical clique-product embeddings of a fixed connected chordal graph have uniformly bounded distortion over all admissible clique data. The answer depends on the sizes of the clique-tree separators.

\begin{utheorem}[Block-graph distortion dichotomy]
\label{T:separator-size-distortion-dichotomy}
Let $G=(V,E)$ be a finite connected chordal graph with $|V|\geq2$, and let $\ct$ be a clique-tree of $G$. Call a partial symmetric matrix $A$ with specification graph $G$ \emph{admissible} if every maximal clique principal submatrix is a nonsingular Lorentz-Gram matrix. For each admissible $A$, let $B_A$ be its clique-product completion from Theorem~\ref{T:clique-product-completion}, and let $f_A$ be any embedding realization of $B_A$. Equip $V$ with the shortest-path metric $d_G$ induced by the edge lengths $\lambda_{ij}:=\ach(a_{ij})$, $\{i,j\}\in E$. Suppose $\diam(G)$ denotes the diameter of $G$ with each edge assigned unit length. Then the following dichotomy holds.
\begin{enumerate}
\item
If every separator of $\ct$ is a singleton, then $\displaystyle\sup_A\dist(f_A)=\diam(G)^{1/2}$.

\item
If not every separator of $\ct$ is a singleton, then $\displaystyle\sup_A\dist(f_A)=\infty$.
\end{enumerate}
Both suprema range over all admissible partial matrices with specification graph $G$. The dichotomy is independent of the choice of $\ct$: the multiset of clique-tree separators, counted with the multiplicity with which each occurs on the edges of $\mathcal T$, is the same for every clique-tree of $G$; see {\cite[Theorem~4.4]{BlairPeyton1993}} or \cite{Ho-Lee,Lundquist}.
\end{utheorem}

\begin{remark}
A connected chordal graph whose clique-tree separators are all singletons is precisely a \emph{block graph}, that is, a graph each of whose blocks -- the maximal subgraphs without a cut vertex -- is a clique. Equivalently, these are precisely the graphs containing no induced cycle of length at least four and no induced diamond $K_4-e$. Thus Theorem~\ref{T:separator-size-distortion-dichotomy} says that the canonical embeddings have uniformly bounded distortion over all admissible data exactly when $G$ is a block graph; in that case, the sharp uniform bound is $\diam(G)^{1/2}$.
\end{remark}

\section{Preliminaries}

We begin by recording some elementary facts about finite-dimensional subspaces of Lorentz spaces and their Gram matrices. We use the following conventions throughout.

\begin{defn} We define signature, inertia, and orthogonal complement.

\begin{enumerate}
\item 
For a finite-dimensional real vector space $W$ equipped with a nondegenerate symmetric bilinear form, its signature is the pair $(p,q)$, where $p$ and $q$ are respectively the maximal dimensions of subspaces on which the form is positive definite and negative definite. Thus $p+q=\dim W$.

\item 
For a real symmetric matrix $M$, its inertia is the triple $\inn(M):=(n_+(M),n_-(M),n_0(M))$, where $n_+(M)$, $n_-(M)$, and $n_0(M)$ denote respectively the numbers of positive, negative, and zero eigenvalues of $M$, counted with multiplicity. By Sylvester's law of inertia, these numbers are invariant under congruence.

\item Recall Definition~\ref{defn:lob-space} for the hyperboloid model of Lobachevsky space and the Lorentz form $[\cdot,\cdot]$. For a subspace $W\subseteq \R\oplus H$, we write $W^\perp$ to denote all $x\in\R\oplus H$ such that $[x,w]=0$ for every $w\in W$. In particular, for $u\in\Lob(H)$, we write $u^\perp$ for $(\R u)^\perp$.
\end{enumerate}
\end{defn}

Let $H$ be a real Hilbert space and suppose $u\in\Lob(H)$. Since $[u,u]=1$, for every $x\in\R\oplus H$, we have $[x-[x,u]u,u]=[x,u]-[x,u][u,u]=0$. Thus $x-[x,u]u\in u^\perp$. The restriction of the Lorentz form to $u^\perp$ is negative definite, as shown next.

\begin{lemma}\label{L:neg-def}
For all $u\in\Lob(H)$, the restriction of the Lorentz form to $u^\perp$ is negative definite.
\end{lemma}
\begin{proof}
Let $0\neq y=(y_0,{\bf y}_1)\in u^\perp$, and write $u=(u_0,{\bf u}_1)$. Therefore, we have $0=[y,u]=y_0u_0-\langle {\bf y}_1,{\bf u}_1\rangle$, and so $y_0={\langle {\bf y}_1,{\bf u}_1\rangle}/{u_0}$. Using this and the Cauchy--Schwarz inequality, we obtain

\begin{align*}
[y,y]
=y_0^2-\|{\bf y}_1\|^2&=\frac{\langle {\bf y}_1,{\bf u}_1\rangle^2}{u_0^2}-\|{\bf y}_1\|^2 \\
\leq
\frac{\|{\bf y}_1\|^2\|{\bf u}_1\|^2}{u_0^2}
-\|{\bf y}_1\|^2 &= \bigg{(}\frac{\|{\bf u}_1\|^2}{u_0^2}
-1\bigg{)}\|{\bf y}_1\|^2 = -\frac{\|{\bf y}_1\|^2}{u_0^2}.
\end{align*}
If ${\bf y}_1=0$ then $y_0={\langle {\bf y}_1,{\bf u}_1\rangle}/{u_0}=0$. Thus $y\neq0$ implies ${\bf y}_1\neq0$, and hence $[y,y]<0$.
\end{proof}

\begin{remark}
In the two-dimensional case where $H=\R$, the Cauchy--Schwarz inequality is an equality, and the computation reduces to
\[
[x-[x,u]u,\;x-[x,u]u]
=
-(u_0{x}_1-{u}_1x_0)^2.
\]   
\end{remark}

The preceding lemma immediately determines the signature of every finite-dimensional subspace spanned by points of Lobachevsky space, and hence the inertia of the corresponding Lorentz-Gram matrix.

\begin{lemma}\label{lem:sign-inn}
Let $H$ be a real Hilbert space, and suppose $W$ is the real span of distinct $u_1,\dots,u_k\in\Lob(H)$. If $r:=\dim W$ then $W$ is nondegenerate of signature $(1,r-1)$, and the Lorentz-Gram matrix $\bigl([u_i,u_j]\bigr)_{i,j=1}^{k}$ has inertia $(1,r-1,k-r)$. In particular, the Lorentz-Gram matrix is nonsingular if and only if $u_1,\dots,u_k$ are linearly independent.
\end{lemma}

\begin{proof}
Since $u_1\in W$ and $[u_1,u_1]=1$, every $w\in W$ can be written uniquely as $w=[w,u_1]u_1+\bigl(w-[w,u_1]u_1\bigr)$, with $w-[w,u_1]u_1\in W\cap u_1^\perp$. Hence $W=\R u_1\oplus (W\cap u_1^\perp)$ as an orthogonal direct sum. The Lorentz form is positive definite on the one-dimensional space $\R u_1$, while by Lemma~\ref{L:neg-def} it is negative definite on $W\cap u_1^\perp$. Since $\dim(W\cap u_1^\perp)=r-1$, the restricted form on $W$ is nondegenerate and has signature $(1,r-1)$.

The family $u_1,\ldots,u_k$ spans $W$. Choosing a basis of $W$ from among these vectors and applying Sylvester's law of inertia shows that its Lorentz-Gram matrix has one positive eigenvalue, $r-1$ negative eigenvalues, and $k-r$ zero eigenvalues. Thus $\inn(G)=(1,r-1,k-r)$. In particular, $G$ is nonsingular precisely when $k=r$, which is equivalent to the linear independence of $u_1,\ldots,u_k$.
\end{proof}

The preceding lemma characterizes nonsingularity in terms of linear independence. We next record that two realizations of the same nonsingular Lorentz-Gram matrix are isometric on their spans. This will allow us to compare prescribed geometric configurations with the vectors produced by orthogonal innovation.

\begin{lemma}[Isometry of nonsingular Lorentz-Gram realizations]
\label{L:LG-realization-isometry}
Let $H,\widetilde H$ be a real Hilbert spaces, and let $I$ be a nonempty finite set. Suppose $(u_i)_{i\in I}\subseteq\Lob(H)$ and $(\widetilde u_i)_{i\in I}\subseteq\Lob(\widetilde H)$ have the same nonsingular Lorentz-Gram matrix:
\[
[u_i,u_j]=[\widetilde u_i,\widetilde u_j],
\qquad i,j\in I.
\]
Then there is a unique linear isometry $T:\Span\{u_i:i\in I\}\to\Span\{\widetilde u_i:i\in I\}$ such that $Tu_i=\widetilde u_i$ for every $i\in I$. Moreover, $T$ maps the positive unit hyperboloid in the first span onto that in the second and preserves hyperbolic distances.
\end{lemma}

\begin{proof}
Nonsingularity of the common Lorentz-Gram matrix implies that both families are linearly independent. Thus $T\left(\sum_{i\in I}a_i u_i\right):=\sum_{i\in I}a_i\widetilde u_i$ defines a unique linear bijection between their spans. Equality of the Lorentz-Gram matrices gives
\[
\Big[
T\Big(\sum_i a_i u_i\Big),
T\Big(\sum_j b_j u_j\Big)
\Big]
=
\sum_{i,j}a_i b_j[u_i,u_j]
=
\Big[\sum_i a_i u_i,\sum_j b_j u_j\Big].
\]
Hence $T$ is a Lorentz isometry. Fix $i_0\in I$. If $z$ belongs to the positive unit hyperboloid in the first span, then $[Tz,Tz]=[z,z]=1$ and $[Tz,\widetilde u_{i_0}]=[z,u_{i_0}]\geq1$. Thus $Tz$ lies on the positive sheet. Applying the same argument to $T^{-1}$ proves the asserted correspondence. Preservation of hyperbolic distances follows from $d_{\Lob }(z,z')=\ach[z,z']$.
\end{proof}

Lemma~\ref{lem:sign-inn} gives the finite-dimensional inertia condition satisfied by Lorentz-Gram matrices. A classical theorem of Krein shows that, at the level of arbitrary metric spaces, this condition is also sufficient for realization in Lobachevsky space. Since this result is the historical precursor of the kernel criterion developed in the next section, we record it here.

\begin{defn}[Lobachevsky space with parameter $R$]\label{defn:lob-space-of-radius}
Following Iohvidov and Krein~\cite{IK-1,IK-2}, for $R>0$, let $\Lob_R(H)$ denote the set $\Lob(H)$ (Definition~\ref{defn:lob-space}) equipped with the rescaled hyperbolic distance $d_{\Lob_R}(x,y):= R~\ach \left([x,y]\right)$, for all $x,y\in\Lob_R(H)$, or equivalently, $[x,y]=\cosh \left(\frac{1}{R}d_{\Lob_R}(x,y)\right)$. Thus $\Lob_1(H)$ coincides with $\Lob(H)$, with the normalization used throughout this paper.
\end{defn}

\begin{theorem}[Krein's Lobachevsky embedding theorem \cite{Krein1948}; see also {\cite[Theorem~6.1]{IK-2}}]\label{Krein-embedding} Let $R>0$, and let $(Q,\rho)$ be a set $Q$ equipped with a map $\rho:Q\times Q\to\R_+$ such that $\rho(p,q)=\rho(q,p)$ and $\rho(p,p)=0$ for all $p,q\in Q$. Then there exist a real Hilbert space $H$ and a map $p\longmapsto x_p$ from $Q$ to $\Lob_R(H)$ such that, 
\[
d_{\Lob_R}(x_p,x_q)=\rho(p,q) \quad \mbox{for all}\quad p,q\in Q,
\]
if and only if the real symmetric matrix $\left[\cosh\!\left(\frac{1}{R}\rho(q_i,q_j)\right)\right]_{i,j=1}^n$ has exactly one positive eigenvalue for every $q_1,\ldots,q_n\in Q$ and all integers $n\geq 1$.
\end{theorem}

The necessity follows from the same inertia principle as Lemma~\ref{lem:sign-inn}. The converse may be proved by an indefinite analogue of the Gelfand--Naimark--Segal (GNS) construction. One introduces formal symbols $x_p$, $p\in Q$, and defines on their finite linear span the symmetric bilinear form
\[
\Big[
\sum_j \xi_j x_{p_j},
\sum_k \eta_k x_{q_k}
\Big]
:=\sum_{j,k}
\xi_j\eta_k
\cosh\!\Big(\frac{1}{R}\rho(p_j,q_k)\Big).
\]
The hypothesis implies that this form has at most one positive square on every finite-dimensional subspace; since $[x_p,x_p]=1$, its positive index is exactly one. After quotienting by the radical, the standard Pontryagin-space construction yields a space of positive index one, which may be represented as $\R \oplus H$ with its Lorentz form. The images of the symbols, still denoted by $x_p$, satisfy $[x_p,x_q]=\cosh\!\left(\frac{1}{R}\rho(p,q)\right)$ and $[x_p,x_p]=1$. Moreover, since $[x_p,x_q]=\cosh\!\left(\frac{1}{R}\rho(p,q)\right)\geq 1$ for all $p,q\in Q$, the vectors $x_p$ lie on the same sheet of the unit hyperboloid. After changing all signs simultaneously if necessary, they lie in $\Lob_R(H)$. By Definition~\ref{defn:lob-space-of-radius},
\[
d_{\Lob_R}(x_p,x_q)
=
R \ach \left([x_p,x_q]\right)
=
\rho(p,q),
\]
which gives the desired realization. Krein announced this theorem in~\cite{Krein1948}; a complete proof appears in the work of Iohvidov and Krein~\cite[Theorem~6.1]{IK-2}. The result was rediscovered in the recent hyperbolic distance matrix literature by Tabaghi and Dokmani\'c~\cite{TabaghiDokmanic2020}.

For the purposes of the present paper, we will use the normalized case $R=1$, for which $\Lob_1(H)=\Lob(H)$. We also seek a formulation directly in terms of normalized Lorentz-Gram kernels, in which the unique positive direction is isolated by fixing an anchor. This reduces the indefinite realization problem to an ordinary positive semidefinite one and leads to the kernel criteria developed in the next section.

\section{Lorentzian decompositions and kernel criteria}

We now prove Theorem~\ref{T:LG-ker-char}. The proof separates naturally into the anchored positive semidefinite criterion and its intrinsic reformulation in terms of one positive square.

\subsection{Anchor reduction}

Fixing one vector in a Lorentz-Gram realization separates the unique positive direction from its negative definite orthogonal complement. The resulting orthogonal decomposition gives the anchored positive semidefinite criterion.

\begin{theorem}[Lorentz-Gram anchor reduction]\label{T:LG-anchor-char}
Let $I$ be nonempty, and let $K=(k_{ij})_{i,j\in I}$ be a real symmetric kernel satisfying $k_{ij}\geq k_{ii}=1$ for all $i,j\in I$. Fix $i_0\in I$, and define the kernel $P_{i_0}=\bigl(p^{(i_0)}_{ij}\bigr)_{i,j\in I}$ by $p^{(i_0)}_{ij}:=k_{ii_0}k_{ji_0}-k_{ij}$. Then the following are equivalent:
\begin{enumerate}
\item There exist a real Hilbert space $H$ and a family $(u_i)_{i\in I}\subseteq \Lob(H)$ such that $k_{ij}=\ip{u_i}{u_j}$, for all $i,j\in I$.

\item The anchored kernel $P_{i_0}$ is positive semidefinite.
\end{enumerate}
Whenever these conditions hold, $p^{(i_0)}_{i_0,j}=0$ for every $j\in I$. Moreover, the minimal dimension of $H$ in~{\rm(1)} is the supremum of $\rank \bigl(p^{(i_0)}_{ij}\bigr)_{i,j\in F}$ where $F\subseteq I$ runs over all finite sets.
\end{theorem}
\begin{proof}
$(1)\implies(2)$. Suppose $(1)$ holds and for each $i\in I$, define
\[
v_i
:=
u_i-\ip{u_i}{u_{i_0}}u_{i_0}
=
u_i-k_{ii_0}u_{i_0}.
\]
Since $\ip{u_{i_0}}{u_{i_0}}=1$, we have $\ip{v_i}{u_{i_0}}=\ip{u_i}{u_{i_0}}-k_{ii_0}\ip{u_{i_0}}{u_{i_0}}=0$. Thus $v_i\in u_{i_0}^{\perp}$ for every $i\in I$. By Lemma~\ref{L:neg-def}, the Lorentz form is negative definite on $u_{i_0}^{\perp}$, so the kernel $\bigl(-\ip{v_i}{v_j}\bigr)_{i,j\in I}$ is positive semidefinite. A direct calculation gives
\begin{align*}
\ip{v_i}{v_j}
&=
\ip{u_i-k_{ii_0}u_{i_0}}
   {u_j-k_{ji_0}u_{i_0}}\\
&=
\ip{u_i}{u_j}
-k_{ji_0}\ip{u_i}{u_{i_0}}
-k_{ii_0}\ip{u_{i_0}}{u_j}
+k_{ii_0}k_{ji_0}\ip{u_{i_0}}{u_{i_0}}\\
&=
k_{ij}-k_{ii_0}k_{ji_0}.
\end{align*}
Hence $p^{(i_0)}_{ij}=k_{ii_0}k_{ji_0}-k_{ij}=-\ip{v_i}{v_j}$. Thus, for every finitely supported real family $(c_i)_{i\in I}$,
\begin{align*}
\sum_{i,j\in I}c_ic_j~p^{(i_0)}_{ij}
=
-\sum_{i,j\in I}c_ic_j[v_i,v_j]
=
-\left[\sum_{i\in I}c_iv_i,\sum_{i\in I}c_iv_i\right]\geq0,
\end{align*}
because $\sum_i c_iv_i\in u_0^\perp$. Thus $P_{i_0}$ is positive semidefinite.

$(2)\implies(1)$.
Suppose that $P_{i_0}$ is positive semidefinite. By the Gram representation of positive semidefinite kernels, there exist a real Hilbert space $H$ and vectors ${\bf x}_i\in H$, $i\in I$, such that
\[
\langle {\bf x}_i,{\bf x}_j\rangle_H
=
p^{(i_0)}_{ij}
=
k_{ii_0}k_{ji_0}-k_{ij}.
\]
Since $\|{\bf x}_{i_0}\|^2=p^{(i_0)}_{i_0i_0}=k_{i_0i_0}^2-k_{i_0i_0}=0$, we have ${\bf x}_{i_0}=0$. Define $u_i:=(k_{ii_0},{\bf x}_i)\in\R\oplus H$ for all $i\in I$. Then
\begin{align*}
\ip{u_i}{u_j}
&=
k_{ii_0}k_{ji_0}-\langle {\bf x}_i,{\bf x}_j\rangle_H\\
&=
k_{ii_0}k_{ji_0}
-\bigl(k_{ii_0}k_{ji_0}-k_{ij}\bigr)\\
&=
k_{ij}.
\end{align*}
In particular, $\ip{u_i}{u_i}=k_{ii}=1$. Moreover, the zeroth coordinate of $u_i$ is $k_{ii_0}\geq1$. Hence $u_i\in\Lob(H)$ for every $i\in I$, and so $(1)$ holds.

Next, for every $j\in I$, since $K$ is symmetric, we have $p^{(i_0)}_{i_0,j}=k_{i_0i_0}k_{ji_0}-k_{i_0j}=k_{ji_0}-k_{i_0j}=0$.

Finally, let $r$ be the supremum of $\rank\bigl(p^{(i_0)}_{ij}\bigr)_{i,j\in F}$ as $F\subseteq I$ runs over finite sets. For any realization as in~$(1)$, the vectors $v_i\in u_{i_0}^\perp$ defined above satisfy $p^{(i_0)}_{ij}=-\ip{v_i}{v_j}$. Since $-\ip{\cdot}{\cdot}$ is a positive definite inner product on $u_{i_0}^{\perp}$, for every finite $F\subseteq I$, $\rank\bigl(p^{(i_0)}_{ij}\bigr)_{i,j\in F}\leq \dim H$. Hence $r\leq\dim H$. Conversely, since $P_{i_0}$ is positive semidefinite, it admits a Gram representation $p^{(i_0)}_{ij}=\langle {\bf x}_i,{\bf x}_j\rangle_H$ in a real Hilbert space $H$ of dimension $r$. Applying the construction in the implication $(2)\implies(1)$ gives a Lorentz-Gram realization of $K$ in $\Lob(H)$. Therefore the minimal possible dimension of $H$ is as desired.
\end{proof}

The anchored criterion depends formally on a distinguished index; we next show that it is equivalent to the intrinsic condition of having one positive square.

\subsection{Positive-square criterion}

We now pass from the anchored criterion to the intrinsic notion of Definition~\ref{defn:pos-ker}. The next result identifies the two.

\begin{theorem}[One positive square criterion]
\label{T:LG-int-char}
Let $I$ be a nonempty set, and let $K=(k_{ij})_{i,j\in I}$ be a real symmetric kernel satisfying $k_{ij}\geq k_{ii}=1$ for all $i,j\in I$. Fix $i_0\in I$, and define the anchored kernel $P_{i_0}=\bigl(p^{(i_0)}_{ij}\bigr)_{i,j\in I}$ by $p^{(i_0)}_{ij}:=k_{ii_0}k_{ji_0}-k_{ij}$. Then the following are equivalent:
\begin{enumerate}
\item
The kernel $K$ has one positive square.

\item
The anchored kernel $P_{i_0}$ is positive semidefinite.
\end{enumerate}

Moreover, if $P_{i_0}$ is positive semidefinite for one choice of $i_0\in I$, then the anchored kernel corresponding to every choice of $i_0\in I$ is positive semidefinite.
\end{theorem}

\begin{proof}
$(2)\implies(1)$. Suppose that $P_{i_0}$ is positive semidefinite. By Theorem~\ref{T:LG-anchor-char}, there exist a real Hilbert space $H$ and vectors $u_i\in\Lob(H)$, $i\in I$, such that $k_{ij}=\ip{u_i}{u_j}$ for all $i,j\in I$. Since the Lorentz form on $\R\oplus H$ has positive index one, every finite Lorentz-Gram matrix $(k_{ij})_{i,j\in F}$ for $F\subseteq I$ finite, has at most one positive eigenvalue. Since $k_{ii}=1$ for every $i\in I$, each nonempty finite restriction has a positive eigenvalue. Hence $K$ has one positive square.

$(1)\implies(2)$. Suppose that $K$ has one positive square. Define, for $i,j\in I$,
\[
\rho(i,j):=\ach(k_{ij}).
\]
Since $K$ is symmetric, $k_{ij}\geq1$, and $k_{ii}=1$, the map $\rho$ is symmetric, nonnegative, and satisfies $\rho(i,i)=0$. Moreover, $\cosh\rho(i,j)=k_{ij}$. Since $K$ has one positive square and $k_{ii}=1$, every nonempty finite restriction of $K$ has exactly one positive eigenvalue. Hence, by Theorem~\ref{Krein-embedding} with $R=1$, there exist a real Hilbert space $H$ and vectors $u_i\in\Lob(H)$ such that
\[
d_{\Lob}(u_i,u_j)=\rho(i,j).
\]
Consequently, $[u_i,u_j]=\cosh d_{\Lob}(u_i,u_j)=\cosh\rho(i,j)=k_{ij}$. Thus $K$ is a Lorentz-Gram kernel. Applying Theorem~\ref{T:LG-anchor-char} yields that $P_{i_0}$ is positive semidefinite.

Finally, if $P_{i_0}$ is positive semidefinite for one choice of $i_0$, then the implication $(2)\implies(1)$ shows that $K$ has one positive square. Applying $(1)\implies(2)$ with any $i_1\in I$ then shows that the corresponding anchored kernel $P_{i_1}$ is positive semidefinite. This proves the final assertion.
\end{proof}

Combining Theorems~\ref{T:LG-anchor-char} and~\ref{T:LG-int-char} proves Theorem~\ref{T:LG-ker-char}.

\section{Local gluing and chordality}

We now prove Theorem~\ref{T:chordal-characterization}. The argument has two ingredients. First, we show that two Lorentz-Gram realizations that agree on a common clique can be placed in a common Lorentz space so as to agree on that clique. This gives a one-edge completion lemma. We then combine this local gluing step with the standard one-edge-at-a-time construction for chordal graphs.

\subsection{Chordal completion by local gluing}

The local step is to glue two Lorentz-Gram realizations along their common fully specified part, as formalized next. Here, the only new entry is then the Lorentz inner product between the two vectors lying outside the common clique.

\begin{lemma}[Local gluing]\label{L:one-edge-lob-completion}
Let $u$ and $v$ be distinct indices not in $S$, and let $C$ be a partial symmetric matrix indexed by $S\cup\{u,v\}$ whose only unspecified entries are $c_{uv}=c_{vu}$. If the principal submatrices $C_1:=C_{(S\cup\{u\})\times(S\cup\{u\})}$ and $C_2:=C_{(S\cup\{v\})\times(S\cup\{v\})}$ are Lorentz-Gram matrices, then $c_{uv}=c_{vu}$ can be chosen so that $C$ is a Lorentz-Gram matrix.
\end{lemma}

\begin{proof}
If $S=\varnothing$ then choose arbitrary vectors $x_u$ and $y_v$ in a common Lobachevsky space and define $c_{uv}=c_{vu}:=[x_u,y_v]$. Suppose now that $S\neq\varnothing$. Choose Lorentz-Gram realizations $(x_s)_{s\in S}\cup\{x_u\}$ and $(y_s)_{s\in S}\cup\{y_v\}$ of $C_1$ and $C_2$, respectively. Since the two realizations involve only finitely many vectors, we may assume that they lie in finite-dimensional Lorentz spaces $\R\oplus H_1$ and $\R\oplus H_2$, respectively. By embedding these spaces into $\R\oplus(H_1\oplus H_2)=:\R\oplus H$ in the natural way, we may regard both realizations as lying in a common finite-dimensional Lorentz space. Also, note that the two families $(x_s)_{s\in S}$ and $(y_s)_{s\in S}$ have the same Lorentz-Gram matrix
$C_0:=C_{S\times S}$; that is $[y_s,y_t]=c_{st}=[x_s,x_t]$.

Set $Y:=\Span\{y_s:s\in S\}$ and $X:=\Span\{x_s:s\in S\}$. Observe that $X$ and $Y$ are nondegenerate. Indeed, fix $s_0\in S$. Since $x_{s_0},y_{s_0}\in\Lob(H)$, we have $[x_{s_0},x_{s_0}]=[y_{s_0},y_{s_0}]=1$. By Lemma~\ref{L:neg-def}, the Lorentz form is negative definite on $x_{s_0}^{\perp}$ and $y_{s_{0}}^{\perp}$, and so it follows that $X\cap X^\perp=\{0\}=Y\cap Y^\perp$. We next show that the correspondence $y_s\mapsto x_s$, $s\in S$, defines a linear map from $Y$ to $X$. Suppose that $\sum_{s\in S}\alpha_s y_s=0$ and set $z:=\sum_{s\in S}\alpha_s x_s$. For every $t\in S$, equality of the Gram matrices gives
\[
[z,x_t]
=
\sum_{s\in S}\alpha_s[x_s,x_t]
=
\sum_{s\in S}\alpha_s[y_s,y_t]
=
0.
\]
Thus $z\in X\cap X^\perp$, and hence $z=0$. Therefore the rule $T_0\left(\sum_{s\in S}\alpha_s y_s\right) := \sum_{s\in S}\alpha_s x_s$ is well defined. Moreover, if $y=\sum_{s\in S}\alpha_s y_s$ and $y'=\sum_{t\in S}\beta_t y_t$, then
\begin{align*}
[T_0y,T_0y']
=
\sum_{s,t\in S}\alpha_s\beta_t[x_s,x_t]
=
\sum_{s,t\in S}\alpha_s\beta_t[y_s,y_t]
=[y,y'].
\end{align*}
Hence $T_0:Y\to X$ is a Lorentz isometry satisfying $T_0y_s=x_s$ for every $s\in S$. By Witt's extension theorem \cite{ArtinGeometricAlgebra}, the isometry $T_0$ extends to a Lorentz isometry $T$ of the ambient finite-dimensional $\R\oplus H$. Furthermore, $T$ preserves the positive sheet. Indeed, $Ty_{s_0}=T_0y_{s_0}=x_{s_0}$, and both $y_{s_0}$ and $x_{s_0}$ belong to $\Lob(H)$. Since a Lorentz isometry either preserves the two sheets of the unit hyperboloid or interchanges them, $T$ must preserve the positive sheet. In particular, $Ty_v\in\Lob(H)$. The vectors $(x_s)_{s\in S}\cup\{x_u,Ty_v\}$ now realize all the prescribed entries of $C$. Indeed, for every $s\in S$, $[x_u,x_s]=c_{us}$, while
\[
[Ty_v,x_s]
=
[Ty_v,Ty_s]
=
[y_v,y_s]
=
c_{vs}.
\]
Define the unspecified entries by $c_{uv}=c_{vu}:=[x_u,Ty_v]$, completing the proof.
\end{proof}

The chordal structure allows this local step to be iterated. Indeed, the standard edge-addition characterization of chordal graphs permits the missing edges to be inserted one at a time while preserving chordality, and Lemma~\ref{L:one-edge-lob-completion} supplies the required matrix entry at each step. We can therefore prove the global completion theorem.

\begin{theorem}[Chordal existence]\label{T:chordal-existence}
Let $A$ be a partial symmetric matrix whose specification graph is chordal. Then $A$ admits a Lorentz-Gram completion if and only if every fully specified clique submatrix of $A$ is a Lorentz-Gram matrix.
\end{theorem}

\begin{proof}
Necessity follows by restricting any completion to each fully specified clique. For sufficiency, we follow the one-edge-at-a-time completion method of Bakonyi and Johnson \cite[Theorem~3.3]{Bakonyi-Johnson}. By \cite[Lemmas~3 and~4]{Grone-Johnson-Sa-Wolkowicz}, there is a sequence of chordal graphs
\[
G=G_0,G_1,\ldots,G_m=K_n
\]
such that, for each $r$, the graph $G_r$ is obtained from $G_{r-1}$ by adding a single edge $\{u_r,v_r\}$, which is contained in a unique maximal clique $V_r$ of $G_r$.

We construct partial matrices $A^{(0)},A^{(1)},\ldots,A^{(m)}$ inductively so that $A^{(r)}$ has specification graph $G_r$, agrees with $A$ on every originally prescribed entry, and has a Lorentz-Gram matrix on every clique of $G_r$. Set $A^{(0)}:=A$; the hypotheses give the required properties for $r=0$. Suppose that $A^{(r-1)}$ has been constructed, and set $S_r:=V_r\setminus\{u_r,v_r\}$. Since $V_r$ is a clique of $G_r$ and $\{u_r,v_r\}$ is the only edge of $G_r$ absent from $G_{r-1}$, the only unspecified entries of $A^{(r-1)}_{V_r\times V_r}$ are those indexed by $(u_r,v_r)$ and $(v_r,u_r)$. Moreover, $S_r\cup\{u_r\}$ and $S_r\cup\{v_r\}$ are cliques of $G_{r-1}$. By the induction hypothesis, the corresponding principal submatrices are Lorentz-Gram matrices. Lemma~\ref{L:one-edge-lob-completion} therefore allows us to choose $a^{(r)}_{u_rv_r}=a^{(r)}_{v_ru_r}$ so that $A^{(r)}_{V_r\times V_r}$ is a Lorentz-Gram matrix. Leave all other entries unchanged.

It remains to verify the induction invariant. Let $Q$ be a clique of $G_r$. If $Q$ is already a clique of $G_{r-1}$, then $A^{(r)}_{Q\times Q}=A^{(r-1)}_{Q\times Q}$ is a Lorentz-Gram matrix by the induction hypothesis. Otherwise, $Q$ contains the newly added edge $\{u_r,v_r\}$. Since $V_r$ is the unique maximal clique of $G_r$ containing this edge, $Q\subseteq V_r$. Hence $A^{(r)}_{Q\times Q}$ is a principal submatrix of the Lorentz-Gram matrix $A^{(r)}_{V_r\times V_r}$ and is therefore itself a Lorentz-Gram matrix. Thus the induction proceeds. Thus, $G_m=K_n$, and so $A^{(m)}$ is fully specified, and hence is a Lorentz-Gram completion of $A$.
\end{proof}

\subsection{Obstructions on nonchordal graphs}

It remains to prove the converse direction of Theorem~\ref{T:chordal-characterization}. The obstruction already appears on a chordless cycle: one edge can be prescribed to have length larger than the sum of the remaining edge lengths, violating the triangle inequality in any global realization. We first record the elementary inheritance principle that allows such an obstruction to be detected on an induced subgraph.

\begin{proposition}[Inheritance of obstruction]
\label{P:completion-obstruction-inheritance}
Let $G=(V,E)$ be a finite graph, let $W\subseteq V$, and let $H:=G[W]$ be the induced subgraph of $G$. Suppose that $A=(a_{ij})_{i,j\in V}$ is a partial symmetric matrix with specification graph $G$. If the restriction $A_{W\times W}$ admits no Lorentz-Gram completion, then $A$ admits no Lorentz-Gram completion.
\end{proposition}

\begin{proof}
Suppose, to the contrary, that $A$ admits a Lorentz-Gram completion $B=(b_{ij})_{i,j\in V}$. Then there exist a real Hilbert space $H_0$ and vectors $u_i\in\Lob(H_0)$, $i\in V$, such that $b_{ij}=\ip{u_i}{u_j}$, $i,j\in V$. Restricting to the indices in $W$, the principal submatrix $B_{W\times W}
=
\bigl(\ip{u_i}{u_j}\bigr)_{i,j\in W}$ is a Lorentz-Gram matrix. Since $G[W]$ is induced, the prescribed entries of $A_{W\times W}$ are precisely the restrictions of the prescribed entries of $A$ to $W\times W$. Hence $B_{W\times W}$ is a Lorentz-Gram completion of $A_{W\times W}$, contrary to hypothesis. Therefore $A$ admits no Lorentz-Gram completion.
\end{proof}

We now exhibit the basic obstruction.

\begin{proposition}[Chordless-cycle obstruction]
\label{P:cycle-obstruction}
Let $C_k=v_1v_2\cdots v_kv_1$ denote a cycle on $k\geq 4$ vertices $V:=\{v_1,\dots,v_k\}$. Define a partial symmetric matrix $A=(a_{ij})_{i,j\in V}$ with specification graph $C_k$ by prescribing
\[
a_{v_iv_i}=1,
\qquad
a_{v_iv_{i+1}}=\cosh 1
\quad \mbox{for}\quad i=1,\dots,k-1, \quad\mbox{and}\quad
\qquad
a_{v_kv_1}=\cosh k.
\]
Then every fully specified clique principal submatrix of $A$ is Lorentz-Gram, but the full partial matrix $A$ admits no Lorentz-Gram completion.
\end{proposition}

\begin{proof}
The maximal cliques of $C_k$ are its edges, and every prescribed $2\times2$ principal submatrix is Lorentz-Gram. Suppose that a Lorentz-Gram completion existed. Then there would be points $u_{v_1},\ldots,u_{v_k}$ in Lobachevsky space such that $d_{\Lob}(u_{v_i},u_{v_{i+1}})=1$ for all $1\leq i\leq k-1$, while $d_{\Lob}(u_{v_k},u_{v_1})=k$. By the triangle inequality,
\[
d_{\Lob}(u_{v_1},u_{v_k})
\leq
\sum_{i=1}^{k-1}
d_{\Lob}(u_{v_i},u_{v_{i+1}})
=
k-1,
\]
contradicting $d_{\Lob}(u_{v_1},u_{v_k})=k$. Hence no Lorentz-Gram completion exists.
\end{proof}

To pass from a chordless cycle to an arbitrary nonchordal graph, it
remains only to extend these obstructed cycle data to the remaining
edges while preserving the Lorentz-Gram condition on every clique.

\begin{theorem}[Nonchordal obstruction]
\label{T:nonchordal-obstruction}
Let $G=(V,E)$ be a finite nonchordal graph. Then there exists a
partial symmetric matrix $A=(a_{ij})_{i,j\in V}$ with specification
graph $G$ such that every fully specified clique principal submatrix of $A$ is Lorentz-Gram, but $A$ admits no Lorentz-Gram completion.
\end{theorem}
\begin{proof}
Since $G$ is nonchordal, it contains an induced chordless cycle $C_k=v_1v_2\cdots v_kv_1$, $k\geq4$. Choose $M\geq k/2$. We prescribe $a_{ii}=1$ for every $i\in V$
and, writing $a_{ij}=\cosh\lambda_{ij}$ for $\{i,j\}\in E$, we set
\[
\lambda_{v_iv_{i+1}}=1,
\qquad 1\leq i\leq k-1,
\qquad
\lambda_{v_kv_1}=k.
\]
For every remaining prescribed edge, we set
\[
\lambda_{ij}
:=
\begin{cases}
M, & \text{if exactly one of $i,j$ belongs to $C_k$},\\
0, & \text{if $i,j\notin C_k$}.
\end{cases}
\]
Verify the clique condition for $A=(a_{ij})_{i,j\in V}$. Let $Q$ be a clique of $G$; then $A_{Q\times Q}$ is fully specified (by the above definition of $A$). Since $C_k$ is induced and
chordless, $Q$ contains at most two vertices
of $C_k$, and if it contains two vertices, they must be consecutive in the cycle. Now, there are three possibilities, and we present with the realization in Lobachevsky space.

\begin{enumerate}
\item 
If $Q\cap C_k=\varnothing$, all prescribed distances within $Q$ are zero, so $A_{Q\times Q}$ is realized by a single point of Lobachevsky space.

\item 
If $Q\cap C_k=\{v_i\}$, then choose two points $u,w$ with $d_{\Lob}(u,w)=M$, and represent $v_i$ by $u$, and every vertex of $Q\setminus\{v_i\}$ by $w$.

\item 
Suppose that $Q\cap C_k=\{v_i,v_j\}$. Then $\{v_i,v_j\}$ is an edge of $C_k$, of length
$\lambda\in\{1,k\}$. Since $2M\geq k\geq\lambda$, the numbers $M,M,\lambda$
satisfy the triangle inequalities. More explicitly, by the hyperbolic law of cosines, it suffices to choose $\theta\in[0,\pi]$ such that
\[
\cosh\lambda
=
\cosh^2 M-\sinh^2 M\cos\theta.
\]
Such a $\theta$ exists because $\lambda\leq2M$, which implies
\[
\begin{aligned}
\frac{\cosh^2 M-\cosh\lambda}{\sinh^2 M}
=
\frac{\sinh^2 M+1-\cosh\lambda}{\sinh^2 M}
&=
1-\frac{\cosh\lambda-1}{\sinh^2 M}\\
=
1-\frac{2\sinh^2(\lambda/2)}{\sinh^2 M}
=
1-2\left(\frac{\sinh(\lambda/2)}{\sinh M}\right)^2 &\in [-1,1].
\end{aligned}
\]
Once such a $\theta\in[0,\pi]$ is chosen, we set the following in $\Lob(\R^2)$:
\begin{align*}
w:=(1,0,0),\qquad
u_i&:=(\cosh M,\sinh M,0),\\
\mbox{and}\qquad\qquad u_j&:=
(\cosh M,\sinh M\cos\theta,\sinh M\sin\theta).
\end{align*}
Then $[w,u_i]=[w,u_j]=\cosh M$, while
$[u_i,u_j]
=
\cosh^2 M-\sinh^2 M\cos\theta
=
\cosh\lambda$. Thus there exist points $u_i,u_j,w$ in Lobachevsky space such that
\[
d_{\Lob}(u_i,u_j)=\lambda,
\qquad
d_{\Lob}(u_i,w)=d_{\Lob}(u_j,w)=M.
\]
Now represent $v_i,v_j$ by $u_i,u_j$, respectively, and every vertex of $Q\setminus\{v_i,v_j\}$ by $w$. Thus obtain that $A_{Q\times Q}$ is Lorentz-Gram.
\end{enumerate}
Since $Q$ is arbitrary, every fully specified clique principal submatrix of $A$ is Lorentz-Gram.

On the other hand, the restriction $A_{C_k\times C_k}$ is precisely the chordless-cycle obstruction in Proposition~\ref{P:cycle-obstruction}: its first $k-1$ edge lengths are $1$, while the remaining edge has length $k$. Hence $A_{C_k\times C_k}$ admits no Lorentz-Gram completion. Since $C_k$ is an induced subgraph of $G$, Proposition~\ref{P:completion-obstruction-inheritance} implies that $A$ admits no Lorentz-Gram completion.
\end{proof}

The chordal existence Theorem~\ref{T:chordal-existence} and obstruction Theorem~\ref{T:nonchordal-obstruction} together prove Theorem~\ref{T:chordal-characterization}. Thus chordality is exactly the condition under which local Lorentz-Gram compatibility on the cliques guarantees a global Lorentz-Gram completion.

\section{Edge signings and geodesic rectification}\label{S:geodesicrectification}

We now examine the signed-path construction of Theorem~\ref{T:tree-completion} in more detail. Besides proving that the construction gives a Lorentz-Gram completion, we describe the geometric
role of the edge-signing, the resulting freedom in the unspecified
entries, and the corresponding geodesic realization in
$\Lob(\R)$. For finite trees, this also yields a direct
recursive algorithm for constructing the realizing points.

\begin{proof}[Proof of Theorem~\ref{T:tree-completion}] Consider the space $\Lob(\R)$. For each $i\in V$, define $u_i:=\bigl(\ch(\bt_i),\sh(\bt_i)\bigr)$. Since $[u_i,u_i]=\ch^2(\bt_i)-\sh^2(\bt_i)=1$ and $\ch(\bt_i)>0$, it follows that $u_i\in\Lob(\R)$. The hyperbolic difference identity yields:
\begin{align*}
[u_i,u_j]
&=
\ch(\bt_i)\ch(\bt_j)-\sh(\bt_i)\sh(\bt_j)\\
&=
\ch(\bt_i-\bt_j)
=
b_{ij}.
\end{align*}
Therefore $B$ is a Lorentz-Gram kernel. It remains to verify that $B$ agrees with $A$ on every prescribed entry. Note that for each $i\in V$, we have $b_{ii}=\ch(0)=1=a_{ii}$. Now let $\{i,j\}\in E$. Since the tree is rooted at $0$, one of $i$ and $j$ is the parent of the other. After interchanging $i$ and $j$, if necessary, we may assume that $i$ is the parent of $j$. Then we have the disjoint union: $P(0,j)=P(0,i)\sqcup\{\{i,j\}\}$. Consequently, we have $\bt_j-\bt_i=\varepsilon_{\{i,j\}}\lambda_{\{i,j\}}$. Therefore, since $\ch$ is even, we obtain the following:
\begin{align*}
b_{ij}
&=
\ch(\bt_i-\bt_j)=
\ch\bigl(\varepsilon_{\{i,j\}}\lambda_{\{i,j\}}\bigr)\\
&=
\ch(\lambda_{\{i,j\}})
=
a_{ij}.
\end{align*}
Therefore, $B$ is a completion of $A$. Since $\bt_0=0$, we have $u_0=(1,0)$, and hence $b_{i0}=[u_i,u_0]=\ch(\bt_i)$, which gives the stated formula for the anchor row and column. Finally,
\begin{align*}
b_{i0}b_{j0}-b_{ij}
&=
\ch(\bt_i)\ch(\bt_j)-\ch(\bt_i-\bt_j)\\
&=
\ch(\bt_i)\ch(\bt_j)
-
\bigl(
\ch(\bt_i)\ch(\bt_j)
-
\sh(\bt_i)\sh(\bt_j)
\bigr)\\
&=
\sh(\bt_i)\sh(\bt_j),
\end{align*}
for $i,j\in V$. Setting $s_i:=\sh(\bt_i)$, we obtain $\bp_0(B)=ss^T$. Hence $\bp_0(B)$ is positive semidefinite and has rank at most one. This completes the proof.
\end{proof}

\begin{remark}[Geometric meaning of the edge-signing]
For an edge $e=\{i,j\}$, the prescribed value
$a_{ij}=\cosh(\lambda_e)$ determines the hyperbolic distance
$\lambda_e$, but not the orientation of the corresponding displacement along $\Lob(\R)$. Indeed,
\[
\bt_j-\bt_i=\pm\lambda_e,
\]
and the sign $\varepsilon_e$ records this choice. For instance, consider the path graph $0\!-\!1\!-\!2$, and set
$\lambda_{01}:=\ach(a_{01})$ and
$\lambda_{12}:=\ach(a_{12})$. After fixing
$\varepsilon_{\{0,1\}}=1$, the two choices
$\varepsilon_{\{1,2\}}=\pm1$ give $\bt_1=\lambda_{01}$,
and $\bt_2=\lambda_{01}\pm\lambda_{12}$. Thus the prescribed entries remain unchanged, whereas $b_{02}=\ch(\lambda_{01}\pm\lambda_{12})$
may depend on the choice of sign.
\end{remark}

% \begin{remark}
% Theorem~\ref{T:tree-completion} gives the entire Lorentz-Gram completion directly.
% Moreover, the kernel appearing in the anchor criterion of Theorem~\ref{T:LG-ker-char} is explicitly $\bp_0(B):=ss^T$ where each 
% $s_i=\sh(\bt_i)$. Thus, in the tree case, the positive semidefinite kernel associated
% with the anchor reduction admits an explicit rank-one realization once
% an edge-signing is fixed.
% \end{remark}

% \begin{remark}[Geometric meaning of the edge-signing]
% For an edge $e=\{i,j\}$, the prescribed value
% $a_{ij}=\ch(\lambda_e)$ determines the hyperbolic distance
% $\lambda_e$, but not the orientation of the corresponding
% displacement along $\Lob(\R)$. Indeed,
% \[
% \bt_j-\bt_i=\pm\lambda_e,
% \]
% and the sign $\varepsilon_e$ records this choice. Thus the signed
% path coordinates encode the successive choices of direction made while
% placing the vertices on a common hyperbolic geodesic. Different
% edge-signings may consequently produce different values for the
% unspecified entries while preserving all prescribed entries. 

% For instance, consider the path graph $0\!-\!1\!-\!2$, and set
% $\lambda_{01}:=\ach(a_{01})$ and
% $\lambda_{12}:=\ach(a_{12})$. After fixing
% $\varepsilon_{\{0,1\}}=1$, the two choices
% $\varepsilon_{\{1,2\}}=\pm1$ give $\bt_1=\lambda_{01}$,
% and $\bt_2=\lambda_{01}\pm\lambda_{12}$. Thus the prescribed entries remain unchanged, whereas $b_{02}=\ch(\lambda_{01}\pm\lambda_{12})$
% may depend on the choice of sign.
% \end{remark}

\begin{remark}[A global sign redundancy]
The edge-signing parametrization has a global twofold redundancy.
Indeed, replacing $\varepsilon$ by $-\varepsilon$ sends each
$\bt_i$ to $-\bt_i$, while $\ch\bigl((-\bt_i)-(-\bt_j)\bigr)
=
\ch(\bt_i-\bt_j)$. Thus $\varepsilon$ and $-\varepsilon$ produce the same completion.
\end{remark}

%\begin{remark}[The canonical rooted completion]
% The constant edge-signing $\varepsilon_e=1$, $e\in E$, determines a distinguished completion, which we call the canonical rooted completion relative to the anchor $0$. 
\begin{remark}[The all-positive signing]
The constant edge-signing $\varepsilon_e=1$, $e\in E$, determines a
distinguished completion relative to the root $0$. In this case, $\bt_i=\sum_{e\in P(0,i)}\lambda_e$, so $\bt_i$ is the weighted graph distance from $0$ to $i$, with edge lengths $\lambda_e$. Consequently, $b_{i0}
=
\ch\left(\sum_{e\in P(0,i)}\lambda_e\right)$for $i\in V$. The canonical anchor row may also be constructed recursively. Suppose that $j$ is a child of $i$ in the tree rooted at $0$. Then $\bt_j=\bt_i+\lambda_{\{i,j\}}$, and hence
\begin{align}
b_{j0}
&=
\ch\bigl(\bt_i+\lambda_{\{i,j\}}\bigr)\notag\\
&=
\ch(\bt_i)\ch(\lambda_{\{i,j\}})
+
\sh(\bt_i)\sh(\lambda_{\{i,j\}})\notag\\
&=
b_{i0}a_{ij}
+
\sqrt{b_{i0}^{\,2}-1}\sqrt{a_{ij}^{\,2}-1}.
\label{eq:tree-anchor-recursion}
\end{align}
Indeed, in this construction one has $\bt_i\geq0$ and $\lambda_{\{i,j\}}\geq0$, so that $\sh(\bt_i)=\sqrt{b_{i0}^{\,2}-1}$ and 
$\sh(\lambda_{\{i,j\}})=\sqrt{a_{ij}^{\,2}-1}$. Starting from $b_{00}=1$, formula \eqref{eq:tree-anchor-recursion} determines the entire anchor row and column by proceeding successively away from the root.
\end{remark}

\begin{remark}[Geodesic rectification]
Since $\lambda_{\{i,j\}}=\ach(a_{ij})$ is the hyperbolic distance
prescribed on the edge $\{i,j\}$, Theorem~\ref{T:tree-completion} may be viewed as a
geodesic rectification of the edge-weighted rooted tree in
$\Lob(\R)$, with the edge-signing recording the direction of each
edge displacement along the geodesic. See Figure~\ref{fig:tree-rect-hyperbola} for a visual guide.
\end{remark}

For a finite tree, the signed-path completion and its realization can be constructed recursively as follows.

\begin{algorithm}[H]
\caption{Signed-path Lorentz-Gram completion and realization of a tree}
\label{alg:one-geodesic-tree-completion}
\begin{algorithmic}[1]

\Require
A finite tree $G=(V,E)$ with $a_{ii}=1$ and prescribed edge entries
$a_{ij}\geq1$; a root $0\in V$; an edge-signing
$\varepsilon:E\to\{\pm1\}$.

\Ensure
Points $(u_i)_{i\in V}\subset\Lob(\R)$ and their Lorentz-Gram
matrix $B$ completing $A$.

\State Orient every edge away from $0$ and order
$V=\{v_1,\ldots,v_n\}$ so that $v_1=0$ and every parent precedes
its children.
\State $\bt_{v_1}\gets0$ and $u_{v_1}\gets(1,0)$.

\For{$k=2,\ldots,n$}
    \State Let $v_{\pi(k)}$ be the parent of $v_k$.
    \State $\lambda_k\gets\ach(a_{v_{\pi(k)}v_k})$.
    \State
    $\bt_{v_k}\gets
    \bt_{v_{\pi(k)}}+
    \varepsilon_{\{v_{\pi(k)},v_k\}}\lambda_k$.
    \State
    $u_{v_k}\gets\bigl(\ch\bt_{v_k},\sh\bt_{v_k}\bigr)$.
\EndFor

\State
$B\gets\bigl(\ch(\bt_i-\bt_j)\bigr)_{i,j\in V}$.
\State \Return $\bigl((u_i)_{i\in V},B\bigr)$.

\end{algorithmic}
\end{algorithm}

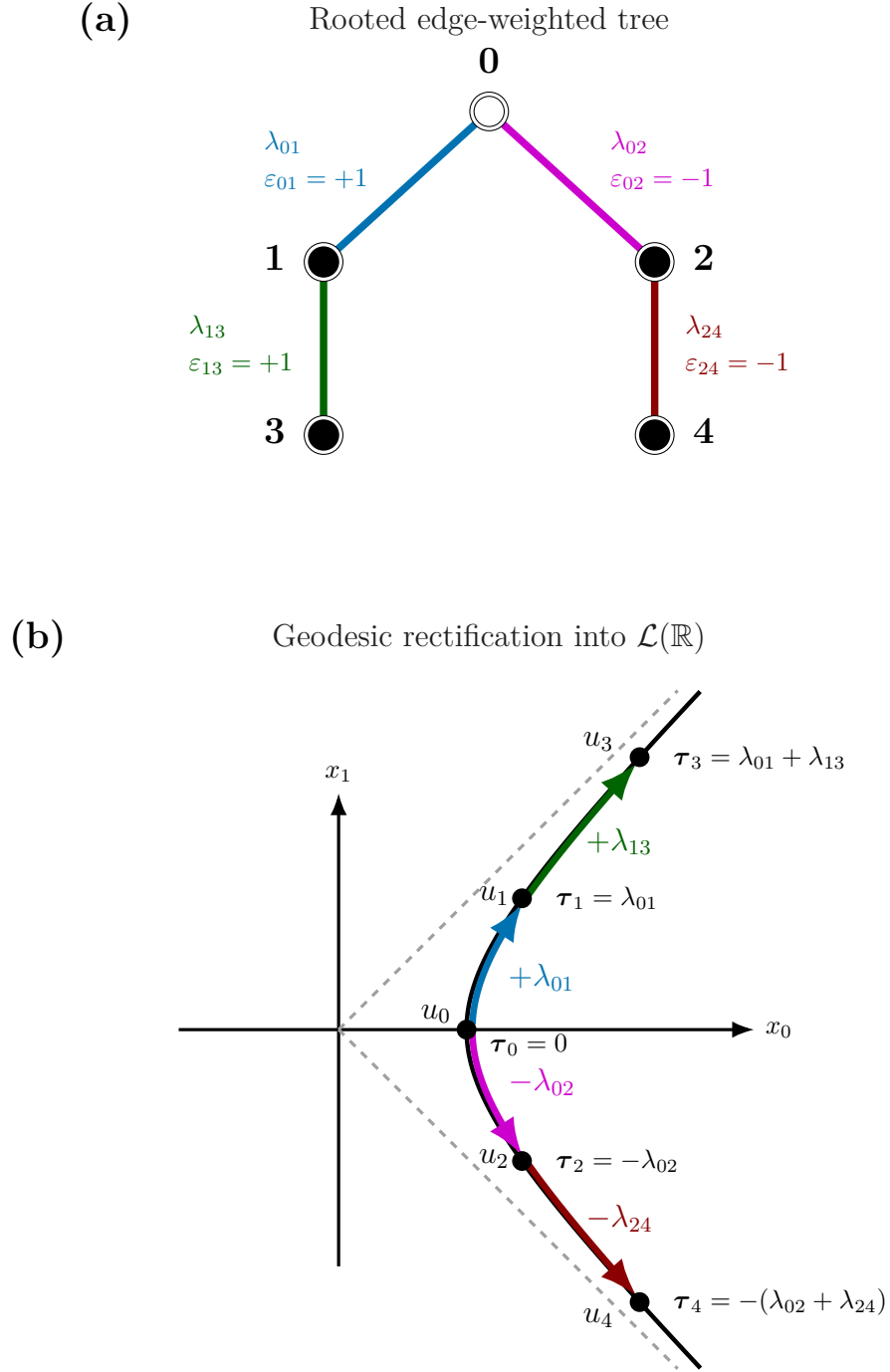
\begin{figure}[p]
\centering

\begin{tikzpicture}[
    >=Latex,
    every node/.style={font=\small},
    edgeAstyle/.style={draw=edgeA, line width=2.8pt},
    edgeBstyle/.style={draw=edgeB, line width=2.8pt},
    edgeCstyle/.style={draw=edgeC, line width=2.8pt},
    edgeDstyle/.style={draw=edgeD, line width=2.8pt},
    hyperbola/.style={draw=black, line width=1.7pt},
    asymptote/.style={draw=gray!75, dashed, line width=0.8pt},
    axis/.style={draw=black, thick, -{Latex[length=2.8mm]}},
    point/.style={circle, fill=black, inner sep=2.4pt},
    displacement/.style={
        line width=2.6pt,
        postaction={
            decorate
        },
        decoration={
            markings,
            mark=at position 0.62 with
            {\arrow{Latex[length=3.2mm,width=2.4mm]}}
        }
    }
]

%%%%%%%%%%%%%%%%%%%%%%%%%%%%%%%%%%%%%%%%%%%%%%%%%%%%%%%%%%%%%
%% PANEL (a): ROOTED TREE
%%%%%%%%%%%%%%%%%%%%%%%%%%%%%%%%%%%%%%%%%%%%%%%%%%%%%%%%%%%%%
\begin{scope}
    \node[font=\bfseries\Large] at (-5.1,1.2) {(a)};
    \node[font=\large, text=titlecol] at (0,1.2)
        {Rooted edge-weighted tree};

    % Coordinates
    \coordinate (r)  at (0,0);
    \coordinate (l1) at (-2.2,-2.0);
    \coordinate (r1) at ( 2.2,-2.0);
    \coordinate (l2) at (-2.2,-4.3);
    \coordinate (r2) at ( 2.2,-4.3);

    % Edges
    \draw[edgeAstyle] (r)  -- (l1); % 01
    \draw[edgeCstyle] (r)  -- (r1); % 02
    \draw[edgeBstyle] (l1) -- (l2); % 13
    \draw[edgeDstyle] (r1) -- (r2); % 24

    % Nodes
    \node[
        circle,
        draw,
        double,
        double distance=1.2pt,
        fill=white,
        inner sep=4.6pt
    ] at (r) {};

    \foreach \P in {l1,r1,l2,r2}
    {
        \node[
            circle,
            draw,
            double,
            double distance=1.2pt,
            fill=black,
            inner sep=4.6pt
        ] at (\P) {};
    }

    % Vertex labels
    \node[font=\Large] at (0,0.68)       {$\mathbf 0$};
    \node[font=\Large] at (-2.85,-1.95)  {$\mathbf 1$};
    \node[font=\Large] at ( 2.85,-1.95)  {$\mathbf 2$};
    \node[font=\Large] at (-2.85,-4.25)  {$\mathbf 3$};
    \node[font=\Large] at ( 2.85,-4.25)  {$\mathbf 4$};

    % Edge labels
    \node[align=left, text=edgeA] at (-2.3,-0.65)
    {
        $\lambda_{01}$\\[0.6mm]
        $\varepsilon_{01}=+1$
    };

    \node[align=left, text=edgeC] at (2.3,-0.65)
    {
        $\lambda_{02}$\\[0.6mm]
        $\varepsilon_{02}=-1$
    };

    \node[align=left, text=edgeB] at (-3.3,-3.10)
    {
        $\lambda_{13}$\\[0.6mm]
        $\varepsilon_{13}=+1$
    };

    \node[align=left, text=edgeD] at (3.3,-3.10)
    {
        $\lambda_{24}$\\[0.6mm]
        $\varepsilon_{24}=-1$
    };

\end{scope}

%%%%%%%%%%%%%%%%%%%%%%%%%%%%%%%%%%%%%%%%%%%%%%%%%%%%%%%%%%%%%
%% PANEL (b): RECTIFICATION ON THE HYPERBOLA
%%%%%%%%%%%%%%%%%%%%%%%%%%%%%%%%%%%%%%%%%%%%%%%%%%%%%%%%%%%%%
\begin{scope}[yshift=-9.0cm]

\node[font=\bfseries\Large] at (-6.0,2.0) {(b)};

\node[font=\large, text=titlecol] at (0,2.0) {Geodesic rectification into $\Lob(\R)$};

%\node[font=\large] at (4,3.15) {$\gamma(t)=(\ch t,\sh t)$};

% TikZ-safe definitions of cosh and sinh
\pgfmathdeclarefunction{mycosh}{1}{%
    \pgfmathparse{(exp(#1)+exp(-#1))/2}%
}
\pgfmathdeclarefunction{mysinh}{1}{%
    \pgfmathparse{(exp(#1)-exp(-#1))/2}%
}

% Numerical parameters for the placement of points
\def\Tfour{-1.50}
\def\Ttwo{-0.90}
\def\Tzero{0}
\def\Tone{0.90}
\def\Tthree{1.50}

\begin{scope}[shift={(-2,-3.2)},xscale=1.7,yscale=1.7]

% Axes
\draw[->,very thick](-1.25,0) -- (3.25,0) node[right] {$x_0$};

\draw[->,very thick](0,-1.85) -- (0,1.85) node[above] {$x_1$};

% Asymptotes
\draw[dashed,gray!75,very thick](0,0) -- (2.65,2.65);

\draw[dashed,gray!75,very thick](0,0) -- (2.65,-2.65);

% The genuine hyperbola

\draw[
    black,
    line width=1.7pt,
    domain=-1.70:1.70,
    samples=160,
    smooth,
    variable=\t
]
plot ({mycosh(\t)},{mysinh(\t)});

% Coordinates of the five points

\coordinate (u4) at ({mycosh(\Tfour)},{mysinh(\Tfour)});
\coordinate (u2) at ({mycosh(\Ttwo)},{mysinh(\Ttwo)});
\coordinate (u0) at ({mycosh(\Tzero)},{mysinh(\Tzero)});
\coordinate (u1) at ({mycosh(\Tone)},{mysinh(\Tone)});
\coordinate (u3) at ({mycosh(\Tthree)},{mysinh(\Tthree)});

% Colored displacement arcs

\draw[
    draw=edgeA,
    line width=2.5pt,
    domain=0.04:0.66,
    samples=50,
    smooth,
    variable=\t
]
plot ({1.045*mycosh(\t)},{1.045*mysinh(\t)});

\draw[
    draw=edgeB,
    line width=2.5pt,
    domain=0.85:1.36,
    samples=55,
    smooth,
    variable=\t
]
plot ({1.045*mycosh(\t)},{1.045*mysinh(\t)});

\draw[
    draw=edgeC,
    line width=2.5pt,
    domain=-0.66:-0.009,
    samples=50,
    smooth,
    variable=\t
]
plot ({1.045*mycosh(\t)},{1.045*mysinh(\t)});

\draw[
    draw=edgeD,
    line width=2.5pt,
    domain=-1.4:-0.86,
    samples=55,
    smooth,
    variable=\t
]
plot ({1.045*mycosh(\t)},{1.045*mysinh(\t)});

% Arrowheads placed separately

\draw[->,draw=edgeA,line width=2.5pt]
({1.045*mycosh(0.7)},{1.045*mysinh(0.7)})
--
({1.045*mycosh(0.84)},{1.045*mysinh(0.84)});

\draw[
    ->,
    draw=edgeB,
    line width=2.5pt
]
({1.045*mycosh(1.38)},{1.045*mysinh(1.38)})
--
({1.045*mycosh(1.44)},{1.045*mysinh(1.44)});

\draw[
    ->,
    draw=edgeC,
    line width=2.5pt
]
({1.045*mycosh(-0.7)},{1.045*mysinh(-0.7)})
--
({1.045*mycosh(-0.84)},{1.045*mysinh(-0.84)});

\draw[->,
    draw=edgeD,
    line width=2.5pt
]
({1.045*mycosh(-1.38)},{1.045*mysinh(-1.38)})
--
({1.045*mycosh(-1.44)},{1.045*mysinh(-1.44)});

% Points
\foreach \P in {u4,u2,u0,u1,u3}
    \fill (\P) circle (2.2pt);

% Vertex labels
\node[font=\large,anchor=east,xshift=-6pt,yshift=6pt] at (u3) {$u_3$};
\node[font=\large,anchor=east,xshift=-1pt,yshift=1pt] at (u1) {$u_1$};
\node[font=\large,anchor=east,xshift=-2pt,yshift=6pt] at (u0) {$u_0$};
\node[font=\large,anchor=east,xshift=-1pt,yshift=1pt] at (u2) {$u_2$};
\node[font=\large,anchor=east,xshift=-6pt,yshift=-6pt] at (u4) {$u_4$};

% Signed-coordinate labels
\node[anchor=west,xshift=9pt] at (u3)
    {$\bt_3=\lambda_{01}+\lambda_{13}$};

\node[anchor=west,xshift=9pt] at (u1)
    {$\bt_1=\lambda_{01}$};

\node[anchor=west,xshift=5pt,yshift=-6pt] at (u0)
    {$\bt_0=0$};

\node[anchor=west,xshift=9pt] at (u2)
    {$\bt_2=-\lambda_{02}$};

\node[anchor=west,xshift=9pt] at (u4)
    {$\bt_4=-(\lambda_{02}+\lambda_{24})$};

% Colored displacement labels
\node[
    text=edgeA,
    font=\large,
    anchor=west
]
at ({1.18*mycosh(0.34)},{1.18*mysinh(0.34)})
{$+\lambda_{01}$};

\node[
    text=edgeB,
    font=\large,
    anchor=west
]
at ({1.13*mycosh(1.08)},{1.13*mysinh(1.08)})
{$+\lambda_{13}$};

\node[
    text=edgeC,
    font=\large,
    anchor=west
]
at ({1.18*mycosh(-0.34)},{1.18*mysinh(-0.34)})
{$-\lambda_{02}$};

\node[
    text=edgeD,
    font=\large,
    anchor=west
]
at ({1.13*mycosh(-1.08)},{1.13*mysinh(-1.08)})
{$-\lambda_{24}$};
\end{scope}
\end{scope}
\end{tikzpicture}
\caption{Geodesic rectification of a rooted edge-weighted tree in
$\Lob(\R):=\{(x_0,x_1):x_0^2-x_1^2=1,x_0>0\} = \{\gamma(\bt):=(\ch \bt,\sh \bt):\bt \in \R\}$. \textup{(a)} A rooted tree equipped with hyperbolic edge lengths $\lambda_e=\ach(a_{ij})$ and an edge-signing
$\varepsilon_e\in\{\pm1\}$. \textup{(b)} The resulting Lobachevsky realization on $\gamma(\bt)$. Each vertex $i$ is assigned the signed coordinate $\bt_i=\sum_{e\in P(0,i)}\varepsilon_e\lambda_e$. Thus $u_i=\gamma(\bt_i)=\bigl(\ch\bt_i,\sh\bt_i\bigr)$, and every tree edge is represented by an arc of the hyperbola having the prescribed hyperbolic length and signed orientation.}
\label{fig:tree-rect-hyperbola}
\end{figure}

\section{Orthogonal innovation}

We develop a combinatorial-Lorentzian principle that couples classical representations of chordal graphs, with orthogonal innovation, to produce explicit Lorentz-Gram completions. The search for such explicit completions leads naturally to simplicial elimination and, in the general chordal setting, to the complementary clique-tree representations. Simplicial elimination builds a chordal graph by successively adjoining vertices along cliques of earlier neighbours. This combinatorial procedure has a natural geometric counterpart. Given a Lorentz-Gram realization of the earlier vertices, the prescribed inner products determine the component of each new vector in the span of its earlier-neighbour clique, and the remaining component is then placed in a fresh negative direction orthogonal to all previously constructed vectors. We call this component the \emph{orthogonal innovation}.

We begin with trees, where each nonempty earlier-neighbour clique is a singleton and simplicial elimination reduces to adjoining leaves one at a time. In this case, orthogonal innovation reveals the product-distance completion and its underlying path geometry. We then pass to chordal graphs. Here the earlier-neighbour sets are cliques rather than single vertices, while clique-trees replace the unique vertex paths of a tree by unique paths through maximal cliques. The resulting scalar path products are thereby naturally replaced by matrix-valued transfers across clique separators.

\subsection{Simplicial elimination and product-distance matrices}\label{SS:orthogonal-innovation}

We first implement the orthogonal innovation principle for trees. The
combinatorial mechanism is simplicial elimination (which will also
underlie the subsequent chordal case along with more involved clique-trees). The existence of a simplicial vertex is due to Dirac \cite{Dirac1961}, while its iterative
use appears in the work of Fulkerson and Gross \cite[Section~7]{FulkersonGross1965}. We recall the precise form needed
here.

\begin{theorem}[Simplicial elimination
\cite{Dirac1961,FulkersonGross1965}]
\label{T:chordal-seo}
A finite simple graph $G=(V,E)$ is chordal if and only if its vertices
can be labelled as $V=\{v_1,\ldots,v_n\}$ so that, for every $k$, the
vertex $v_k$ is simplicial in the subgraph induced by
$\{v_1,\ldots,v_k\}$. Equivalently, the set of earlier neighbours
$N(v_k)\cap\{v_1,\ldots,v_{k-1}\}$ is a clique for every $k$, where
$N(v):=\{w\in V:\{v,w\}\in E\}$ denotes the open neighbourhood of
$v$. Such a labelling is called a \emph{simplicial elimination
ordering}; the vertices are successively eliminated in the order
$v_n,v_{n-1},\ldots,v_1$.
\end{theorem}

% For a tree, simplicial elimination reduces to adjoining leaves one at a time. In this special (chordal) case, the orthogonal innovation idea yields a distinguished completion: the product-distance matrix of Bapat and Sivasubramanian \cite{BapatSivasubramanian2012}, which extends the exponential distance matrices studied by Bapat, Lal, and Pati \cite{BapatLalPati2006}. We present this canonical completion and exhibit its hyperbolic geometry through an orthogonal innovation realization in the proof and in Algorithm~\ref{alg:tree-orthogonal-innovation}, and schematically in Figure~\ref{F:tree-orthogonal-innovation}.

For a tree, simplicial elimination reduces to adjoining leaves one at a
time. We now show that the corresponding orthogonal-innovation
construction produces exactly the product-distance completion of
Theorem~\ref{T:tree-path-product-completion}.

% \begin{theorem}[Product-distance Lobachevsky completion]
% \label{T:tree-path-product-completion}
% Let $A=(a_{ij})_{i,j\in V}$ be a partial symmetric matrix whose specification graph is a finite tree $G=(V,E)$. Assume that $a_{ii}=1$ for all $i\in V$ and $a_{ij}\geq1$ for all $\{i,j\}\in E$. For an edge $e=\{i,j\}$, write $a_e:=a_{ij}$. For $i,j\in V$, let $P(i,j)$ denote the unique path from $i$ to $j$, viewed as its set of edges, and define
% \begin{align*}
% \label{eq:canonical-tree-completion}
% b_{ij}:=\prod_{e\in P(i,j)}a_e,
% \end{align*}
% where the empty product is equal to $1$. Then $B=(b_{ij})_{i,j\in V}$ is a Lorentz-Gram completion of $A$. Moreover, $\log [B]:=(\log b_{ij})_{i,j\in V}$ is exactly the additive tree pseudometric with edge lengths $\log a_e$, and it is a metric when each $a_e>1$. In summary, the completion preserves the prescribed hyperbolic lengths $\ach(a_e)$ on the edges, while its logarithm recovers the additive tree geometry associated with the transformed edge lengths $\log a_e$.
% \end{theorem}

\begin{proof}[Proof of Theorem~\ref{T:tree-path-product-completion}]
It is immediate that $B$ agrees with $A$ on every prescribed entry. Assume first that $V$ is finite; the general case is reduced to this one at the end of the proof. We prove that $B$ is a Lorentz-Gram matrix by constructing an orthogonal innovation. Choose an arbitrary root $v_1$ and label the vertices $V=\{v_1,\ldots,v_n\}$ so that every parent precedes its children. For $k\geq2$, let $v_{\pi(k)}$ denote the parent of $v_k$; so $\pi(k)<k$. Let $H$ be a real Hilbert space of dimension at least $|V|-1$. Fix $u_\ast\in\Lob(H)$, and set $u_{v_1}:=u_\ast$. Inductively, choose a nonzero vector $e_k$ orthogonal to all previously constructed vectors and normalize it so that
$[e_k,e_k]=-1$ (Lemma~\ref{L:neg-def}). Define
\begin{equation}
\label{eq:tree-orthogonal-innovation}
u_{v_k}
:=
a_{v_{\pi(k)}v_k} u_{v_{\pi(k)}}
+
\sqrt{a_{v_{\pi(k)}v_k}^2-1}\,e_k.
\end{equation}
Since $[u_{v_{\pi(k)}},u_{v_{\pi(k)}}]=1$, $[e_k,e_k]=-1$, and $[e_k,u_{v_{\pi(k)}}]=0$, we obtain $[u_{v_k},u_{v_k}]
= 1$. Moreover, $[u_{v_k},u_{v_{\pi(k)}}]=a_{v_{\pi(k)}v_k}\geq1$. Because $u_{v_{\pi(k)}}$ lies on the positive sheet, this shows that
$u_{v_k}$ lies on the same sheet. Hence $u_{v_k}\in\Lob(H)$. Next, for every $j<k$, the orthogonality of $e_k$ gives $[u_{v_k},u_{v_j}]=a_{v_{\pi(k)}v_k}[u_{v_{\pi(k)}},u_{v_j}]$. Iterating this along the unique path from $v_k$ to $v_j$
yields
\[
[u_{v_k},u_{v_j}]
=
\prod_{e\in P(v_k,v_j)}a_e
=
b_{v_kv_j}.
\]
It follows that $B=\bigl([u_i,u_j]\bigr)_{i,j\in V}$, and hence $B$ is a Lorentz-Gram completion of $A$. It remains to identify the metric encoded by log of the entries. By definition,
\[
\log b_{ij}
=
\log\bigg[\prod_{e\in P(i,j)}a_e\bigg]
=
\sum_{e\in P(i,j)}\log a_e.
\]
Since $a_e\geq1$, $\log a_e$ are nonnegative. Hence
$\log b_{ij}$ is precisely the path pseudometric on $G$ with edge
lengths $\log a_e$. If $a_e>1$ for every $e\in E$, then all these
edge lengths are strictly positive, and therefore
$\log b_{ij}=0$ if and only if $i=j$; thus $\log[B]$ is a metric.

Finally, let $V$ be arbitrary. The entries of $B$ are defined by finite products and the identity $b_{ij}=a_{ij}$ on prescribed entries has already been verified, so it remains only to produce a Lorentz-Gram realization. Let $F\subseteq V$ be finite and nonempty, fix $i_0\in F$, and let $T_F$ be the subgraph of $G$ with edge set $\bigcup_{i\in F}P(i_0,i)$. Being a finite union of paths through $i_0$, the graph $T_F$ is connected, and it is acyclic as a subgraph of a tree; hence $T_F$ is a finite tree, and $F\subseteq V(T_F)$. Moreover $P(i,j)\subseteq P(i,i_0)\cup P(i_0,j)\subseteq T_F$ for all $i,j\in V(T_F)$, so $T_F$ carries the $G$-paths between its own vertices, and $T_F=G[V(T_F)]$ because a further edge of $G$ inside $V(T_F)$ would create a cycle. Consequently the restriction of $B$ to $V(T_F)$ is exactly the product-distance completion of the restriction of $A$ to $V(T_F)$, and is therefore a Lorentz-Gram matrix by the finite case treated above. Its principal submatrix $B_{F\times F}$ is then a Lorentz-Gram matrix as well, so it has at most one positive eigenvalue; since $b_{ii}=1$, it has exactly one. As $F$ was an arbitrary finite subset, the kernel $B$ has one positive square, and Theorem~\ref{T:LG-ker-char} provides a real Hilbert space $H$ and vectors $(u_i)_{i\in V}\subseteq\Lob(H)$ with $b_{ij}=[u_i,u_j]$ for all $i,j\in V$. Thus $B$ is a Lorentz-Gram completion of $A$ for every index set $V$.

% For infinite $V$ the assertion reduces to the finite case. Let $F\subseteq V$ be
% finite and let $T_F$ be the subtree of $G$ spanned by $F$, that is, the union of
% the paths $P(i,j)$ for $i,j\in F$; it is a finite tree whose edges are edges of
% $G$. For $i,j\in V(T_F)$ the path $P(i,j)$ lies in $T_F$, so the restriction of
% $B$ to $V(T_F)$ is precisely the product-distance completion of the restriction
% of $A$, and is therefore a Lorentz--Gram matrix by the finite case. Hence
% $B_{F\times F}$, a principal submatrix of it, is a Lorentz--Gram matrix, so it
% has one positive square. As $F$ was arbitrary and $b_{ii}=1$, the kernel $B$ has
% one positive square, and Theorem~\ref{T:LG-ker-char} yields a real Hilbert space
% $H$ and vectors $(u_i)_{i\in V}\subseteq\Lob(H)$ with $b_{ij}=[u_i,u_j]$.}

\end{proof}

\begin{remark}[Geodesic interpretation of orthogonal innovation]
The orthogonal innovation has an interpretation in terms
of hyperbolic geodesics. Indeed,
let $u\in\Lob(H)$ and let $e\in u^{\perp}$ be such that $[e,e]=-1$ (Lemma~\ref{L:neg-def}). The curve
\[
\gamma(t):=\cosh(t)u+\sinh(t)e,
\qquad t\in \R,
\]
lies in the subspace spanned by $u$ and $e$, and
satisfies $[\gamma(t),\gamma(t)]=1$. It therefore traces the
intersection of this plane with the hyperboloid, which is a
hyperbolic geodesic. Moreover, $\gamma(0)=u$ and $\gamma'(0)=e$, so
the geodesic issues from $u$ in the innovation direction $e$. Since
$[\gamma'(t),\gamma'(t)]=-1$, it is parametrized by hyperbolic arc
length. To apply this observation to
Theorem~\ref{T:tree-path-product-completion}, set
$\lambda_k:=\ach(a_{v_{\pi(k)}v_k})$. Then
$a_{v_{\pi(k)}v_k}=\cosh\lambda_k$ and
$\sqrt{a_{v_{\pi(k)}v_k}^{\,2}-1}=\sinh\lambda_k$, so the
orthogonal innovation formula becomes $u_{v_k}
=
\cosh(\lambda_k)u_{v_{\pi(k)}}
+
\sinh(\lambda_k)e_k$. Thus, $u_{v_k}$ is obtained by travelling the prescribed hyperbolic
distance $\lambda_k$ from $u_{v_{\pi(k)}}$ along the geodesic
determined by the new orthogonal direction $e_k$. This is illustrated in
Figure~\ref{F:tree-orthogonal-innovation}.
\end{remark}

\begin{figure}[htbp]
\centering
\includegraphics[width=0.65\textwidth]{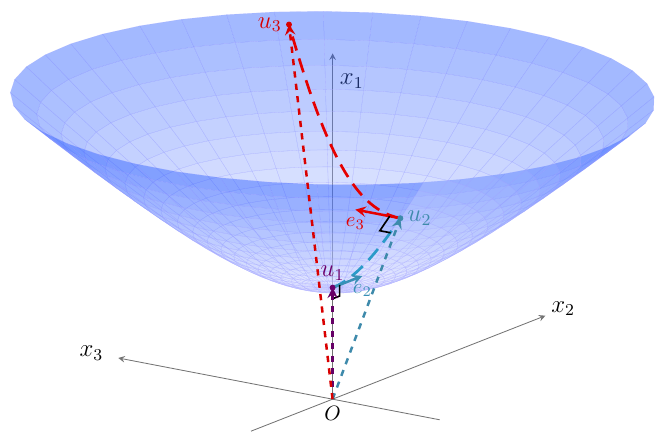}
\caption{Two orthogonal innovations in the hyperboloid
model for the path graph $v_1 - v_2 - v_3$. The geodesic from $u_1$ to $u_2$ issues in the innovation direction $e_2$, while the geodesic from $u_2$ to $u_3$ issues in the fresh direction $e_3$, which is Lorentz-orthogonal to the previously constructed vectors. The dashed rays from $O$ represent the ambient vectors $u_1,u_2,u_3$, and the dashed curved lines are Lorentzian geodesics.}
\label{F:tree-orthogonal-innovation}
\end{figure}

The orthogonal innovation also gives a direct algorithmic realization in $\Lob(\R^{n-1})$.

\begin{algorithm}[H]
\caption{Orthogonal innovation realization for a tree}
\label{alg:tree-orthogonal-innovation}
\begin{algorithmic}[1]
\Require A tree $G=(V,E)$ with $|V|=n$, with each $a_{ii}=1$, and $a_{ij}\geq1$ for $\{i,j\}\in E$
\Ensure $(u_v)_{v\in V}\subset\Lob(\R^{n-1})$ whose
Lorentz-Gram is the product-distance completion of $A$

\State Choose a root $v_1$ and re-label $V=\{v_1,\ldots,v_n\}$ so that every parent precedes its children
\State Let $v_{\pi(k)}$ be the parent of $v_k$ for $k=2,\ldots,n$
\State Let $\varepsilon_1,\ldots,\varepsilon_{n-1}$ be the standard
basis of $\R^{n-1}$
\State Set $u_{v_1}\gets(1,0)\in\R\oplus\R^{n-1}$

\For{$k=2,\ldots,n$}
%    \State $a_k\gets a_{v_{\pi(k)}v_k}$
    \State $e_k\gets(0,\varepsilon_{k-1})$
    \State $u_{v_k}\gets
    a_{v_{\pi(k)}v_k} u_{v_{\pi(k)}}+\sqrt{a_{v_{\pi(k)}v_k}^2-1}\,e_k$
\EndFor

\State \Return $(u_{v_1},\ldots,u_{v_n})$
\end{algorithmic}
\end{algorithm}

We next establish the structural properties of the product-distance
completion recorded in Proposition~\ref{P:tree-product-distance-structure}.
The determinant and inverse formulas are classical consequences of the
product-distance matrix structure, while the congruence and inertia
follow directly from successive leaf elimination. The metric comparison
then follows from the elementary bounds relating $\ach x$ and
$\log x$.

\begin{proof}[Proof of Proposition~\ref{P:tree-product-distance-structure}] The proof of $\det B$ follows from {\cite[Lemma~2.1]{BapatSivasubramanian2012}} by taking each $q_e=t_e=a_e$, and the formula for $B^{-1}$ follows from {\cite[Theorem~2.2]{BapatSivasubramanian2012}}. The final inequality is an application of $\log x\leq \ach(x) \leq \log x+\log2$ applied to $x=e^{\delta(i,j)}$. We prove the congruence and inertia in two ways, both via the matrix counterpart of the simplicial elimination in Theorem~\ref{T:chordal-seo}. For a tree, simplicial elimination is successive leaf elimination. Choose a root $v_1$ and label
$V=\{v_1,\ldots,v_n\}$ so that every parent precedes its children. Then, for each $k\geq2$, the vertex $v_k$ is a leaf of the subtree induced by $\{v_1,\ldots,v_k\}$, and its unique earlier neighbour $v_{\pi(k)}$ is its parent. Write $B_k:=B_{\{v_1,\ldots,v_k\}\times\{v_1,\ldots,v_k\}}$.

\emph{First proof (elementary row/column operations).}
For every $j<k$, the unique path from $v_k$ to $v_j$ first traverses the edge $\{v_k,v_{\pi(k)}\}$. Hence $b_{v_kv_j}=a_{v_{\pi(k)}v_k}b_{v_{\pi(k)}v_j}$. Thus, apart from its diagonal entry, the last row of $B_k$ is $a_{v_{\pi(k)}v_k}$ times the row indexed by $v_{\pi(k)}$. The single operation $R_{v_k}\longleftarrow R_{v_k}-a_{v_{\pi(k)}v_k}R_{v_{\pi(k)}}$ therefore clears every entry in the last row preceding the diagonal and changes the diagonal entry to $1-a_{v_{\pi(k)}v_k}^2$. Performing the matching column operation clears the last column and leaves $B_{k-1}$ unchanged. Consequently,
\[
B_k\sim B_{k-1}\oplus(1-a_{v_{\pi(k)}v_k}^2).
\]
Successive elimination for $k=n,n-1,\ldots,2$ gives $B\sim(1)\oplus\bigoplus_{e\in E}(1-a_e^2)$; inertia follows.

\emph{Second proof (factorization).}
Let $e_{\pi(k)}\in\R^{k-1}$ be the coordinate vector
corresponding to $v_{\pi(k)}$. The same path-product identity gives
\[
B_k
=
\begin{pmatrix}
B_{k-1} & a_{v_{\pi(k)}v_k}B_{k-1}e_{\pi(k)}\\
a_{v_{\pi(k)}v_k}e_{\pi(k)}^tB_{k-1} & 1
\end{pmatrix}.
\]
Since
$e_{\pi(k)}^tB_{k-1}e_{\pi(k)}=1$, direct multiplication yields
\[
\begin{pmatrix}
I&-a_{v_{\pi(k)}v_k}e_{\pi(k)}\\
0&1
\end{pmatrix}^{t}
B_k
\begin{pmatrix}
I&-a_{v_{\pi(k)}v_k}e_{\pi(k)}\\
0&1
\end{pmatrix}
=
B_{k-1}\oplus(1-a_{v_{\pi(k)}v_k}^2).
\]
Iterating from $B_1=(1)$ gives the required congruence and inertia.
\end{proof}

The two tree constructions serve different purposes and generally
produce different completions. The signed-path construction realizes
the data on a single geodesic and records the freedom arising from the
choices of edge directions. The product-distance construction instead
selects a root-independent completion with an explicit path-product
formula and, in the nonsingular case, a sparse inverse. It is this
second construction that extends naturally to chordal graphs through much involved clique-tree representations and simplicial eliminations.

\subsection{Clique-trees and matrix-valued transfers}\label{SS:clique-tree-completion}

The tree case in Theorem~\ref{T:tree-path-product-completion} is
pleasantly explicit for two reasons: $(a)$ every pair of vertices is joined by a unique path, and $(b)$ each new vertex is adjoined along a single earlier neighbour (via the simplicial elimination). The completed entries are therefore path products, and each new row is obtained from its parent row by scalar multiplication. For a chordal graph, simplicial elimination remains available, but vertex paths need not be unique. The appropriate replacement is a clique-tree, which organizes the maximal cliques of the given chordal specification graph $G$ along unique paths and records how they meet through clique separators. We use the classical clique-tree characterization of chordal graphs, following Blair and Peyton \cite[Section~3, Theorems~3.2--3.5]{BlairPeyton1993}; see also Gavril \cite{Gavril1974} and Buneman \cite{Buneman1974Rigid} for the original characterizations. The clique-tree separator property is discussed in \cite[Section~4]{BlairPeyton1993}.

\begin{theorem}[Clique-trees {\cite{BlairPeyton1993}}]
\label{T:clique-trees}
Let $G=(V,E)$ be a finite connected graph, and let $\ctv(G)$ denote the family of all its maximal cliques. Then the following are equivalent.
\begin{enumerate}
\item The graph $G$ is chordal.

\item \emph{(Induced subtree property.)} There exists a tree $\ct$ with vertex set $\ctv(G)$ whose edges are chosen so that, for every $v\in V$, the maximal cliques in $\ctv(G)$ containing $v$ induce a connected subtree of $\ct$. Such a tree is called a \emph{clique-tree} of $G$.

\item \emph{(Running intersection property.)} The maximal
cliques in $\ctv(G)$ can be ordered as $C_1,\ldots,C_m$ so that, for every $r\geq2$, there exists $\pi(r)<r$ satisfying
$C_r\cap(C_1\cup\cdots\cup C_{r-1})\subseteq C_{\pi(r)}$.
Joining each $C_r$ to its parent $C_{\pi(r)}$ then produces a clique
tree.
\item \emph{(Clique intersection property.)} A tree $\ct$ on $\ctv(G)$ is a clique-tree if and only if, for any two maximal cliques $C,D$, every clique on the unique path from $C$ to $D$ in $\ct$ contains $C\cap D$.
\end{enumerate}

Furthermore, if $CD$ is an edge of $\ct$ with clique-tree separator $S:=C\cap D$, then deleting $CD$ separates $\ct$ into two components $\ct_1$ and $\ct_2$. Setting $V_r:=\bigcup_{K\in V(\ct_r)}K$ for $r=1,2$, one has $V_1\cap V_2=S$, and there are no edges of $G$ between $V_1\setminus S$ and $V_2\setminus S$.
\end{theorem}

Blair and Peyton give a particularly clear and systematic account of
this structure through the induced subtree, running intersection, and clique intersection properties
{\cite[Theorems~3.2--3.5, Corollary~3.7]{BlairPeyton1993}}. Their treatment of clique separators further shows how the edges of a clique-tree encode the separation structure of the underlying chordal graph
{\cite[Section~4.1, Lemmas~4.1--4.2, Theorem~4.3]{BlairPeyton1993}}. The separator property makes the analogy with the tree case precise:
clique-tree edges carry clique separators, and passage across each
separator is governed by a matrix-valued transfer. We now show that
composing these transfers along clique-tree paths produces exactly the
completion in Theorem~\ref{T:clique-product-completion}.

\begin{proof}[Proof of Theorem~\ref{T:clique-product-completion}]
Since $A_{C\times C}$ is a nonsingular Lorentz-Gram matrix for every maximal clique $C$, Lemma~\ref{lem:sign-inn} shows that $A_{Q\times Q}$ is a nonsingular Lorentz-Gram matrix for every clique $Q$ of $G$; we use this repeatedly.

We construct the completion by rooting the clique-tree and attaching one maximal clique at a time. At each attachment, the new cross-entries are determined by transfer through the corresponding clique separator. We then verify that the resulting matrix has inertia $(1,|V|-1,0)$ and finally recover the clique-product formula.

\medskip
\noindent
\textit{\textbf{Step 1:} Decomposing the vertex set along the rooted clique-tree.}\medskip

Choose a maximal clique $C_1$ as the root of $\ct$, and order the
maximal cliques (or the vertices of $\ct$) as $C_1,\ldots,C_m$ so that every parent precedes its
children. For $r\geq2$, let $C_{\pi(r)}$ be the parent of $C_r$
and set
\[
S_r:=C_r\cap C_{\pi(r)},\qquad
R_1:=C_1,\qquad
R_r:=C_r\setminus S_r\quad (r\geq2).
\]
Then the residual sets $R_1,\ldots,R_m$ form a partition of $V$. Indeed, for each $v\in V$, the maximal cliques containing $v$
induce a connected subtree of $\ct$, by
Theorem~\ref{T:clique-trees}. Let $C_r$ be the member of this subtree
nearest to the root $C_
1$. If $r=1$, then $v\in R_1$. If $r\geq2$,
the parent $C_{\pi(r)}$ does not contain $v$, by the choice of
$C_r$, and hence $v\in R_r$. Every other clique $C_s$ containing
$v$ lies below $C_r$; its parent also contains $v$, since the
cliques containing $v$ form a connected subtree. Therefore
$v\notin R_s$. Thus every vertex belongs to exactly one $R_r$. Now, set
\[
W_r:=R_1\sqcup\cdots\sqcup R_r.
\]
The set $W_r$ is the collection of vertices introduced after the
first $r$ cliques have been processed. In particular, we have $W_m=V$.

\medskip
\noindent\emph{\textbf{Step 2:} Attaching a clique through its separator.}\medskip

We construct the principal submatrices $B_{W_r\times W_r}$
inductively for $r=1,\dots,m$. Set $B_{W_1\times W_1}:=A_{C_1\times C_1}$, and
suppose that $B_{W_{r-1}\times W_{r-1}}$ has already been
constructed. We have $C_r=S_r\sqcup R_r$, i.e., the maximal clique $C_r$ consists of the separator $S_r$ and the new residual vertices in $R_r$. Moreover, $S_r\neq\varnothing$.
Indeed, if $S_r=\varnothing$, then deleting the clique-tree edge $C_{\pi(r)}C_r$ in $\ct$ would, by the separator property in Theorem~\ref{T:clique-trees}, partition $V$ into two nonempty sets with no edges of $G$ between them, contradicting the connectedness of $G$. Finally,
$A_{S_r\times S_r}$ is nonsingular by hypothesis. 

Notice that deleting the edge $C_{\pi(r)}C_r$ in $\ct$ separates $C_r$ and its descendants in $\ct$ from the previously processed cliques. Indeed, denote the two components of the resulting forest by $\ct_1$ and $\ct_2$, where
$C_{\pi(r)}\in V(\ct_1)$ and $C_r\in V(\ct_2)$, and let
\[
V_\ell:=\bigcup_{C\in V(T_\ell)}C,\qquad \ell=1,2.
\]
By Theorem~\ref{T:clique-trees}, $V_1\cap V_2=S_r$, and there are
no edges of $G$ between $V_1\setminus S_r$ and
$V_2\setminus S_r$. Now,
$R_r=C_r\setminus S_r\subseteq V_2\setminus S_r$. Note, $\ct_2$ consists of $C_r$ and its descendants in $\ct$, and since every parent precedes its children in the chosen ordering, every previously processed clique belongs to $\ct_1$. Consequently,
$W_{r-1}\subseteq V_1$, and hence
$W_{r-1}\setminus S_r\subseteq V_1\setminus S_r$. It follows that
there are no edges of $G$ between $R_r$ and
$W_{r-1}\setminus S_r$. So, the entries in the block $R_r\times (W_{r-1}\setminus S_r)$ block are unspecified. Thus, relative to the decomposition $W_r=(W_{r-1}\setminus S_r)\sqcup S_r\sqcup R_r$, the matrix at the $r$-th attachment has the block decomposition as in Figure~\ref{F:clique-attachment-blocks}.
\begin{figure}[H]
\centering
\[
\renewcommand{\arraystretch}{1.2}
\begin{array}{c@{\qquad}c}
&
\begin{array}{@{}c@{}c@{}c@{}}
\makebox[-1cm][c]{\mbox{$W_{r-1}\setminus S_r$}}
&
\makebox[6cm][c]{\mbox{$S_r$}}
&
\makebox[0cm][c]{\mbox{$R_r$}}
\end{array}
\\
\begin{array}{c}
W_{r-1}\setminus S_r\\
S_r\\
R_r
\end{array}
&
\left(
\begin{array}{@{}c@{}|c@{}|c@{}}
\makebox[4cm][c]{\mbox{constructed}}
&
\makebox[2.4cm][c]{\mbox{constructed}}
&
\makebox[2.4cm][c]{\mbox{unspecified}}
\\ \hline
\makebox[4cm][c]{\mbox{constructed}}
&
\makebox[2.4cm][c]{\mbox{given}}
&
\makebox[2.4cm][c]{\mbox{given}}
\\ \hline
\makebox[4cm][c]{\mbox{unspecified}}
&
\makebox[2.4cm][c]{\mbox{given}}
&
\makebox[2.4cm][c]{\mbox{given}}
\end{array}
\right)=B_{_{W_r\times W_r}}.
\end{array}
\]
\caption{Block structure of the $r$-th clique attachment.}
\label{F:clique-attachment-blocks}
\end{figure}
We complete the two unspecified cross-blocks as follows. Define $B_{W_r\times W_r}$, relative to $W_r=W_{r-1}\sqcup R_r$, by retaining the previously constructed principal block $B_{W_{r-1}\times W_{r-1}}$, and setting
\[
B_{R_r\times W_{r-1}}
:=
A_{R_r\times S_r}A_{S_r\times S_r}^{-1}
B_{S_r\times W_{r-1}},
\qquad
B_{W_{r-1}\times R_r}
:=
B_{R_r\times W_{r-1}}^{\,t},
\]
and $B_{R_r\times R_r}:=A_{R_r\times R_r}$.
Equivalently,
\[
B_{W_r\times W_r}
:=
\begin{pmatrix}
B_{W_{r-1}\times W_{r-1}}
&
B_{W_{r-1}\times R_r}
\\
B_{R_r\times W_{r-1}}
&
A_{R_r\times R_r}
\end{pmatrix}.
\]
This definition preserves the prescribed entries. Indeed,
$B_{S_r\times S_r}=A_{S_r\times S_r}$, because $S_r$ is
contained in a previously processed clique. Next, we use the standard restriction rule for block matrix products \cite{ZhangSchur}: if matrix $X$ is indexed by $I\times K$, matrix $Y$ is indexed by $K\times J$, and $J'\subseteq J$, then $(XY)_{I\times J'}=X\,Y_{K\times J'}$. Hence, restricting the identity $B_{R_r\times W_{r-1}}
=
A_{R_r\times S_r}A_{S_r\times S_r}^{-1}
B_{S_r\times W_{r-1}}$ to the columns indexed by $S_r$, and then using
$A_{S_r\times S_r}^{-1}B_{S_r\times S_r}=A_{S_r\times S_r}^{-1}A_{S_r\times S_r}=I$, gives
$B_{R_r\times S_r} = A_{R_r\times S_r}$. Moreover,
$B_{R_r\times R_r}=A_{R_r\times R_r}$ by definition, while
there are no prescribed entries between $R_r$ and
$W_{r-1}\setminus S_r$. Thus every attachment preserves all
entries specified by $A$. By symmetry, the transfer rule also agrees on $S_r\times R_r$. Therefore each attachment preserves all entries specified by $A$, and after the final attachment the resulting symmetric matrix $B$ is a completion of $A$. We show that it is a Lorentz-Gram completion next.

\medskip
\noindent\emph{\textbf{Step 3:} Inertia and Lorentz-Gram realization.} \medskip

We prove by induction on $r$ that $\inn\bigl(B_{W_r\times W_r}\bigr)
=(1,|W_r|-1,0)$. For $r=1$, this follows from
$B_{W_1\times W_1}=A_{C_1\times C_1}$, which is a nonsingular
Lorentz-Gram matrix by hypothesis. Now let $r\geq2$, and suppose $B_{W_{r-1}\times W_{r-1}}$ has inertia
$(1,|W_{r-1}|-1,0)$; in particular, it is nonsingular. By Step~2, relative to $W_r=W_{r-1}\sqcup R_r$,
\[
B_{W_r\times W_r}
=
\begin{pmatrix}
B_{W_{r-1}\times W_{r-1}}
&
B_{W_{r-1}\times S_r}A_{S_r\times S_r}^{-1}
A_{S_r\times R_r}
\\[1mm]
A_{R_r\times S_r}A_{S_r\times S_r}^{-1}
B_{S_r\times W_{r-1}}
&
A_{R_r\times R_r}
\end{pmatrix}.
\]
We compute the Schur complement of
$B_{W_{r-1}\times W_{r-1}}$ in this matrix. First, since
$B_{S_r\times W_{r-1}}$ is the row block of
$B_{W_{r-1}\times W_{r-1}}$ indexed by $S_r$, the product $B_{S_r\times W_{r-1}}B_{W_{r-1}\times W_{r-1}}^{-1}$ is the corresponding row block of the identity matrix. Therefore,
on multiplying by the column block
$B_{W_{r-1}\times S_r}$, we obtain
\[
B_{S_r\times W_{r-1}}
B_{W_{r-1}\times W_{r-1}}^{-1}
B_{W_{r-1}\times S_r}
=
B_{S_r\times S_r}
=
A_{S_r\times S_r},
\]
where the last equality was established in Step~2. It follows that
the Schur complement is
\[
D_r
:=
A_{R_r\times R_r}
-
A_{R_r\times S_r}A_{S_r\times S_r}^{-1}
A_{S_r\times R_r}.
\]
The same matrix $D_r$ is the Schur complement of
$A_{S_r\times S_r}$ in the fully specified clique matrix
\[
A_{C_r\times C_r}
=
\begin{pmatrix}
A_{S_r\times S_r} & A_{S_r\times R_r}\\
A_{R_r\times S_r} & A_{R_r\times R_r}
\end{pmatrix},
\qquad C_r=S_r\sqcup R_r.
\]
By the inertia additivity formula for Schur complements \cite{Haynsworth1968}, we have
\[
\inn(A_{C_r\times C_r})
=
\inn(A_{S_r\times S_r})
+
\inn(D_r).
\]
Both $A_{C_r\times C_r}$ and $A_{S_r\times S_r}$ are
nonsingular Lorentz-Gram matrices. Their inertias are therefore
$(1,|C_r|-1,0)$ and $(1,|S_r|-1,0)$, respectively. Since
$|C_r|=|S_r|+|R_r|$, it follows that $\inn(D_r)=(0,|R_r|,0)$; thus $D_r$ is negative definite. Applying the same inertia additivity formula to the block matrix
$B_{W_r\times W_r}$ constructed in Step~2 gives
\[
\begin{aligned}
\inn\bigl(B_{W_r\times W_r}\bigr)
&=
\inn\bigl(B_{W_{r-1}\times W_{r-1}}\bigr)
+\inn(D_r)\\
&=
(1,|W_{r-1}|-1,0)+(0,|R_r|,0)=
(1,|W_r|-1,0),
\end{aligned}
\]
because $W_r=W_{r-1}\sqcup R_r$. This completes the induction. Since $W_m=V$, we obtain $\inn(B)=(1,|V|-1,0)$. In particular, $B$ is nonsingular. Therefore, by Sylvester's law of inertia, there exists an invertible matrix
$X$ such that
\[
B=X^t
\begin{pmatrix}
1&0\\
0&-I_{|V|-1}
\end{pmatrix}
X.
\]
Regarding the columns of $X$ as vectors $u_i$ in the
corresponding Lorentz space, we have $b_{ij}=[u_i,u_j]$ for all
$i,j\in V$. Since $b_{ii}=1$, every $u_i$ is a unit timelike
vector and therefore lies on one of the two sheets of the unit
hyperboloid. If $\{i,j\}\in E$, then
$[u_i,u_j]=b_{ij}=a_{ij}\geq1$, so $u_i$ and $u_j$ lie on the
same sheet. Since $G$ is connected, following paths in $G$
shows that all the vectors $u_i$ lie on the same sheet. Replacing
all of them by their negatives, if necessary, we may assume that
$u_i\in\Lob(\R^{|V|-1})$ for every $i\in V$. Hence $B$
is a nonsingular Lorentz-Gram completion of $A$.

\medskip
\noindent\emph{\textbf{Step 4:} Recovering the clique-product formula.}\medskip

Retain the clique-tree $\ct$, its root, and the ordering
$C_1,\ldots,C_m$ used in previous steps. We prove by induction along the same sequence of clique attachments, that the entries of the matrix $B$ constructed in Step~2 are given by the clique-product formula stated in the theorem. We first introduce notation. Let $D$ and $E$ be maximal cliques of $G$, and let $i\in D$ and $j\in E$. If $D\neq E$ and $D=D_0,D_1,\ldots,D_\ell=E$ is the path from $D$ to $E$ in $\ct$, with $T_k:=D_{k-1}\cap D_k$, we define
\begin{align}\label{eq:5.1}
\Phi_{D,E}(i,j):=
A_{\{i\}\times T_1}A_{T_1\times T_1}^{-1}
A_{T_1\times T_2}A_{T_2\times T_2}^{-1}
\cdots
A_{T_{\ell-1}\times T_\ell}A_{T_\ell\times T_\ell}^{-1}
A_{T_\ell\times\{j\}}.
\end{align}
When $D=E$, we define $\Phi_{D,D}(i,j):=a_{ij}$. This expression is well defined: every $T_k$ is a clique, and
$T_{k-1}\cup T_k\subseteq D_{k-1}$, so all the displayed blocks
are prescribed and every separator matrix $A_{T_k\times T_k}$ is nonsingular.

We proceed by induction on the attachment index
$r=1,\ldots,m$. Recall that, after the $r$-th attachment,
the processed vertex set is $W_r=R_1\sqcup\cdots\sqcup R_r$, and Step~2 has constructed $B_{W_r\times W_r}$. We prove that, at each stage $r$, the following assertion holds: for every $i,j\in W_r$, and every pair of (processed) maximal cliques $D,E\in\{C_1,\ldots,C_r\}$ satisfying $i\in D$ and $j\in E$, one has
\[
\Phi_{D,E}(i,j)=b_{ij}.
\]

We begin with the $r=1$ case. Here, the only processed clique is $C_1$, and $B_{W_1\times W_1}=A_{C_1\times C_1}$. Thus $\Phi_{C_1,C_1}(i,j)=a_{ij}=b_{ij}$, proving the assertion. Now let $r\geq2$, and suppose that the assertion holds for the matrix $B_{W_{r-1}\times W_{r-1}}$ constructed at the end of the $(r-1)$st attachment in Step~2. Recall that $C_r$ is attached to its parent $C_{\pi(r)}$ through the separator $S_r=C_r\cap C_{\pi(r)}$, and that $W_r=W_{r-1}\sqcup R_r$.

If both endpoint cliques $D$ and $E$ belong to
$\{C_1,\ldots,C_{r-1}\}$, then \eqref{eq:5.1} follows immediately from the
induction hypothesis. If $D=E=C_r$, then
$\Phi_{C_r,C_r}(i,j)=a_{ij}=b_{ij}$, because Step~2 shows that
$B$ agrees with $A$ on $C_r\times C_r$. It remains to consider the case in which exactly one endpoint clique
is $C_r$. By symmetry, suppose that $D=C_r$ and
$E\in\{C_1,\ldots,C_{r-1}\}$. Since every parent precedes its
children, no previously processed clique is a descendant of $C_r$.
Consequently, the path from $C_r$ to $E$ in $\ct$ begins with
$C_rC_{\pi(r)}$ and its first separator is $S_r=C_r\cap C_{\pi(r)}$. Write this path as
\[
C_r=D_0,\qquad D_1=C_{\pi(r)},\qquad
D_2,\ldots,D_\ell=E,
\]
and put $T_k:=D_{k-1}\cap D_k$. Thus $T_1=S_r$.

We now apply the induction hypothesis to the part of this path from
$C_{\pi(r)}$ to $E$. Fix $s\in S_r$, and apply the induction hypothesis with
\[
x=s,\qquad y=j,\qquad D=C_{\pi(r)},\qquad F=E.
\]
These choices satisfy all its assumptions: $s\in S_r\subseteq C_{\pi(r)}$, $j\in E$, and both
$C_{\pi(r)}$ and $E$ were processed before $C_r$. Consequently,
$s,j\in W_{r-1}$, and hence
\begin{align}\label{eq:5.2}
\Phi_{C_{\pi(r)},E}(s,j)=b_{sj}
\qquad\text{for every }s\in S_r.
\end{align}
After the first edge $C_rC_{\pi(r)}$ is removed from the path, the
remaining path is precisely the path from $C_{\pi(r)}$ to $E$.

More explicitly, when $\ell\geq2$, the formula for
$\Phi_{C_{\pi(r)},E}(s,j)$ begins with the row
$A_{\{s\}\times T_2}$, while all subsequent factors are independent
of $s$. Stacking these rows over $s\in S_r$ therefore gives $\bigl(\Phi_{C_{\pi(r)},E}(s,j)\bigr)_{s\in S_r}=A_{S_r\times T_2}A_{T_2\times T_2}^{-1}\cdots A_{T_\ell\times T_\ell}^{-1}A_{T_\ell\times\{j\}}$. This is precisely the column remaining in the product defining
$\Phi_{C_r,E}(i,j)$ after its first two factors,
$A_{\{i\}\times S_r}A_{S_r\times S_r}^{-1}$, have been separated.
If $\ell=1$, then $E=C_{\pi(r)}$, and the same conclusion holds
because $\bigl(\Phi_{C_{\pi(r)},E}(s,j)\bigr)_{s\in S_r}=(a_{sj})_{s\in S_r}=A_{S_r\times\{j\}}$. Thus, in both cases, we obtain the following, that is, separating the first two factors in the definition of
$\Phi_{C_r,E}(i,j)$ gives
\begin{align}\label{eq:5.3}
\Phi_{C_r,E}(i,j)
=
A_{\{i\}\times S_r}A_{S_r\times S_r}^{-1}
\bigl(\Phi_{C_{\pi(r)},E}(s,j)\bigr)_{s\in S_r}.
\end{align}
By \eqref{eq:5.2}, the column appearing on the right is
\[
\bigl(\Phi_{C_{\pi(r)},E}(s,j)\bigr)_{s\in S_r}
=
(b_{sj})_{s\in S_r}
=
B_{S_r\times\{j\}}.
\]
It follows that
\[
\Phi_{C_r,E}(i,j)
=
A_{\{i\}\times S_r}A_{S_r\times S_r}^{-1}
B_{S_r\times\{j\}}.
\]
There are now two possibilities. If $i\in R_r$, then $i$ is a newly introduced vertex, whereas $j\in E\subseteq W_{r-1}$ is an already processed vertex. Step~2 defined the entire block between these two sets by $B_{R_r\times W_{r-1}}:=A_{R_r\times S_r}A_{S_r\times S_r}^{-1}B_{S_r\times W_{r-1}}$. Taking its row indexed by $i$ and column indexed by $j$ gives $b_{ij}=A_{\{i\}\times S_r}A_{S_r\times S_r}^{-1}B_{S_r\times\{j\}}=\Phi_{C_r,E}(i,j)$, as required. If instead $i\in S_r$, since $A_{\{i\}\times S_r}$ is the row of $A_{S_r\times S_r}$ indexed by $i$, the product $A_{\{i\}\times S_r}A_{S_r\times S_r}^{-1}$ is the corresponding row of the identity matrix. Therefore,
\[
\Phi_{C_r,E}(i,j)=A_{\{i\}\times S_r}A_{S_r\times S_r}^{-1}
B_{S_r\times\{j\}} = I_{\{i\}\times S_r} B_{S_r \times \{j\}}
=b_{ij}.
\]
This proves \eqref{eq:5.1} at the $r$-th attachment and completes the induction. Finally, $W_m=V$, and every maximal clique of $G$ has been
processed. Hence, for any $i,j\in V$ and any maximal cliques
$C_0$ and $C_m$ containing $i$ and $j$, respectively,
the assertion gives
\[
b_{ij}
=
A_{\{i\}\times S_1}A_{S_1\times S_1}^{-1}
A_{S_1\times S_2}A_{S_2\times S_2}^{-1}
\cdots
A_{S_{m-1}\times S_m}A_{S_m\times S_m}^{-1}
A_{S_m\times\{j\}},
\]
where $C_0,C_1,\ldots,C_m$ is the path between the chosen endpoint
cliques and $S_r=C_{r-1}\cap C_r$. The matrix $B$ was constructed in Step~2 without choosing the initial and final cliques. The induction above shows
that every admissible choice of these cliques yields a path product
equal to the same fixed entry $b_{ij}$. Thus, for the fixed
clique-tree, the clique-product formula is independent of the choice
of the initial and final cliques. This completes the proof.
\end{proof}

\subsection{Inverse sparsity, canonicality, and realizations}\label{SS:chordal-canonicity}

The formula in Theorem~\ref{T:clique-product-completion} was
obtained from a chosen clique-tree. The resulting completion is, however, intrinsic. It is characterized by the sparsity pattern of its inverse and can be realized directly by orthogonal innovation along any simplicial elimination ordering.

\begin{proof}[Proof of Theorem~\ref{T:chordal-canonicity}]
We prove the assertions in a few parts.

\medskip
\noindent\emph{\textbf{1.} The Johnson--Lundquist inverse formula for $B=B_{\ct}$.} \medskip

For $I\subseteq V$ and a matrix $X$ indexed by $I$, let
$\iota_I(X)$ denote its extension to a $V\times V$ matrix by
zeros outside $I\times I$. We first show directly that $B$
satisfies the local inverse formula of Johnson and Lundquist
\cite{Johnson-Lundquist}. Retain the rooted clique-tree and the notation from the proof of Theorem~\ref{T:clique-product-completion}. At the $r$-th attachment, put $S:=S_r$, $R:=R_r$, and $M:=B_{W_{r-1}\times W_{r-1}}$. The Schur complement appearing in the attachment is given by the following:
\[
D_r
:=
A_{R\times R}
-
A_{R\times S}A_{S\times S}^{-1}A_{S\times R}.
\]
This is invertible by Step~3 of the proof of Theorem~\ref{T:clique-product-completion}. Applying the block inverse formula \eqref{eq:block-inverse} both to the attached matrix $B_{W_r\times W_r}$ and to the clique matrix
$A_{C_r\times C_r}$, where $C_r=S_r\sqcup R_r$, we obtain the following Johnson--Lundquist inverse formula:
\begin{align}\label{eq:john-lund}
\iota_{W_r}\!
 \left(B_{W_r\times W_r}^{-1}\right)
=
\iota_{W_{r-1}}\!
 \left(B_{W_{r-1}\times W_{r-1}}^{-1}\right)
+
\iota_{C_r}\!\left(A_{C_r\times C_r}^{-1}\right)
-
\iota_{S_r}\!\left(A_{S_r\times S_r}^{-1}\right).
\end{align}
To show this, define a matrix $Q$, with rows indexed by $W_{r-1}$ and columns indexed by $R$:
\[
Q_{S\times R}:=A_{S\times S}^{-1}A_{S\times R},
\qquad
Q_{(W_{r-1}\setminus S)\times R}:=0.
\]
Since $B_{W_{r-1}\times S}$ consists of the columns of $M$
indexed by $S$, the attachment rule in Step 2 of the proof of Theorem~\ref{T:clique-product-completion} gives $B_{W_{r-1}\times R}=MQ$. Indeed, since $Q$ is zero outside its rows indexed by $S$, only the columns of $M$ indexed by $S$ contribute to $MQ$. Explicitly, putting $U:=W_{r-1}\setminus S$, we have
\[
\begin{aligned}
MQ
&=
\begin{pmatrix}
M_{W_{r-1}\times U} & M_{W_{r-1}\times S}
\end{pmatrix}
\begin{pmatrix}
0\\
A_{S\times S}^{-1}A_{S\times R}
\end{pmatrix}\\
&=
M_{W_{r-1}\times S}A_{S\times S}^{-1}A_{S\times R}
=
B_{W_{r-1}\times S}A_{S\times S}^{-1}A_{S\times R}
=
B_{W_{r-1}\times R}.
\end{aligned}
\]
Therefore, relative to $W_r=W_{r-1}\sqcup R$, we have
\[
\begin{aligned}
\begin{pmatrix}
M & MQ\\
Q^{t}M & A_{R\times R}
\end{pmatrix}
&=
\begin{pmatrix}
B_{W_{r-1}\times W_{r-1}}
&
B_{W_{r-1}\times S}A_{S\times S}^{-1}A_{S\times R}
\\[2mm]
A_{R\times S}A_{S\times S}^{-1}B_{S\times W_{r-1}}
&
A_{R\times R}
\end{pmatrix}                                                    \\
&=
\begin{pmatrix}
B_{W_{r-1}\times W_{r-1}} & B_{W_{r-1}\times R}\\
B_{R\times W_{r-1}}       & B_{R\times R}
\end{pmatrix}
=
B_{W_r\times W_r}.
\end{aligned}
\]
Recall from the proof of Theorem~\ref{T:clique-product-completion} that the Schur complement of $M$ in the above is $D_r$. In what follows, we use the following block inverse formula
{\cite{HornJohnson}}. If $X$ and its Schur
complement ${\mathcal S}:=Z-YX^{-1}Y'$ are invertible, then
\begin{align}\label{eq:block-inverse}
\begin{pmatrix}
X & Y'\\
Y & Z
\end{pmatrix}^{-1}
=
\begin{pmatrix}
X^{-1}+X^{-1}Y'{\mathcal S}^{-1}YX^{-1}
&
-X^{-1}Y'{\mathcal S}^{-1}
\\
-{\mathcal S}^{-1}YX^{-1}
&
{\mathcal S}^{-1}
\end{pmatrix}.
\end{align}
This yields:
\[
\begin{aligned}
B_{W_r\times W_r}^{-1}
&=
\begin{pmatrix}
M^{-1}
 +M^{-1}(MQ)D_r^{-1}(Q^{t}M)M^{-1}
&
-M^{-1}(MQ)D_r^{-1}
\\[1mm]
-D_r^{-1}(Q^{t}M)M^{-1}
&
D_r^{-1}
\end{pmatrix}                                                     \\[2mm]
&=
\begin{pmatrix}
M^{-1}+QD_r^{-1}Q^{t} & -QD_r^{-1}\\
-D_r^{-1}Q^{t}        & D_r^{-1}
\end{pmatrix}.
\end{aligned}
\]
Applying the same formula to $A_{C_r\times C_r}$ with respect to $C_r= S \sqcup R$ gives the following:
\begin{align*}
A_{C_r\times C_r}^{-1} &= \begin{pmatrix}
A_{S\times S} & A_{S\times R}\\
A_{R\times S} & A_{R\times R}
\end{pmatrix}^{-1}\\
&=
\begin{pmatrix}
A_{S\times S}^{-1}
 +A_{S\times S}^{-1}A_{S\times R}D_r^{-1}
  A_{R\times S}A_{S\times S}^{-1}
&
-A_{S\times S}^{-1}A_{S\times R}D_r^{-1}
\\
-D_r^{-1}A_{R\times S}A_{S\times S}^{-1}
&
D_r^{-1}
\end{pmatrix}.
\end{align*}
Relative to $W_{r-1}=U \sqcup S$, suppose the matrix $N:=M^{-1}=
\begin{pmatrix}
N_{U\times U}&N_{U\times S}\\
N_{S\times U}&N_{S\times S}
\end{pmatrix}$. Then, for $H:=A_{S\times S}^{-1}A_{S\times R}$ with $Q=\begin{pmatrix}0\\ H\end{pmatrix}$, we get
\[
\begin{aligned}
N+QD_r^{-1}Q^t
&=
N
+
\begin{pmatrix}0\\ H\end{pmatrix}
D_r^{-1}
\begin{pmatrix}0&H^t\end{pmatrix}
=
\begin{pmatrix}
N_{U\times U} & N_{U\times S}\\
N_{S\times U} & N_{S\times S}+HD_r^{-1}H^t
\end{pmatrix}.
\end{aligned}
\]
Thus only the $S\times S$ block of $N$ receives a correction.
Similarly,
\[
-QD_r^{-1}
=
\begin{pmatrix}0\\-HD_r^{-1}\end{pmatrix},
\qquad
-D_r^{-1}Q^t
=
\begin{pmatrix}0&-D_r^{-1}H^t\end{pmatrix}.
\]
Likewise, since $H=A_{S\times S}^{-1}A_{S\times R}$ and $H^t=A_{R\times S}A_{S\times S}^{-1}$, the block inverse formula for $A_{C_r\times C_r}$, relative to $C_r=S\sqcup R$, becomes
\[
\begin{aligned}
A_{C_r\times C_r}^{-1}
&=
\begin{pmatrix}
A_{S\times S}^{-1}+HD_r^{-1}H^t & -HD_r^{-1}\\
-D_r^{-1}H^t & D_r^{-1}
\end{pmatrix}.
\end{aligned}
\]
Therefore, relative to $W_r=U\sqcup S\sqcup R$, the nonzero part of the desired is given by:
\[
\begin{aligned}
&\iota_{W_r}\!\left(B_{W_r\times W_r}^{-1}\right)
-\iota_{W_{r-1}}\!\left(M^{-1}\right)
-\iota_{C_r}\!\left(A_{C_r\times C_r}^{-1}\right)
+\iota_S\!\left(A_{S\times S}^{-1}\right)                         \\[1mm]
&=
\begin{pmatrix}
N_{U\times U}
&
N_{U\times S}
&
0
\\
N_{S\times U}
&
N_{S\times S}+HD_r^{-1}H^t
&
-HD_r^{-1}
\\
0
&
-D_r^{-1}H^t
&
D_r^{-1}
\end{pmatrix}
-
\begin{pmatrix}
N_{U\times U}&N_{U\times S}&0\\
N_{S\times U}&N_{S\times S}&0\\
0&0&0
\end{pmatrix}                                                     \\[2mm]
&\quad-
\begin{pmatrix}
0&0&0\\
0&A_{S\times S}^{-1}+HD_r^{-1}H^t&-HD_r^{-1}\\
0&-D_r^{-1}H^t&D_r^{-1}
\end{pmatrix}
+
\begin{pmatrix}
0&0&0\\
0&A_{S\times S}^{-1}&0\\
0&0&0
\end{pmatrix}                                               =0.
\end{aligned}
\]
Therefore, the desired identity follows. We now iterate this, and claim:
\begin{equation}\label{eq:john-lund-inductive}
\iota_{W_r}\!\left(B_{W_r\times W_r}^{-1}\right)
=
\sum_{q=1}^{r}
\iota_{C_q}\!\left(A_{C_q\times C_q}^{-1}\right)
-
\sum_{q=2}^{r}
\iota_{S_q}\!\left(A_{S_q\times S_q}^{-1}\right) \qquad \mbox{for }r=1,\ldots,m.
\end{equation}
We begin with the base case $r=1$. Here we have $W_1=C_1$ and $B_{W_1\times W_1}=A_{C_1\times C_1}$. Hence
$\iota_{W_1}\!\left(B_{W_1\times W_1}^{-1}\right)=\iota_{C_1}\!\left(A_{C_1\times C_1}^{-1}\right)$, which is \eqref{eq:john-lund-inductive}, since the second sum is empty. Now let $r\geq2$, and suppose that \eqref{eq:john-lund-inductive} holds at stage $r-1$. Applying \eqref{eq:john-lund} at the $r$-th attachment and then using the induction hypothesis gives the following:
\[
\begin{aligned}
&\iota_{W_r}\!\left(B_{W_r\times W_r}^{-1}\right) \\
&=
\iota_{W_{r-1}}\!\left(
B_{W_{r-1}\times W_{r-1}}^{-1}
\right)
+\iota_{C_r}\!\left(A_{C_r\times C_r}^{-1}\right)
-\iota_{S_r}\!\left(A_{S_r\times S_r}^{-1}\right)\\
&=
\sum_{q=1}^{r-1}
\iota_{C_q}\!\left(A_{C_q\times C_q}^{-1}\right)
-
\sum_{q=2}^{r-1}
\iota_{S_q}\!\left(A_{S_q\times S_q}^{-1}\right)
+\iota_{C_r}\!\left(A_{C_r\times C_r}^{-1}\right)
-\iota_{S_r}\!\left(A_{S_r\times S_r}^{-1}\right)\\
&=
\sum_{q=1}^{r}
\iota_{C_q}\!\left(A_{C_q\times C_q}^{-1}\right)
-
\sum_{q=2}^{r}
\iota_{S_q}\!\left(A_{S_q\times S_q}^{-1}\right).
\end{aligned}
\]
This proves \eqref{eq:john-lund-inductive} by induction. Finally, $W_m=V$, so $\iota_{W_m}\!\left(B_{W_m\times W_m}^{-1}\right)=B^{-1}$. Moreover, $C_1,\ldots,C_m$ are precisely the maximal cliques of
$G$, and the parent--child edges $C_{\pi(r)}C_r$ for $r=2,\ldots,m$ are precisely the edges of the clique-tree $\ct$. Since $S_r=C_r\cap C_{\pi(r)}$, the identity at $r=m$ becomes
\begin{align}\label{eq:john-lund-1}
B^{-1}
&=
\sum_{C\in\mathcal C(G)}
\iota_C\!\left(A_{C\times C}^{-1}\right)
-
\sum_{CD\in E(\ct)}
\iota_{C\cap D}\!\left(
A_{(C\cap D)\times(C\cap D)}^{-1}
\right).
\end{align}
The second sum is indexed by the edges of $\ct$; thus, if the
same separator occurs on several clique-tree edges, it is counted
with the corresponding multiplicity.

Now suppose that $i\neq j$ and $\{i,j\}\notin E$. No clique of
$G$ contains both $i$ and $j$, since two distinct vertices
belonging to a common clique must be adjacent. Hence no maximal
clique $C$, and therefore no clique separator $C\cap D$, contains
both indices. Every term on the right-hand side of
\eqref{eq:john-lund-1} consequently has zero $(i,j)$-entry.
It follows that $(B^{-1})_{ij}=0$, which proves that $B^{-1}$ is supported on the diagonal and the edges of $G$.

The determinant formula follows from the same clique attachments. Indeed, for $r\geq2$, the matrix $D_r
=
A_{R_r\times R_r}
-
A_{R_r\times S_r}A_{S_r\times S_r}^{-1}A_{S_r\times R_r}$ is the Schur complement both of
$B_{W_{r-1}\times W_{r-1}}$ in $B_{W_r\times W_r}$ and of
$A_{S_r\times S_r}$ in $A_{C_r\times C_r}$. Hence the Schur determinant formula gives
\[
\det B_{W_r\times W_r}
=
\det B_{W_{r-1}\times W_{r-1}}\det D_r,
\qquad
\det A_{C_r\times C_r}
=
\det A_{S_r\times S_r}\det D_r.
\]
Therefore
\[
\det B_{W_r\times W_r}
=
\det B_{W_{r-1}\times W_{r-1}}
\frac{\det A_{C_r\times C_r}}
     {\det A_{S_r\times S_r}}.
\]
Since $B_{W_1\times W_1}=A_{C_1\times C_1}$, iterating the identity from $r=m$ down to $r=2$ yields 
\[
\det B
=
\left[\prod_{r=1}^m\det A_{C_r\times C_r}\right]
\left[\prod_{r=2}^m\det A_{S_r\times S_r}\right]^{-1}.
\]
Finally, the $C_r$ are the maximal cliques of $G$, while
$S_r=C_r\cap C_{\pi(r)}$, $r\geq2$, are precisely the clique separators indexed by the edges of $\ct$, with their clique-tree multiplicities. Thus the desired formula holds.

\medskip
\noindent\emph{\textbf{2.} Independence of $B$ from the given clique-tree.}
\medskip

Let $C$ be a nonsingular Lorentz-Gram completion of $A$ such that $(C^{-1})_{ij}=0$ whenever $i\neq j$ and $\{i,j\}\notin E$. Since $C$ agrees with $A$ on every clique, its principal submatrices on the maximal cliques and clique separators are the corresponding principal submatrices of $A$, all of which are nonsingular. Since $C^{-1}$ is supported on the diagonal and the edges of the chordal graph $G$, the inverse formula of Johnson and Lundquist \cite{Johnson-Lundquist} gives
\[
C^{-1}
=
\sum_{K\in\mathcal C(G)}
\iota_K\!\left(A_{K\times K}^{-1}\right)
-
\sum_{KL\in E(\ct)}
\iota_{K\cap L}\!\left(
A_{(K\cap L)\times(K\cap L)}^{-1}
\right).
\]
By the previous part, the right-hand side is $B^{-1}$. Hence $C^{-1}=B^{-1}$, and thus $C=B$.

Thus $B$ is the unique nonsingular Lorentz-Gram completion of $A$ whose inverse vanishes on the nonedges of $G$. Finally, for every clique-tree $\ct'$, Theorem~\ref{T:clique-product-completion} produces a completion $B_{\ct'}$, and the previous part shows that $B_{\ct'}^{-1}$ vanishes on the nonedges. By uniqueness, $B_{\ct'}=B$. Hence the completion is independent of the chosen clique-tree.

\medskip
\noindent\emph{\textbf{3.} The orthogonal innovation construction of $B$.}
\medskip

Fix a simplicial elimination ordering
$v_1,\ldots,v_n$ of $G$, and set
\[
W_k:=\{v_1,\ldots,v_k\},
\qquad
S_k:=N_G(v_k)\cap W_{k-1}
\quad (k\geq2).
\]
By simplicial elimination, $S_k$ is a clique. Moreover, every
induced graph $G[W_k]$ is connected. Indeed, $G[W_n]=G$ is
connected, and deleting a simplicial vertex from a connected graph
preserves connectedness: whenever a path passes through that
vertex, its two neighbouring vertices on the path are adjacent and
the path can be shortened to avoid it. Descending induction on
$k$ therefore proves the assertion. In particular,
$S_k\neq\varnothing$ for every $k\geq2$.

We verify Algorithm~
\ref{alg:chordal-orthogonal-innovation} by induction on $k$.
More precisely, we prove that the vectors indexed by $W_k$ are
well defined, that their Lorentz-Gram matrix $\widehat B_k$ is a
nonsingular completion of the restriction of $A$ to $W_k$, and
that $\widehat B_k^{-1}$ vanishes on the nonedges of $G[W_k]$.

For $k=1$, set
$u_{v_1}:=(1,0)\in\R\oplus\R^{n-1}$. Then
$\widehat B_1=(1)$, so all three assertions hold. Now let
$k\geq2$, and suppose that they hold at stage $k-1$. Put
\[
A_k:=A_{S_k\times S_k},
\qquad
g_k:=A_{S_k\times\{v_k\}},
\qquad
c_k:=A_k^{-1}g_k,
\qquad
p_k:=\sum_{s\in S_k}(c_k)_s u_s.
\]
Since $S_k$ is a clique, $A_k$ is nonsingular by hypothesis.
Furthermore, $S_k\cup\{v_k\}$ is a clique, and its prescribed
Lorentz-Gram matrix is
\[
\begin{pmatrix}
A_k & g_k\\
g_k^t & 1
\end{pmatrix}.
\]
This matrix has inertia $(1,|S_k|,0)$, whereas $A_k$ has
inertia $(1,|S_k|-1,0)$. By inertia additivity for Schur
complements \cite{Haynsworth1968}, the scalar Schur complement satisfies $1-g_k^tA_k^{-1}g_k<0$. Consequently, $\rho_k:=\sqrt{g_k^tA_k^{-1}g_k-1}$ is well defined and positive. Inductively, every vector constructed before the $k$-th stage belongs to $\R\oplus\Span\{\varepsilon_1,\ldots,\varepsilon_{k-2}\}$. Thus $e_k:=(0,\varepsilon_{k-1})$ is a negative unit vector
orthogonal to all previously constructed vectors. Define
\[
u_{v_k}:=p_k+\rho_ke_k.
\]
For $t\in S_k$, the induction hypothesis gives
\[
[p_k,u_t]
=
\sum_{s\in S_k}(c_k)_s a_{st}
=
(c_k^tA_k)_t
=
(g_k)_t
=
a_{v_kt}.
\]
Likewise, $[p_k,p_k]=c_k^tA_kc_k=g_k^tA_k^{-1}g_k$. Since $e_k$ is orthogonal to $p_k$ and
$[e_k,e_k]=-1$, it follows that $[u_{v_k},u_{v_k}]=[p_k,p_k]-\rho_k^2=1$. Hence the new vector has the prescribed inner product with every
earlier neighbour and has the prescribed diagonal entry. Since
$S_k$ is exactly the set of earlier neighbours of $v_k$,
these are all the prescribed entries involving $v_k$. Therefore
$\widehat B_k$ is a completion of the restriction of $A$ to
$W_k$. It now remains to verify the other two induction assertions, namely, that $\widehat{B}_k$ is nonsingular and that $\widehat{B}_k^{-1}$ vanishes on the nonedges of $G[W_k]$. Define a column $q_k$, indexed by $W_{k-1}$, by
\[
(q_k)_{S_k}:=c_k,
\qquad
(q_k)_{W_{k-1}\setminus S_k}:=0.
\]
For $j<k$, the orthogonality of $e_k$ gives
\[
\begin{aligned}
[u_{v_j},u_{v_k}]
&=[u_{v_j},p_k+\rho_ke_k]
=[u_{v_j},p_k]
=\sum_{s\in S_k}(c_k)_s[u_{v_j},u_s]\\
&=\sum_{s\in W_{k-1}}
   (\widehat B_{k-1})_{v_j,s}(q_k)_s
=(\widehat B_{k-1}q_k)_{v_j},
\end{aligned}
\]
where the penultimate equality follows because
$(q_k)_s=(c_k)_s$ for $s\in S_k$ and $(q_k)_s=0$ for $s\in W_{k-1}\setminus S_k$. Consequently,
\begin{equation}\label{eq:innovation-block}
\widehat B_k
=
\begin{pmatrix}
\widehat B_{k-1} & \widehat B_{k-1}q_k\\
q_k^t\widehat B_{k-1} & 1
\end{pmatrix}.
\end{equation}
Because $q_k$ is supported on $S_k$ and
$(\widehat B_{k-1})_{S_k\times S_k}=A_k$, the Schur complement
of $\widehat B_{k-1}$ in \eqref{eq:innovation-block} is given by
\[
\delta_k
:=
1-q_k^t\widehat B_{k-1}q_k
=
1-c_k^tA_kc_k
=
-\rho_k^2<0.
\]
Thus $\widehat B_k$ is nonsingular. Furthermore, the block inverse formula \eqref{eq:block-inverse} gives
\[
\widehat B_k^{-1}
=
\begin{pmatrix}
\widehat B_{k-1}^{-1}
   +\delta_k^{-1}q_kq_k^t
&
-\delta_k^{-1}q_k
\\
-\delta_k^{-1}q_k^t
&
\delta_k^{-1}
\end{pmatrix}.
\]
By the induction hypothesis,
$\widehat B_{k-1}^{-1}$ is supported on the diagonal and the
edges of $G[W_{k-1}]$. Since $q_k$ is supported on $S_k$,
the correction $q_kq_k^t$ is supported on
$S_k\times S_k$, while the new off-diagonal block is supported
on $S_k\times \{v_k\}$. All these positions lie in the clique
$S_k\cup\{v_k\}$. Hence $\widehat B_k^{-1}$ vanishes on the
nonedges of $G[W_k]$, completing the induction.

Let $\widehat B:=\widehat B_n$. Every constructed vector is unit timelike. Moreover, if $\{i,j\}\in E$, then
$[u_i,u_j]=a_{ij}\geq1$, so $u_i$ and $u_j$ lie on the same
sheet of the unit hyperboloid. Since $G$ is connected and
$u_{v_1}$ lies on the positive sheet, all the vectors belong to $\Lob(\R^{n-1})$. Thus $\widehat B$ is a nonsingular Lorentz-Gram completion of $A$, and its inverse vanishes on the nonedges of $G$. By previous steps, $\widehat B=B$. Consequently, the Lorentz-Gram matrix produced by the algorithm is independent
of the chosen simplicial elimination ordering.
\end{proof}

\begin{algorithm}[H]
\caption{Orthogonal innovation realization of the canonical chordal completion}
\label{alg:chordal-orthogonal-innovation}
\begin{algorithmic}[1]
\Require A connected chordal graph $G=(V,E)$, where
$|V|=n$, and a partial matrix $A$ satisfying the hypotheses of Theorem~\ref{T:clique-product-completion}
\Ensure Vectors $(u_v)_{v\in V}\subseteq\Lob(\R^{n-1})$
whose Lorentz-Gram matrix is the completion $B$
\State Choose a simplicial elimination ordering
$V=\{v_1,\ldots,v_n\}$
\State Let $\varepsilon_1,\ldots,\varepsilon_{n-1}$ be the
standard basis of $\R^{n-1}$
\State Set $u_{v_1}\gets(1,0)\in\R\oplus\R^{n-1}$
\For{$k=2,\ldots,n$}
    \State $S_k\gets
    N_G(v_k)\cap\{v_1,\ldots,v_{k-1}\}$
    \State $A_k\gets A_{S_k\times S_k}$
    \State $g_k\gets A_{S_k\times\{v_k\}}$
    \State $c_k\gets A_k^{-1}g_k$
    \State $p_k\gets\displaystyle
    \sum_{s\in S_k}(c_k)_s u_s$
    \State $\rho_k\gets
    \sqrt{g_k^tA_k^{-1}g_k-1}$
    \State $e_k\gets(0,\varepsilon_{k-1})$
    \State $u_{v_k}\gets p_k+\rho_ke_k$
\EndFor
\State $B\gets\bigl([u_i,u_j]\bigr)_{i,j\in V}$
\State \Return $(u_v)_{v\in V}$ and $B$
\end{algorithmic}
\end{algorithm}

We conclude this subsection by proving
Theorem~\ref{T:LG-variational-characterization}. The differential of
the log-absolute-determinant, together with the inverse-sparsity
characterization above, identifies the canonical completion as the
unique stationary nonsingular Lorentz-Gram completion. A determinant
inequality for the negative-definite Schur complements arising in
clique attachments then establishes its unique maximality in
absolute determinant.

\begin{proof}[Proof of Theorem~\ref{T:LG-variational-characterization}]
We first prove the stationary characterization. Fix
$X\in\A(A)\cap GL_n$, and let $H=(h_{ij})_{i,j\in V}$ be a symmetric matrix satisfying $h_{ij}=0$ whenever $i=j$ or $\{i,j\}\in E$. Then $X+tH$ retains all prescribed entries and remains nonsingular for sufficiently small $t$. Set $g(t):=\det(X+tH)$. Since $g(0)=\det X\neq0$,
continuity implies that $g(t)$ is nonzero and has constant
sign for sufficiently small $t$. Jacobi's formula gives
\[
g'(t)
=\det(X+tH)\,
\operatorname{Tr}\bigl((X+tH)^{-1}H\bigr).
\]
Consequently,
\[
\begin{aligned}
\left.\frac{d}{dt}\log|\det(X+tH)|\right|_{t=0}
&=\frac{g'(0)}{g(0)}
=\frac{\det X\,\operatorname{Tr}(X^{-1}H)}{\det X}
=\operatorname{Tr}(X^{-1}H).
\end{aligned}
\]
Since $X^{-1}$ and $H$ are symmetric, and $H$ vanishes at
every prescribed position,
\[
\operatorname{Tr}(X^{-1}H)
=
2\sum_{\substack{i<j\\\{i,j\}\notin E}}
(X^{-1})_{ij}h_{ij}.
\]
The entries $h_{ij}$ in this sum may be chosen independently.
Thus the derivative vanishes for every admissible $H$ if and
only if $(X^{-1})_{ij}=0$ whenever $i\neq j$ and $\{i,j\}\notin E$. This proves the equivalence. By Theorem~\ref{T:chordal-canonicity}, $B$ is the unique nonsingular Lorentz-Gram completion of $A$ with this inverse sparsity pattern. Hence it is also the unique such completion that is stationary for $\Phi$.

We now prove the absolute-determinant assertion. We use Fischer's determinant inequality, including its equality condition~\cite{Fischer1908,HornJohnson}: if $K=
\begin{pmatrix}
P&Z\\
Z^t&Q
\end{pmatrix}
\succ0$, then $\det K\leq\det P\,\det Q$, with equality if and only if $Z=0$. Let $X$ be a Lorentz-Gram completion of $A$. If $X$ is singular, then $|\det X|=0<|\det B|$ so suppose that $X$ is nonsingular. Every nonempty principal
submatrix of $X$ is then also nonsingular Lorentz-Gram: a
realizing family for $X$ is linearly independent, and the same
is true of every subfamily. Use the rooted clique-tree and attachment notation from the proof
of Theorem~\ref{T:clique-product-completion}:
\[
W_1=C_1,
\qquad
S_r=C_r\cap C_{\pi(r)},
\qquad
R_r=C_r\setminus S_r,
\qquad
W_r=W_{r-1}\sqcup R_r.
\]
Fix $r\geq2$, and abbreviate $S:=S_r$, $R:=R_r$, and $U:=W_{r-1}\setminus S$. 
Since $X$ agrees with $A$ on the clique $C_r=S\sqcup R$,
ordering the indices of $W_r$ as $S\sqcup U\sqcup R$ gives
\[
X_{W_r\times W_r}
=
\begin{pmatrix}
A_{S\times S} & X_{S\times U} & A_{S\times R}\\
X_{U\times S} & X_{U\times U} & X_{U\times R}\\
A_{R\times S} & X_{R\times U} & A_{R\times R}
\end{pmatrix}.
\]
The block $A_{S\times S}$ is nonsingular. Its Schur
complement, with the remaining indices ordered as
$U\sqcup R$, is therefore
\[
\begin{aligned}
&
\begin{pmatrix}
X_{U\times U} & X_{U\times R}\\
X_{R\times U} & A_{R\times R}
\end{pmatrix}
-
\begin{pmatrix}
X_{U\times S}\\
A_{R\times S}
\end{pmatrix}
A_{S\times S}^{-1}
\begin{pmatrix}
X_{S\times U} & A_{S\times R}
\end{pmatrix}
\\
&\qquad=
\begin{pmatrix}
X_{U\times U}
-X_{U\times S}A_{S\times S}^{-1}X_{S\times U}
&
X_{U\times R}
-X_{U\times S}A_{S\times S}^{-1}A_{S\times R}
\\
X_{R\times U}
-A_{R\times S}A_{S\times S}^{-1}X_{S\times U}
&
A_{R\times R}
-A_{R\times S}A_{S\times S}^{-1}A_{S\times R}
\end{pmatrix}
=:
-\begin{pmatrix}
P&Z\\
Z^t&Q
\end{pmatrix}
=:-K_r.
\end{aligned}
\]
Here the lower-left block equals $-Z^t$ because $X$ and
$A_{S\times S}^{-1}$ are symmetric. Both $X_{W_r\times W_r}$ and $A_{S\times S}$ have exactly
one positive eigenvalue and no zero eigenvalues. Inertia additivity
therefore shows that their Schur complement is negative definite. Hence $K_r\succ0$. The Schur determinant formula gives
\[
\begin{aligned}
|\det X_{W_r\times W_r}|
&=
|\det A_{S\times S}|\det K_r,\\
|\det X_{W_{r-1}\times W_{r-1}}|
&=
|\det A_{S\times S}|\det P,\\
|\det A_{C_r\times C_r}|
&=
|\det A_{S\times S}|\det Q.
\end{aligned}
\]
Applying Fischer's inequality, we obtain
\[
|\det X_{W_r\times W_r}|
\leq
|\det A_{S_r\times S_r}|\,\det P\,\det Q
=
|\det X_{W_{r-1}\times W_{r-1}}|
\frac{|\det A_{C_r\times C_r}|}
     {|\det A_{S_r\times S_r}|}.
\]
Since $X_{W_1\times W_1}=A_{C_1\times C_1}$, iteration yields
\[
|\det X|
\leq
\frac{\displaystyle\prod_{r=1}^{m}
|\det A_{C_r\times C_r}|}
{\displaystyle\prod_{r=2}^{m}
|\det A_{S_r\times S_r}|}
=
|\det B|,
\]
where the last equality is the determinant formula in Theorem~\ref{T:chordal-canonicity}. For $m=1$, the denominator is an empty product and $X=B=A$. Finally, suppose that $|\det X|=|\det B|$. All determinants appearing in the preceding inequalities have positive absolute value, so equality must hold at every attachment. Thus $Z=0$ at each step, which gives $X_{U\times R}=X_{U\times S}A_{S\times S}^{-1}A_{S\times R}$. The same identity on the rows indexed by $S$ holds automatically, because $X_{S\times S}=A_{S\times S}$ and $X_{S\times R}=A_{S\times R}$. Consequently, $X_{W_{r-1}\times R_r}=X_{W_{r-1}\times S_r}A_{S_r\times S_r}^{-1}A_{S_r\times R_r}$. This is exactly the attachment rule defining $B$. Starting from $X_{W_1\times W_1}=A_{C_1\times C_1}=B_{W_1\times W_1}$, induction therefore gives $X_{W_r\times W_r}=B_{W_r\times W_r}$ at every stage, and hence $X=B$. The converse is immediate.
\end{proof}

This completes the algebraic and variational description of the canonical clique-product completion. We now turn to the metric geometry of the resulting hyperbolic embeddings.

\subsection{Comparison with graph metrics}

We now compare the hyperbolic distances induced by the canonical completion with the shortest-path metric of the specification graph. Since the realization preserves every prescribed edge length, the hyperbolic triangle inequality gives $d_{\Lob}(f_A(i),f_A(j))\leq d_G(i,j)$. Controlling the distortion therefore amounts to obtaining a lower bound for the realized distances. We consider two questions: how uniform scaling affects the tree construction, and when a fixed chordal graph admits a distortion bound uniform over all admissible data.

\subsubsection{Scaling limits for trees}

Uniformly enlarging the prescribed edge lengths makes the product-distance realization increasingly faithful to the tree metric. Each edge is realized exactly, and Theorem~\ref{T:asymptotically-isometric-tree-embeddings} shows that the ratio between the realized distance and the scaled tree distance tends to one uniformly over all pairs of distinct vertices. We prove this directly from the path-product formula.

\begin{proof}[Proof of Theorem~\ref{T:asymptotically-isometric-tree-embeddings}] 

We first verify that the completion $B^{(\tau)}$ invoked in the statement is well defined. For every edge $\{i,j\}\in E$, the corresponding maximal clique matrix of $A^{(\tau)}$ is
\begin{align}\label{eq:2by2-scaling}
\begin{pmatrix}
1&\cosh(\tau\lambda_{ij})\\
\cosh(\tau\lambda_{ij})&1
\end{pmatrix}.
\end{align}
Its eigenvalues are $1+\cosh(\tau\lambda_{ij})>0$ and $1-\cosh(\tau\lambda_{ij})<0$, since $\tau>0$ and $\lambda_{ij}>0$.  Hence this matrix has exactly one positive eigenvalue and no zero eigenvalue; since its diagonal entries are $1$ and its off-diagonal entries are at least $1$, Theorem~\ref{T:LG-ker-char} shows that it is a nonsingular Lorentz--Gram matrix. Thus every maximal-clique principal submatrix of $A^{(\tau)}$ is a nonsingular Lorentz--Gram matrix, and Theorem~\ref{T:tree-path-product-completion} applies and produces the completion $B^{(\tau)}$.
% Its eigenvalues are $1+\cosh(\tau\lambda_{ij})>0$ and $1-\cosh(\tau\lambda_{ij})<0$, since $\tau>0$ and $\lambda_{ij}>0$. Thus every maximal clique submatrix of $A^{(\tau)}$ is a nonsingular Lorentz--Gram matrix {\color{blue}by Theorem~\ref{T:LG-ker-char}}. Consequently, Theorem~\ref{T:tree-path-product-completion} applies and produces the completion $B^{(\tau)}$.

By Lemma~\ref{L:LG-realization-isometry}, the distortion is independent of the chosen embedding realization. We may therefore use the realization produced by Algorithm~\ref{alg:tree-orthogonal-innovation}. By Theorem~\ref{T:tree-path-product-completion}, the obtained vectors $u_v^{(\tau)}\in \Lob(\R^{n-1})$, and their Lorentz-Gram matrix is the product-distance completion of $A^{(\tau)}$. Hence with $f_{\tau}(i)=u_i^{(\tau)}$, we have
\[
\bigl[f_{\tau}(i),f_{\tau}(j)\bigr]
=\bigl[u_i^{(\tau)},u_j^{(\tau)}\bigr]
=
\prod_{e\in P(i,j)}a_e^{(\tau)}
=
\prod_{e\in P(i,j)}\cosh(\tau\lambda_e).
\]
Therefore, the hyperboloid distance formula gives 
\[
d_{\Lob}\bigl(f_\tau(i),f_\tau(j)\bigr)= \ach\Big[
\prod_{e\in P(i,j)}\cosh(\tau\lambda_e)
\Big].
\]
Since each $\lambda_e>0$, the product-distance completion of $A^{(\tau)}$ is nonsingular. Therefore by Lemma~\ref{lem:sign-inn} $f_\tau$ is injective, it is a metric embedding of $(V,\tau d_T)$ into $\Lob(\R^{n-1})$. We estimate its distortion. For $x\geq 1$, we have $\ach(x)=\log\left(x+\sqrt{x^2-1}\right)\geq\log x$. Therefore
\begin{align*}
d_\Lob(f_\tau(i),f_\tau(j))
&\geq \sum_{e\in P(i,j)}\log\cosh(\tau\lambda_e)
\geq
\sum_{e\in P(i,j)}
\bigl(\tau\lambda_e-\log2\bigr)\\
&=
\tau d_T(i,j)-|P(i,j)|\log2,
\end{align*}
where we use $\cosh t=\tfrac{1}{2}(e^t+e^{-t})\geq\frac{e^t}{2}$ for $t\geq 0$. Since $d_T(i,j)= \sum_{e\in P(i,j)}\lambda_e\geq |P(i,j)|\lambda_{\min}$, we have $|P(i,j)|\leq \frac{d_T(i,j)}{\lambda_{\min}}$. Therefore, we get the following estimate:
\[
d_\Lob(f_\tau(i),f_\tau(j))
\geq
\left(
1-\frac{\log2}{\tau\lambda_{\min}}
\right)\tau d_T(i,j).
\]
The other estimate $d_\Lob(f_\tau(i),f_\tau(j))\leq\tau d_T(i,j)$ follows from the non-expansive nature of the embedding. Therefore, the distortion is at most $\left(1-\frac{\log2}{\tau\lambda_{\min}}\right)^{-1}$, where the right-hand side tends to $1$ as $\tau\to\infty$. The desired limit follows from this.
\end{proof}

\begin{remark}[Scaling on block graphs and an entrywise preserver obstruction]
The preceding proof extends, for any fixed $\tau>0$, to a block graph -- that is, a chordal graph whose minimal separators are singletons -- provided that all scaled maximal clique matrices remain nonsingular Lorentz--Gram matrices. Indeed, define the scaled partial data by
\[
a_{ii}^{(\tau)}:=1,
\qquad
a_{ij}^{(\tau)}
:=
\cosh\bigl(\tau\ach(a_{ij})\bigr),
\quad \{i,j\}\in E.
\]
Thus scaling the prescribed distances by $\tau$ amounts to applying entrywise the function
\[
\boldsymbol{\phi}_{\tau}(t)
:=
\cosh\bigl(\tau\ach(t)\bigr)
\]
to every maximal clique matrix.

For a tree, every maximal clique is an edge, and \eqref{eq:2by2-scaling} shows that its scaled clique matrix is a nonsingular Lorentz--Gram matrix for every $\tau>0$. Hence the scaled partial data are automatically admissible. For a block graph with larger maximal cliques, this admissibility is no longer automatic and must be verified for the given clique data.

Although every graph under consideration is finite, an unconditional extension to all finite block graphs and all admissible data would amount to an entrywise transform principle in all finite orders: indeed, every complete graph $K_n$ is a block graph. By Theorem~3.2 and Section~3.1, especially equation~(3.4), of \cite{Belton-Guillot-Khare-Putinar}, every entrywise Lorentz--Gram preserver $f$ satisfies
\[
f(t)=O(t)\qquad(t\to\infty).
\]
On the other hand, for $x=\tau \ach t$, we have
\[
\boldsymbol{\phi}_{\tau}(t)
=\cosh \tau x
= \frac{e^{\tau x}+ e^{-\tau x}}{2}
=
\frac{
 \bigl(t+\sqrt{t^2-1}\bigr)^\tau
 +
 \bigl(t+\sqrt{t^2-1}\bigr)^{-\tau}}
 {2}
\sim 2^{\tau-1}t^\tau.
\]
Thus $\boldsymbol{\phi}_{\tau}$ does not preserve Lorentz--Gram matrices of all finite orders when $\tau>1$. Consequently, the tree hypothesis makes the scaled data in Theorem~\ref{T:asymptotically-isometric-tree-embeddings} automatically admissible, whereas its extension to a fixed block graph is conditional on the scaled maximal clique matrices remaining admissible for all sufficiently large $\tau$.
\end{remark}

\subsubsection{Singleton versus nonsingleton separators}

We next determine how the separator structure governs distortion over all admissible data on a fixed chordal graph. When every separator is a singleton, the clique-tree identifies a sequence of vertices through which every connecting path must pass. The graph distance is a sum of successive lengths, whereas the canonical completion gives a product of their hyperbolic cosines. Comparing these expressions yields the sharp bound in terms of $\diam(G)$.

A separator with at least two vertices permits a different geometric configuration. Two vertices on opposite sides can approach a common point in the hyperbolic span of the separator while remaining uniformly separated from every separator vertex. We construct admissible data for which their distance in the canonical realization tends to zero, although every graph path between them has length bounded away from zero.

\begin{proof}[Proof of Theorem~\ref{T:separator-size-distortion-dichotomy}]

We begin by proving the first case.
     
\medskip
\textit{\textbf{1.} Clique-tree separators are singletons.}
\medskip

\noindent Observe that every prescribed edge length is strictly positive. Indeed, if $\{i,j\}\in E$, then $i,j$ belong to a maximal clique $C$. A realization of $A_{C\times C}$ is linearly independent by Lemma~\ref{lem:sign-inn}, so its vectors corresponding to $i,j$ are distinct. Hence $a_{ij}>1$ and $\lambda_{ij}=\ach(a_{ij})>0$.

First, for distinct vertices $i,j\in V$, we will show that for all admissible $A$,
\begin{align}\label{eq:sep-size-1-claim1}
\frac{d_G(i,j)}{d_{\Lob}(f_A(i),f_A(j))} \leq \diam (G)^{1/2}.
\end{align}

For distinct vertices $i,j\in V$ choose maximal cliques $C_0,C_m$ containing $i,j$, respectively, and let $C_0,C_1,\ldots,C_m$ be the unique path joining them in $\ct$. 

First, let us rule out the case $C_0=C_m$. Then $\{i,j\}\in E$. Since $f_A$ preserves the given edge length, $d_G(i,j)=d_{\Lob}(f_A(i),f_A(j))$. Therefore, \eqref{eq:sep-size-1-claim1} holds.

We therefore assume $C_0\neq C_m$, and write
\[
C_{r-1}\cap C_r=\{s_r\}
\qquad(1\leq r\leq m),
\]
and set $x_0:=i$, $x_r:=s_r$ for $1\leq r\leq m$, and $x_{m+1}:=j$. Now, delete consecutive repetitions from $(x_k)_{k=0}^{m+1}$ to obtain $y_0=i,y_1,\ldots,y_q=j$. Then the following conclusions hold.

\medskip
\emph{$\bullet$ Consecutive distinct vertices are adjacent.} Before repetitions are deleted, we have
\[
i,s_1\in C_0,\qquad
s_r,s_{r+1}\in C_r\quad(1\leq r<m),\qquad
s_m,j\in C_m.
\]
Hence every two consecutive distinct terms of the reduced sequence belong to a common maximal clique and are therefore adjacent in $G$.

\medskip
\emph{$\bullet$ The reduced vertices are all distinct.} Suppose $s_r=s_t=v$ for some $r<t$. Since $v$ belongs to $C_r$ and $C_{t-1}$, the clique intersection property (or the induced subtree property) in Theorem~\ref{T:clique-trees} implies that every clique between them in $\ct$ contains $v$. Therefore
\[
s_r=s_{r+1}=\cdots=s_t=v.
\]
Similarly, if $i=s_t$, then $i=s_1=\cdots=s_t$, and if $s_r=j$, then $s_r=\cdots=s_m=j$. Thus every repetition in $x_0,\ldots,x_{m+1}$ is consecutive, and the reduced vertices $y_0,\ldots,y_q$ are pairwise distinct.

\medskip
\emph{$\bullet$ The internal vertices are separators.} Clearly, each internal vertex $y_k$ is equal to some $s_{r}$. Delete the clique-tree edge $C_{r-1}C_{r}$. Let $\ct_r^{-}$ and $\ct_r^{+}$ be the two resulting components containing $C_0$ and $C_m$ in $\ct$, respectively. Let $V_r^{-}$ and $V_r^{+}$ be the unions of the cliques belonging to these components. The separator property in Theorem~\ref{T:clique-trees} gives
\[
V_r^{-}\cap V_r^{+}
=C_{r-1}\cap C_{r}
=\{y_k\}.
\]
Moreover, no edge of $G$ joins $V_r^{-}\setminus\{y_k\}$ to $V_r^{+}\setminus\{y_k\}$: every edge of $G$ is contained in a maximal clique, and every maximal clique belongs to exactly one of the two components of the deleted clique tree. Hence $y_k$ separates $i$ from $j$ in $G$, i.e., every $i$--$j$ path must pass through $y_k=s_{r}$.

\medskip
\emph{$\bullet$ The separators occur in order.} Let $y_k=s_r$ and $y_{k+1}=s_t$, where $r<t$. Then $s_t$ lies on the $C_m$-side of the cut $C_{r-1}C_r$, whereas $i$ lies on its $C_0$-side. Since the only common vertex of the two sides is $s_r$, every path from $i$ to $s_t$ must first pass through $s_r$. Thus the first occurrences of $y_1,y_2,\ldots,y_{q-1}$ along every $i$--$j$ path appear in this order. In particular, this holds for every shortest $i$--$j$ path, which is simple because all edge lengths are positive.

% Thus every $i$--$j$ path encounters $y_1,y_2,\ldots,y_{q-1}$ in this order.

\medskip
\emph{$\bullet$ Finally.} Every $i$--$j$ path has at least $q$ edges, while $y_0,y_1,\ldots,y_q$ is itself an $i$--$j$ path with $q$ edges. Therefore $q=d_G^{\mathrm{unw}}(i,j)\leq\diam(G)$, where $d_G^{\mathrm{unw}}$ is the graph metric for $G$ with all unit edge weights.\medskip

\begin{figure}[htbp]
\centering
\includegraphics[width=1\textwidth]{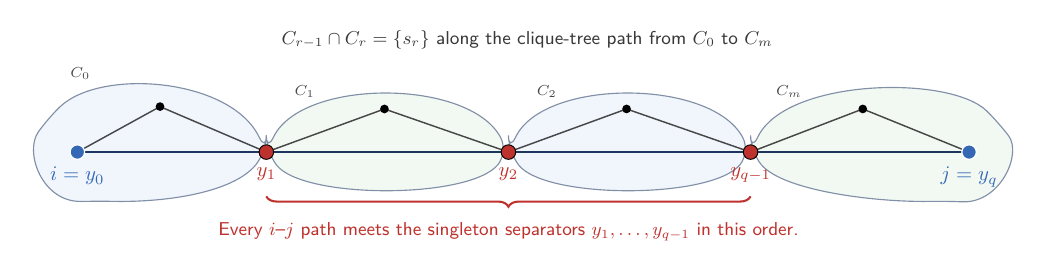}
\vspace{-2.4em}
\caption{Chordal graph structure with all singleton separators.}
\label{F:singleton-graph-structure}
\end{figure}

$\bullet$ Retain the reduced path $y_0=i,y_1,\ldots,y_q=j$ constructed above, and define
\[
\ell_r:=\ach(a_{y_ry_{r+1}}),
\qquad 0\leq r<q.
\]
Since $y_r$ and $y_{r+1}$ are distinct adjacent vertices, $\ell_r>0$. Moreover, $f_A$ preserves their prescribed edge length, and hence $d_{\Lob}\bigl(f_A(y_r),f_A(y_{r+1})\bigr)=\ell_r$. The edge $\{y_r,y_{r+1}\}$ gives $d_G(y_r,y_{r+1})\leq\ell_r$, whereas the nonexpansiveness of $f_A$ gives $\ell_r=d_{\Lob}\bigl(f_A(y_r),f_A(y_{r+1})\bigr)\leq d_G(y_r,y_{r+1})$. Therefore
\[
d_G(y_r,y_{r+1})=\ell_r.
\]

$\bullet$ We have shown that every $i$--$j$ path encounters $y_0=i,y_1,\ldots,y_q=j$ in this order. Its portion between $y_r$ and $y_{r+1}$ consequently has length at least $\ell_r$, and hence every such path has length at least $\sum_{r=0}^{q-1}\ell_r$. Conversely, $y_0,y_1,\ldots,y_q$ is itself an $i$--$j$ path, and its length is precisely this sum. Thus
\[
d_G(i,j)=\sum_{r=0}^{q-1}\ell_r.
\]

$\bullet$ Since every separator is a singleton, its Lorentz-Gram matrix is $(1)_{1\times1}$, and Theorem~\ref{T:clique-product-completion} gives $b_{ij}=\prod_{r=0}^{m}a_{x_rx_{r+1}}$. Every factor corresponding to a repetition $x_r=x_{r+1}$ is the
diagonal entry $a_{x_rx_r}=1$. Deleting these factors therefore yields
\[
b_{ij}
=
\prod_{r=0}^{q-1}a_{y_ry_{r+1}}
=
\prod_{r=0}^{q-1}\cosh\ell_r.
\]
Consequently, $d_{\Lob}\bigl(f_A(i),f_A(j)\bigr)=\ach\left(\prod_{r=0}^{q-1}\cosh\ell_r\right)$.

\begin{figure}[htbp]
\centering
\includegraphics[width=1\textwidth]{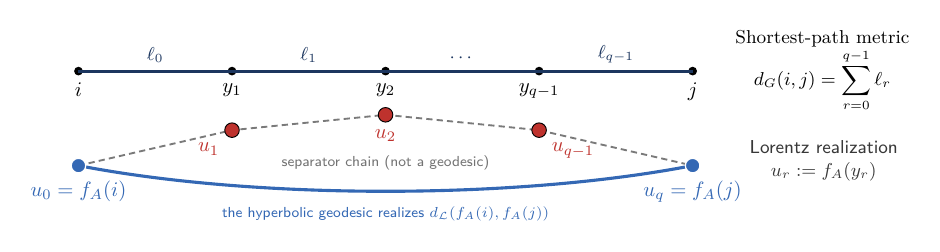}
\vspace{-2.4em}
\caption{The shortest-path metric and realization.}
\label{F:reduced-path-canonical-realization}
\end{figure}

$\bullet$ We prove the elementary inequality:
\[
\prod_{r=0}^{q-1}\cosh\ell_r
\geq
\cosh \sqrt{\sum_{r=0}^{q-1}\ell_r^2} .
\]
Define $F(x):=\log\cosh\sqrt{x}$, $x\geq0$. Then $F(0)=0$, and, for $x>0$, $F'(x)=\frac{\tanh\sqrt{x}}{2\sqrt{x}}$. The function $t\mapsto \tanh(t)/t$ is decreasing on $(0,\infty)$, since
\[
\frac{d}{dt}\left(\frac{\tanh t}{t}\right)
=
\frac{t\sech^2t-\tanh t}{t^2}
=
\frac{t-\sinh t\cosh t}{t^2\cosh^2t}
\leq0,
\]
since $\sinh t \geq t$ and $\cosh t \geq 1$. Hence $F'$ is decreasing, and therefore $F$ is concave. We claim that $F$ is subadditive. Indeed, for $x,y\geq0$, with
$x+y>0$, concavity and $F(0)=0$ give
\[
F(x)
=
F\left(\frac{x}{x+y}(x+y)+\frac{y}{x+y}\,0\right)
\geq
\frac{x}{x+y}F(x+y),
\]
and similarly, $F(y)\geq\frac{y}{x+y}F(x+y)$. Adding these inequalities yields $F(x+y)\leq F(x)+F(y)$. Applying subadditivity to $\ell_0^2,\ldots,\ell_{q-1}^2$ and exponentiating, we obtain
\[
\cosh \sqrt{\sum_{r=0}^{q-1}\ell_r^2} 
=\exp F\left(\sum_{r=0}^{q-1}\ell_r^2\right) 
\leq\exp \sum_{r=0}^{q-1}F(\ell_r^2) 
=\prod_{r=0}^{q-1}\cosh\ell_r.
\]
Thus, the desired inequality follows.\medskip

$\bullet$ Since $\ach$ is increasing, it follows that
\[
d_{\Lob}\bigl(f_A(i),f_A(j)\bigr)
=
\ach\left(\prod_{r=0}^{q-1}\cosh\ell_r\right)
\geq
\sqrt{\sum_{r=0}^{q-1}\ell_r^2}.
\]
Using this and Cauchy--Schwarz,
\[
\frac{d_G(i,j)}
     {d_{\Lob}(f_A(i),f_A(j))}
\leq
\frac{\sum_{r=0}^{q-1}\ell_r}
     {\sqrt{\sum_{r=0}^{q-1}\ell_r^2}}
\leq q^{1/2}
\leq \diam(G)^{1/2}.
\]
Taking the maximum gives $\dist(f_A)\leq \diam(G)^{1/2}$. \medskip

$\bullet$ To prove sharpness, put $d:=\diam(G)$ and choose vertices $i,j$ whose unweighted distance is $d$. For $\varepsilon>0$, prescribe $a_{ii}^{(\varepsilon)}:=1$ for every $i\in V$ and $a_{ij}^{(\varepsilon)}:=\cosh\varepsilon$ for every $\{i,j\}\in E$. On a clique of size $p$, the resulting matrix is
\[
(1-\cosh\varepsilon)I_p
+\cosh\varepsilon\,\mathbf1_p\mathbf1_p^t.
\]
Its eigenvalues are $1+(p-1)\cosh\varepsilon>0$ and, when $p\geq2$, $1-\cosh\varepsilon<0$ with multiplicity $p-1$. Hence every maximal-clique principal submatrix of $A^{(\varepsilon)}$ has exactly one positive eigenvalue and no zero eigenvalue.  Since its diagonal entries are equal to $1$ and its off-diagonal entries are at least $1$, Theorem~\ref{T:LG-ker-char} shows that it is a nonsingular Lorentz--Gram matrix.  Consequently $A^{(\varepsilon)}$ is admissible. The preceding path identities give
\[
d_G(i,j)=d\varepsilon,
\qquad
b_{ij}^{(\varepsilon)}=(\cosh\varepsilon)^d.
\]
Consequently, $d_{\Lob}\bigl(f_{A^{(\varepsilon)}}(i),f_{A^{(\varepsilon)}}(j)\bigr)=\ach\bigl((\cosh\varepsilon)^d\bigr)$. As $\varepsilon\downarrow0$, we have $(\cosh\varepsilon)^d=1+\frac{d}{2}\varepsilon^2+O(\varepsilon^4)$, and, using $\ach (1+t)=\sqrt{2t}+O(t^{3/2})$ as $t\downarrow0$, we obtain $\ach \bigl((\cosh\varepsilon)^d\bigr)=\sqrt d\,\varepsilon+O(\varepsilon^3)$. It follows that as $\varepsilon \downarrow 0$,
\[
\dist\bigl(f_{A^{(\varepsilon)}}\bigr)
\geq
\frac{d_G(i,j)}{d_{\Lob}\bigl(f_{A^{(\varepsilon)}}(i), f_{A^{(\varepsilon)}}(j)\bigr)}
=
\frac{d\varepsilon}
     {\sqrt d\,\varepsilon+O(\varepsilon^3)}
\longrightarrow \sqrt d.
\]
This yields the desired sharpness.

\medskip
\textit{\textbf{2.} A clique-tree separator has size at least two.}
\medskip

Suppose that $CD\in E(\ct)$ and $S:=C\cap D$, $|S|\geq2$. Note, $C\setminus S\ne\varnothing\ne D\setminus S$ because $C$ and $D$ are distinct maximal cliques. Choose $x\in C\setminus S$ and $y\in D\setminus S$. Deleting $CD$ from $\ct$ gives two components whose unions of cliques meet precisely in $S$. Hence every path in $G$ from $x$ to $y$ meets $S$.

\begin{figure}[htbp]
\centering
\includegraphics[width=1\textwidth]{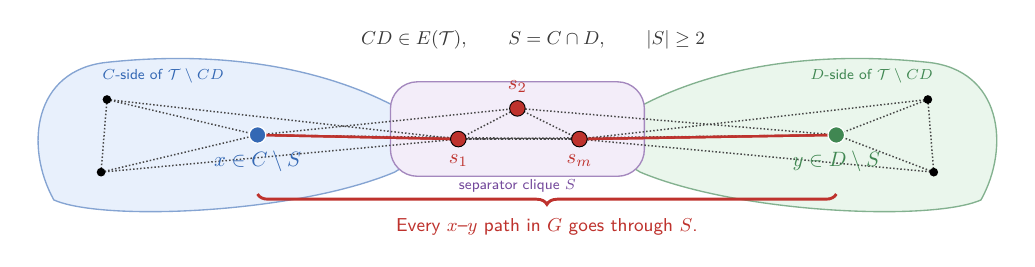}
\vspace{-2.4em}
\caption{Graph separation in the case of a bigger separator.}
\label{F:bigger-separator-path}
\end{figure}

The idea is to construct an admissible $A^{(\varepsilon)}$ on the given graph $G$, for $\varepsilon>0$, such that the corresponding distortion tends to infinity as $\varepsilon\downarrow0$. Specifically, we show that the shortest-path distance $d_G(x,y)$ remains away from zero uniformly in $\varepsilon$, while $d_{\Lob}\bigl(f_{A^{(\varepsilon)}}(x),f_{A^{(\varepsilon)}}(y)\bigr) \to 0$ as $\varepsilon\downarrow0$.

We begin by considering the vertices in $S$. Choose linearly independent vectors $(u_s)_{s\in S}\subseteq\Lob(\R^{|S|-1})$, and put $W:=\Span\{u_s:s\in S\}$. Since $|S|\geq2$, the set $W\cap\Lob(\R^{|S|-1})$ contains infinitely many points. Indeed, the $|S|$ linearly independent separator vectors span $\R\oplus\R^{|S|-1}$, whose positive unit hyperboloid contains the distinct points $(\cosh t,\sinh t,0,\ldots,0)$, $t>0$, with $|S|$ coordinates. Choose $w\in W\cap\Lob(\R^{|S|-1})$ distinct from every $u_s$, $s\in S$. (In contrast, if $S=\{s\}$, then $W=\R u_s$ and $W\cap\Lob(\R^0)=\{u_s\}$, so no choice of $w\neq u_s$ would be possible.) As $w\in W$, there is a column vector $c=(c_s)_{s\in S}$ such that
\[
w=\sum_{s\in S}c_su_s,
\qquad\text{and we set}\qquad
\delta:=\min_{s\in S}d_{\Lob}(w,u_s)>0.
\]
% Embed $W$ in $\R\oplus\R^{|V|-1}$ by adjoining zero spatial coordinates. Choose the vectors $(e_v)_{v\in V\setminus S}$ along the newly adjoined spatial coordinate directions, so that
% \[
% [e_v,e_z]=-\delta_{vz},
% \qquad
% [e_v,u_s]=0
% \quad(v,z\in V\setminus S,\ s\in S).
% \]
Since the $|S|$ vectors $(u_s)_{s\in S}$ are linearly independent, $\dim W=|S|$ and $W=\mathbb R\oplus\mathbb R^{|S|-1}$, on which the Lorentz form has signature $(1,|S|-1)$ by Lemma~\ref{lem:sign-inn}. Embed $W$ isometrically into $\mathbb R\oplus\mathbb R^{|V|-1}$ as $\mathbb R\oplus\mathbb R^{|S|-1}\oplus\{0\}$, by adjoining zero spatial coordinates.  Its orthogonal complement $W^{\perp}=\{0\}\oplus\{0\}\oplus\mathbb R^{|V|-|S|}$ is negative definite of dimension
\[
  (|V|-1)-(|S|-1)=|V|-|S|=|V\setminus S| ,
\]
so the newly adjoined coordinate directions supply exactly as many pairwise orthogonal negative unit vectors as there are vertices outside $S$. Let $(e_v)_{v\in V\setminus S}$ be the corresponding coordinate vectors of $W^{\perp}$, so that
\[
  [e_v,e_z]=-\delta_{vz},\qquad [e_v,u_s]=0
  \qquad (v,z\in V\setminus S,\ s\in S).
\] 
For $\varepsilon>0$, define
\[
u_s^{(\varepsilon)}:=u_s
\quad(s\in S),
\qquad
u_v^{(\varepsilon)}
:=\cosh(\varepsilon)w+\sinh(\varepsilon)e_v
\quad(v\in V\setminus S).
\]
For every $v\notin S$, we have $[u_v^{(\varepsilon)},u_v^{(\varepsilon)}]=\cosh^2(\varepsilon)-\sinh^2(\varepsilon)=1$. Moreover, $e_v$ has zero time coordinate, so the time coordinate of $u_v^{(\varepsilon)}$ is $\cosh(\varepsilon)w_0>0$. Thus all the vectors belong to $\Lob(\R^{|V|-1})$. They are also linearly independent: taking the Lorentz inner product of a linear relation with each $e_v$ first forces all coefficients indexed by $V\setminus S$ to vanish, since $\sinh(\varepsilon)>0$; the remaining coefficients vanish by the linear independence of $(u_s)_{s\in S}$. Form the full Lorentz-Gram matrix $M^{(\varepsilon)} :=\bigl([u_i^{(\varepsilon)},u_j^{(\varepsilon)}]\bigr)_{i,j\in V}$, and let $A^{(\varepsilon)}$ be its restriction to the diagonal and the edges of $G$:
\[
a_{ij}^{(\varepsilon)}
:=[u_i^{(\varepsilon)},u_j^{(\varepsilon)}]
\qquad
\text{if }i=j\text{ or }\{i,j\}\in E.
\]
Every clique is represented by a linearly independent family of positive-sheet unit timelike vectors, so its principal submatrix is a nonsingular Lorentz-Gram matrix. Consequently, $A^{(\varepsilon)}$ is admissible with respect to $G$. Let $B^{(\varepsilon)}$ denote its clique-product completion. We next compute $b_{xy}^{(\varepsilon)}$. Set $A_S:=\bigl([u_s,u_t]\bigr)_{s,t\in S}$. Since $w=\sum_{s\in S}c_su_s$ and $[w,w]=1$, we have
\[
c^tA_Sc
=\Big[\sum_{s\in S}c_su_s,\sum_{t\in S}c_tu_t\Big]
=[w,w]=1.
\]
Moreover, for every $s\in S$ and $v\in\{x,y\}$, orthogonality gives
\[
a_{sv}^{(\varepsilon)}
=[u_s,\cosh(\varepsilon)w+\sinh(\varepsilon)e_v]
=\cosh(\varepsilon)\sum_{t\in S}[u_s,u_t]c_t
=\cosh(\varepsilon)(A_Sc)_s.
\]
Thus, by symmetry we have $A_{\{x\}\times S}^{(\varepsilon)}=\cosh(\varepsilon)c^tA_S$ and $A_{S\times\{y\}}^{(\varepsilon)}=\cosh(\varepsilon)A_Sc$. The cliques $C$ and $D$ are adjacent in $\ct$, with the separator $C\cap D=S$. Hence the clique-product formula in Theorem~\ref{T:clique-product-completion} yields
\[
\begin{aligned}
b_{xy}^{(\varepsilon)}
=A_{\{x\}\times S}^{(\varepsilon)}
  A_S^{-1}A_{S\times\{y\}}^{(\varepsilon)}
=\cosh^2(\varepsilon)c^tA_Sc
=\cosh^2(\varepsilon).
\end{aligned}
\]
Therefore, $d_{\Lob}\bigl(f_{A^{(\varepsilon)}}(x),
             f_{A^{(\varepsilon)}}(y)\bigr)
=\operatorname{arcosh}\bigl(\cosh^2(\varepsilon)\bigr)
\to0$ as $\varepsilon\downarrow0$. It remains to bound the shortest-path distance from below.
For every $s\in S$, the constructed vectors satisfy
\[
\begin{aligned}
d_{\Lob}(u_x^{(\varepsilon)},u_s)
=d_{\Lob}(u_y^{(\varepsilon)},u_s)
&=\operatorname{arcosh}\bigl(
  \cosh(\varepsilon)\cosh d_{\Lob}(w,u_s)\bigr)
\geq d_{\Lob}(w,u_s)\ge\delta.
\end{aligned}
\]
Here we used $\cosh\varepsilon\geq1$, the monotonicity of $\ach$, and the definition of $\delta$. 

\begin{figure}[htbp]
\centering
\includegraphics[width=1\textwidth]{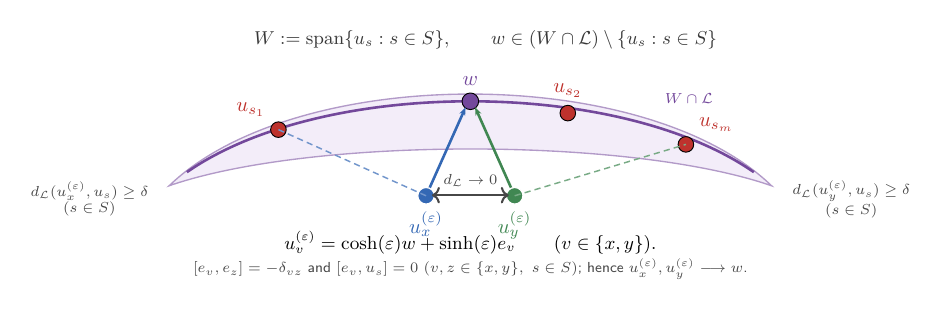}
\vspace{-2.4em}
\caption{Lorentz construction.}
\label{F:bigger-separator-lorentz-construction}
\end{figure}

Let $P$ be any path in $G$ from $x$ to $y$, and let $s,s'$ be its first and last vertices in $S$, respectively. Every edge of $P$ has prescribed length equal to the hyperbolic distance between its corresponding constructed vectors. Applying the hyperbolic triangle inequality to the portion from $x$ to $s$ and the portion from $s'$ to $y$, we obtain
\[
\operatorname{length}(P)
\ge d_{\Lob}(u_x^{(\varepsilon)},u_s)
   +d_{\Lob}(u_{s'},u_y^{(\varepsilon)})
\ge2\delta.
\]
Taking the minimum over all such paths gives $d_G(x,y)\ge2\delta$. Since the canonical realization preserves every prescribed edge length, we conclude that
\[
\dist\bigl(f_{A^{(\varepsilon)}}\bigr)
\geq
\frac{d_G(x,y)}{d_{\Lob}\bigl(f_{A^{(\varepsilon)}}(x),f_{A^{(\varepsilon)}}(y)\bigr)}
\geq
\frac{2\delta}{\operatorname{arcosh}\bigl(\cosh^2(\varepsilon)\bigr)}
\longrightarrow\infty
\qquad\text{as }\varepsilon\downarrow0,
\]
as required. This completes the proof.
\end{proof}

\section{Example}
\label{S:five-vertex-example}

We give a five-vertex example that makes the clique-product formula, inverse sparsity, the determinant formula, Algorithm~\ref{alg:chordal-orthogonal-innovation}, and the second case of Theorem~\ref{T:separator-size-distortion-dichotomy} explicit.

\medskip
\textit{\textbf{1.} Graph construction.}
\medskip

\noindent Let $G=(V,E)$, where $V=\{1,2,3,4,5\}$, be the chordal graph whose maximal cliques are
\[
 C_1:=\{1,2,3\},\qquad
 C_2:=\{2,3,4\},\qquad
 C_3:=\{2,3,5\}.
\]
Thus $G$ consists of three triangles with the common edge $\{2,3\}$. We choose the clique-tree $C_1-C_2-C_3$. Both of its edges have the same separator
\[
 S:=\{2,3\},\qquad
 C_1\cap C_2=C_2\cap C_3=S,
\]
which has cardinality two. The graph and the chosen clique-tree are displayed in Figure~\ref{F:five-vertex-example}.

\begin{figure}[htbp]
\centering
\vspace{1em}

\definecolor{separatorcolor}{HTML}{C65D2E}
\definecolor{cliqueone}{HTML}{3B6FB6}
\definecolor{cliquetwo}{HTML}{2A8C82}
\definecolor{cliquethree}{HTML}{7A5AA6}

\begin{tikzpicture}[
  scale=0.9,
  every node/.style={font=\small},
  vertex/.style={
    circle,
    draw,
    line width=0.8pt,
    minimum size=6mm,
    inner sep=0pt
  },
  cliquenode/.style={
    draw,
    rounded corners=2pt,
    line width=0.8pt,
    minimum width=10mm,
    minimum height=7mm,
    inner sep=2pt
  }
]

% Graph vertices
\node[
  vertex,
  draw=separatorcolor,
  fill=separatorcolor!15
] (2) at (0,0.8) {$2$};

\node[
  vertex,
  draw=separatorcolor,
  fill=separatorcolor!15
] (3) at (0,-0.8) {$3$};

\node[
  vertex,
  draw=cliqueone,
  fill=cliqueone!12
] (1) at (-1.8,0) {$1$};

\node[
  vertex,
  draw=cliquetwo,
  fill=cliquetwo!12
] (4) at (1.8,0.9) {$4$};

\node[
  vertex,
  draw=cliquethree,
  fill=cliquethree!12
] (5) at (1.8,-0.9) {$5$};

% Common separator edge
\draw[
  separatorcolor,
  very thick
] (2)--(3);

% Edges belonging to the three maximal cliques
\draw[
  cliqueone,
  line width=0.9pt
] (1)--(2) (1)--(3);

\draw[
  cliquetwo,
  line width=0.9pt
] (4)--(2) (4)--(3);

\draw[
  cliquethree,
  line width=0.9pt
] (5)--(2) (5)--(3);

\node[font=\small\bfseries] at (0,-1.55) {$G$};

% Clique tree
\node[
  cliquenode,
  draw=cliqueone,
  fill=cliqueone!12
] (C1) at (5.0,0) {$C_1$};

\node[
  cliquenode,
  draw=cliquetwo,
  fill=cliquetwo!12
] (C2) at (6.8,0) {$C_2$};

\node[
  cliquenode,
  draw=cliquethree,
  fill=cliquethree!12
] (C3) at (8.6,0) {$C_3$};

\draw[
  separatorcolor,
  line width=1pt
]
(C1)--node[
  above,
  text=separatorcolor,
  font=\small\bfseries
] {$S$} (C2);

\draw[
  separatorcolor,
  line width=1pt
]
(C2)--node[
  above,
  text=separatorcolor,
  font=\small\bfseries
] {$S$} (C3);

\node[font=\small\bfseries] at (6.8,-1.55)
{clique-tree $\mathcal T$};

\end{tikzpicture}\caption{The five-vertex chordal graph $G$ and the clique-tree
$\mathcal T$. Both clique-tree edges have the size-two separator
$S=\{2,3\}$.}
\label{F:five-vertex-example}
\end{figure}
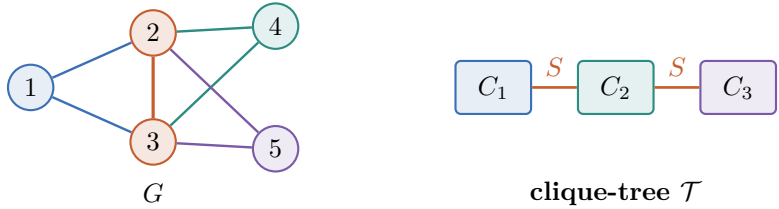

\medskip
\textit{\textbf{2.} Admissible data.}
\medskip

\noindent Fix $\tau>0$, and abbreviate $t:=\cosh\tau$ and $s:=\sinh\tau$.
% Define a partial symmetric matrix $A^{(\tau)}$, with specification
% graph $G$, by
% \[
%  a_{ii}^{(\tau)}:=1,\qquad
%  a_{23}^{(\tau)}:=\frac53,\qquad
%  a_{i2}^{(\tau)}=a_{i3}^{(\tau)}:=q
%  \quad (i\in\{1,4,5\}).
% \]
% All other off-diagonal entries are unspecified. 
Define the partial symmetric matrix
\[
A^{(\tau)}_{5\times 5} = (a_{ij}^{(\tau)})
:=
\begin{pmatrix}
1 & q & q & \ast & \ast \\[1mm] 
q & 1 & 5/3 & q & q \\[1mm]
q & 5/3 & 1 & q & q \\[1mm]
\ast & q & q & 1 & \ast \\[1mm]
\ast & q & q & \ast & 1
\end{pmatrix},
\qquad
q:=\frac{2\cosh\tau}{\sqrt3}.
\]
The symbol $\ast$ denotes an unspecified entry. Its specification graph is $G$. With the indices in the order $i,2,3$, every maximal clique submatrix has the form
\[
 K_\tau=
 \begin{pmatrix}
 1&q&q\\[1mm]
 q&1&\frac53\\[1mm]
 q&\frac53&1
 \end{pmatrix},
 \qquad i\in\{1,4,5\}.
\]
Indeed, put
\[
 A_S:=A^{(\tau)}_{S\times S}
 =\begin{pmatrix}1&\frac53\\[1mm]\frac53&1\end{pmatrix},
 \qquad
 A_S^{-1}
 =\begin{pmatrix}
 -\frac9{16}&\frac{15}{16}\\[1mm]
 \frac{15}{16}&-\frac9{16}
 \end{pmatrix},
 \qquad
 g:=\binom q q.
\]
The matrix $A_S$ has inertia $(1,1,0)$, and
\[
g^tA_S^{-1}g
=
(q,q)
\begin{pmatrix}
-\frac9{16}&\frac{15}{16}\\[2pt]
\frac{15}{16}&-\frac9{16}
\end{pmatrix}
\binom q q
=
\frac{3q^2}{4}
=
t^2.
\]
Therefore, the Schur complement of $A_S$ in $K_\tau$ is
\[
1-g^tA_S^{-1}g=1-t^2=-s^2<0.
\]
Since $q=2\cosh(\tau)/\sqrt3>1$, every entry of $K_\tau$ is at least $1$. Consequently, $K_\tau$ has inertia $(1,2,0)$, and hence is a nonsingular Lorentz-Gram matrix.

\medskip
\textit{\textbf{3.} Clique-product completion.}
\medskip

\noindent We next compute its clique-product completion. The missing entry between $1\in C_1\setminus S$ and $5\in C_3\setminus S$ uses the entire clique-tree path $C_1,C_2,C_3$. Since its two separators are both equal to $S$, Theorem~\ref{T:clique-product-completion} gives
\begin{align*}
 b_{15}^{(\tau)}
 &=A^{(\tau)}_{\{1\}\times S}A_S^{-1}
   A^{(\tau)}_{S\times S}A_S^{-1}
   A^{(\tau)}_{S\times\{5\}}
=g^tA_S^{-1}A_SA_S^{-1}g
  =g^tA_S^{-1}g=t^2.
\end{align*}
For the missing entry between $1\in C_1\setminus S$ and
$4\in C_2\setminus S$, the one-edge transfer formula gives
\[
b_{14}^{(\tau)}
=
A^{(\tau)}_{\{1\}\times S}A_S^{-1}
A^{(\tau)}_{S\times\{4\}}
=
g^tA_S^{-1}g
=
t^2.
\]
Likewise, since $4\in C_2\setminus S$ and
$5\in C_3\setminus S$,
\[
b_{45}^{(\tau)}
=
A^{(\tau)}_{\{4\}\times S}A_S^{-1}
A^{(\tau)}_{S\times\{5\}}
=
g^tA_S^{-1}g
=
t^2.
\]
Therefore, in the natural vertex order $1,2,3,4,5$, the canonical completion is
\[
 B^{(\tau)}=
 \begin{pmatrix}
 1&q&q&t^2&t^2\\[1mm]
 q&1& 5/3 &q&q\\[1mm]
 q& 5/3 &1&q&q\\[1mm]
 t^2&q&q&1&t^2\\[1mm]
 t^2&q&q&t^2&1
 \end{pmatrix}.
\]

\medskip
\textit{\textbf{4.} Inverse and determinant of the completion.}
\medskip

\noindent The inverse can also be written explicitly. Set
\[
 \alpha:=\frac{\sqrt3\,t}{4s^2},\qquad
 \beta:=-\frac9{16}-\frac{9t^2}{16s^2},\qquad
 \gamma:=\frac{15}{16}-\frac{9t^2}{16s^2}.
\]
Then
\[
 \bigl(B^{(\tau)}\bigr)^{-1}=
 \begin{pmatrix}
 -s^{-2}&\alpha&\alpha& {0}&{0}\\
 \alpha&\beta&\gamma&\alpha&\alpha\\
 \alpha&\gamma&\beta&\alpha&\alpha\\
 {0}&\alpha&\alpha&-s^{-2}&{0}\\
 {0}&\alpha&\alpha&{0}&-s^{-2}
 \end{pmatrix}.
\]
The displayed off-diagonal zeros occur exactly at the nonedges of $G$, namely, $\{1,4\}$, $\{1,5\}$, $\{4,5\}$. Thus the inverse-sparsity conclusion of Theorem~\ref{T:chordal-canonicity} is visible in this example.

% \medskip
% \textit{\textbf{5.} Determinant.}
% \medskip

\noindent For completeness, the determinant may be obtained either from the displayed matrix or from the clique-separator formula. We have
\[
 \det A_S=-\frac{16}{9},
 \qquad
 \det K_\tau
 =\det A_S\,(1-g^tA_S^{-1}g)
 =\frac{16}{9}s^2.
\]
There are three maximal cliques and two clique-tree edges, both carrying
the separator $S$. Hence $S$ occurs with multiplicity two, and
\begin{align*}
 \det B^{(\tau)}
 &=\frac{\prod_{r=1}^3\det A^{(\tau)}_{C_r\times C_r}}
         {\prod_{C_rC_{r+1}\in E(\mathcal T)}
          \det A^{(\tau)}_{(C_r\cap C_{r+1})\times
                           (C_r\cap C_{r+1})}}
 =\frac{\left(\frac{16}{9}s^2\right)^3}
         {\left(-\frac{16}{9}\right)^2}
  =\frac{16}{9}s^6.
\end{align*}

\medskip
\textit{\textbf{5.} Lorentz-Gram realization via simplicial elimination and orthogonal innovation.}
\medskip

% \noindent We now carry out Algorithm~\ref{alg:chordal-orthogonal-innovation}. Use the simplicial elimination ordering
% \[
% v_1=2,\qquad v_2=3,\qquad v_3=1,\qquad
% v_4=4,\qquad v_5=5.
% \]

\noindent We now carry out Algorithm~\ref{alg:chordal-orthogonal-innovation}. Use the simplicial elimination ordering as in Figure~\ref{F:five-vertex-elimination-ordering}:
\[
v_1=2,\qquad v_2=3,\qquad v_3=1,\qquad
v_4=4,\qquad v_5=5.
\]
The corresponding earlier-neighbour cliques $S_k:=N_G(v_k)\cap\{v_1,\ldots,v_{k-1}\}$, $k=2,\ldots,5$, are
\[
S_2=\{2\},
\qquad
S_3=S_4=S_5=\{2,3\}.
\]

\begin{figure}[htbp]
\centering
\vspace{0.5em}

\definecolor{separatorcolor}{HTML}{C65D2E}
\definecolor{cliqueone}{HTML}{3B6FB6}
\definecolor{cliquetwo}{HTML}{2A8C82}
\definecolor{cliquethree}{HTML}{7A5AA6}

\begin{tikzpicture}[
  vertex/.style={
    circle,
    draw,
    line width=0.8pt,
    minimum size=6mm,
    inner sep=0pt,
    font=\small
  },
  sepold/.style={
    vertex,
    draw=separatorcolor,
    fill=separatorcolor!10
  },
  sepnew/.style={
    vertex,
    draw=separatorcolor,
    fill=separatorcolor!35,
    line width=1.2pt
  },
  oneold/.style={
    vertex,
    draw=cliqueone,
    fill=cliqueone!10
  },
  onenew/.style={
    vertex,
    draw=cliqueone,
    fill=cliqueone!35,
    line width=1.2pt
  },
  twoold/.style={
    vertex,
    draw=cliquetwo,
    fill=cliquetwo!10
  },
  twonew/.style={
    vertex,
    draw=cliquetwo,
    fill=cliquetwo!35,
    line width=1.2pt
  },
  threenew/.style={
    vertex,
    draw=cliquethree,
    fill=cliquethree!35,
    line width=1.2pt
  },
  separator edge/.style={
    draw=separatorcolor,
    line width=1.3pt
  },
  clique one edge/.style={
    draw=cliqueone,
    line width=0.9pt
  },
  clique two edge/.style={
    draw=cliquetwo,
    line width=0.9pt
  },
  clique three edge/.style={
    draw=cliquethree,
    line width=0.9pt
  },
  transition/.style={
    ->,
    >=stealth,
    draw=gray!70,
    line width=1pt
  },
  stage label/.style={
    font=\small\bfseries
  }
]

% ------------------------------------------------
% First row: W_1 -> W_2 -> W_3
% ------------------------------------------------

% Stage 1
%\node[stage label] at (0,4.15) {$G[W_1]$};
\node[sepnew] (w1-2) at (0,3) {$2$};

% Arrow W_1 -> W_2
\draw[transition] (0.45,3)--(2.65,3);

% Stage 2
%\node[stage label] at (3.4,4.15) {$G[W_2]$};
\node[sepold] (w2-2) at (3.4,3.5) {$2$};
\node[sepnew] (w2-3) at (3.4,2.5) {$3$};
\draw[separator edge] (w2-2)--(w2-3);

% Arrow W_2 -> W_3
\draw[transition] (3.85,3)--(5.9,3);

% Stage 3
%\node[stage label] at (7.15,4.15) {$G[W_3]$};
\node[sepold] (w3-2) at (7.5,3.5) {$2$};
\node[sepold] (w3-3) at (7.5,2.5) {$3$};
\node[onenew] (w3-1) at (6.3,3) {$1$};

\draw[separator edge] (w3-2)--(w3-3);
\draw[clique one edge]
  (w3-1)--(w3-2)
  (w3-1)--(w3-3);

% ------------------------------------------------
% Curved transition from W_3 to W_4
% ------------------------------------------------

% \draw[
%   transition,
%   rounded corners=5pt
% ]
% (8.0,3)
% --(8.8,3)
% --(8.8,0)
% --(8.55,0);

\draw[
  transition,
  rounded corners=5pt
]
% (8.0,3)
% --(11,3)
% --(11,0)
% --(8.55,0);
(8.0,3)
--(11.3,3)
--(11.3,0.6)
--(8.55,0.6);

% ------------------------------------------------
% Second row: W_4 -> W_5, read right to left
% ------------------------------------------------

\begin{scope}[yshift=0.6cm]
% Stage 4
%\node[stage label] at (7.2,1.2) {$G[W_4]$};
\node[sepold] (w4-2) at (7.2,0.5) {$2$};
\node[sepold] (w4-3) at (7.2,-0.5) {$3$};
\node[oneold] (w4-1) at (6.1,0) {$1$};
\node[twonew] (w4-4) at (8.2,0) {$4$};

\draw[separator edge] (w4-2)--(w4-3);
\draw[clique one edge]
  (w4-1)--(w4-2)
  (w4-1)--(w4-3);
\draw[clique two edge]
  (w4-4)--(w4-2)
  (w4-4)--(w4-3);

% Arrow W_4 -> W_5
\draw[transition] (5.75,0)--(4.15,0);

\begin{scope}[xshift=-1cm]
% Stage 5
%\node[stage label] at (3.4,1.2) {$G[W_5]=G$};
\node[sepold] (w5-2) at (3.4,0.5) {$2$};
\node[sepold] (w5-3) at (3.4,-0.5) {$3$};
\node[oneold] (w5-1) at (2.25,0) {$1$};
\node[twoold] (w5-4) at (4.5,0.55) {$4$};
\node[threenew] (w5-5) at (4.5,-0.55) {$5$};
\end{scope}

\end{scope}

\draw[separator edge] (w5-2)--(w5-3);
\draw[clique one edge]
  (w5-1)--(w5-2)
  (w5-1)--(w5-3);
\draw[clique two edge]
  (w5-4)--(w5-2)
  (w5-4)--(w5-3);
\draw[clique three edge]
  (w5-5)--(w5-2)
  (w5-5)--(w5-3);

\end{tikzpicture}
\caption{The simplicial elimination ordering $v_1=2,v_2=3,v_3=1,v_4=4,v_5=5$, displayed in the construction direction used by Algorithm~\ref{alg:chordal-orthogonal-innovation}.}
\label{F:five-vertex-elimination-ordering}
\end{figure}
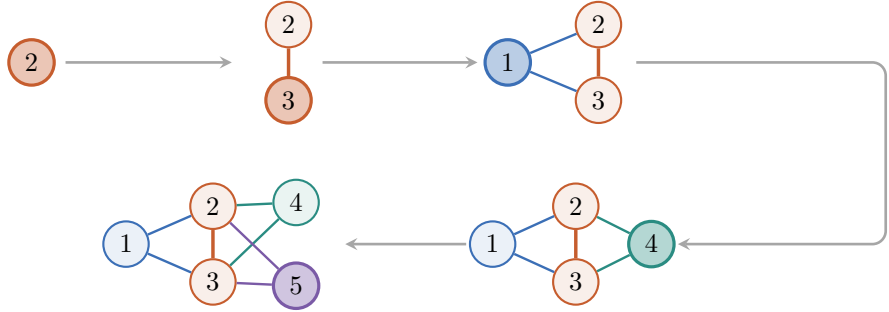

Thus, after the initial attachment of $v_2=3$ along the singleton clique $\{2\}$, each of the vertices $1,4,5$ is attached along the size-two clique $S=\{2,3\}$. Let $\varepsilon_1,\ldots,\varepsilon_4$ be the standard basis of $\R^4$, and set
\[
e_k:=(0,\varepsilon_{k-1}),
\qquad k=2,\ldots,5.
\]
The algorithm begins with
\[
u_{v_1}=u_2=(1,0,0,0,0).
\]
For $k=2$, we have $v_2=3$ and $S_2=\{2\}$. Therefore,
\[
A_2=(1),\qquad
g_2=\left(\frac53\right),\qquad
c_2=A_2^{-1}g_2=\left(\frac53\right),
\]
and hence
\[
p_2=\frac53u_2,
\qquad
\rho_2
=\sqrt{g_2^tA_2^{-1}g_2-1}
=\sqrt{\frac{25}{9}-1}
=\frac43.
\]
Since $e_2=(0,1,0,0,0)$, the algorithm gives
\[
u_3=p_2+\rho_2e_2
=\frac53u_2+\frac43e_2
=\left(\frac53,\frac43,0,0,0\right).
\]
For $k=3$, we have $v_3=1$ and $S_3=\{2,3\}$. Thus
\[
A_3=A_S
=
\begin{pmatrix}
1&\frac53\\[2pt]
\frac53&1
\end{pmatrix},
\qquad
g_3=g=\binom q q.
\]
Using $q=2t/\sqrt3$, we obtain
\[
\begin{aligned}
c_3
&=A_S^{-1}g
=
\begin{pmatrix}
-\frac9{16}&\frac{15}{16}\\[2pt]
\frac{15}{16}&-\frac9{16}
\end{pmatrix}
\binom q q
=\frac{3q}{8}\binom11
=\frac{\sqrt3\,t}{4}\binom11.
\end{aligned}
\]
Consequently,
\[
\begin{aligned}
p_3
=\frac{\sqrt3\,t}{4}u_2
  +\frac{\sqrt3\,t}{4}u_3
=\left(\frac{2t}{\sqrt3},
         \frac{t}{\sqrt3},0,0,0\right),
\end{aligned}
\]
while
\[
\rho_3
=\sqrt{g^tA_S^{-1}g-1}
=\sqrt{t^2-1}
=s.
\]
Since $e_3=(0,0,1,0,0)$, we obtain
\[
u_1=p_3+\rho_3e_3
=\left(\frac{2t}{\sqrt3},
       \frac{t}{\sqrt3},s,0,0\right).
\]
For $k=4$, we have $v_4=4$ and again
\[
S_4=\{2,3\},\qquad
A_4=A_S,\qquad
g_4=g.
\]
Therefore,
\[
c_4=\frac{\sqrt3\,t}{4}\binom11,
\qquad
p_4=\left(\frac{2t}{\sqrt3},
          \frac{t}{\sqrt3},0,0,0\right),
\qquad
\rho_4=s.
\]
Since $e_4=(0,0,0,1,0)$, it follows that
\[
u_4=p_4+\rho_4e_4
=\left(\frac{2t}{\sqrt3},
       \frac{t}{\sqrt3},0,s,0\right).
\]
Finally, for $k=5$, we have $v_5=5$ and
\[
S_5=\{2,3\},\qquad
A_5=A_S,\qquad
g_5=g.
\]
Thus
\[
c_5=\frac{\sqrt3\,t}{4}\binom11,
\qquad
p_5=\left(\frac{2t}{\sqrt3},
          \frac{t}{\sqrt3},0,0,0\right),
\qquad
\rho_5=s.
\]
Since $e_5=(0,0,0,0,1)$, we obtain
\[
u_5=p_5+\rho_5e_5
=\left(\frac{2t}{\sqrt3},
       \frac{t}{\sqrt3},0,0,s\right).
\]
Algorithm~\ref{alg:chordal-orthogonal-innovation} therefore produces
the points
\[
\begin{aligned}
u_1&=\left(\frac{2t}{\sqrt3},
           \frac{t}{\sqrt3},s,0,0\right),&
u_2&=(1,0,0,0,0),\\
u_3&=\left(\frac53,\frac43,0,0,0\right),&
u_4&=\left(\frac{2t}{\sqrt3},
           \frac{t}{\sqrt3},0,s,0\right),\\
u_5&=\left(\frac{2t}{\sqrt3},
           \frac{t}{\sqrt3},0,0,s\right).
\end{aligned}
\]
A direct calculation gives $[u_i,u_i]=1$ for all $i\in V$, and
\[
[u_2,u_3]=\frac53,
\qquad
[u_i,u_2]=[u_i,u_3]=q
\quad (i\in\{1,4,5\}),
\]
and
\[
[u_i,u_j]=t^2
\quad
(i\ne j,\ i,j\in\{1,4,5\}).
\]
Hence the Lorentz-Gram matrix of these points is precisely
$B^{(\tau)}$.

\medskip
\textit{\textbf{6.} Distortion blow-up.}
\medskip

\noindent This family makes the second part of Theorem~\ref{T:separator-size-distortion-dichotomy} concrete. Fix two distinct outer vertices, say $1$ and $4$. Every $1$--$4$ path in $G$ must meet the separator $S=\{2,3\}$, and the two-edge paths $1,2,4$ and $1,3,4$ are shortest. Since every edge joining an outer vertex to $S$ has length
\[
 \ell_\tau:=\ach  q
 =\ach \left(\frac{2\cosh\tau}{\sqrt3}\right),
\]
we obtain $d_G(1,4)=2\ell_\tau$. On the other hand, $b_{14}^{(\tau)}=t^2$, and hence the distance between the completed points is
\[
 d_{\Lob}(u_1,u_4)=\ach (t^2)
 =\ach (\cosh^2\tau).
\]
As $\tau\downarrow0$, we have $d_G(1,4)\to 2\ach \left(\frac2{\sqrt3}\right)>0$, and $d_{\Lob}(u_1,u_4) \to 0$. Consequently,
\[
 \dist (f_{A^{(\tau)}})
 \geq
 \frac{d_G(1,4)}{d_{\Lob}(u_1,u_4)}
 \longrightarrow\infty.
\]
Thus, this elementary common-edge configuration exhibits the unbounded-distortion phenomenon caused by a clique-tree separator of cardinality two.

\section{Applications and ramifications}
\label{S:connections-applications}

\subsection{Exact recovery from sparse hyperbolic measurements}
\label{SS:sparse-hyperbolic-recovery}

Incomplete pairwise distance information is a natural input to hyperbolic distance geometry. Tabaghi and Dokmani\'c \cite[Sections~3 and 4]{TabaghiDokmanic2020} address completion and denoising through semidefinite relaxations, followed by spectral reconstruction of the points. For fully observed hyperbolic distances, the \texttt{hydra} algorithm of Keller-Ressel and Nargang recovers the configuration up to isometry \cite[Theorem~3.1]{KellerResselNargang2020}. Our results connect with these approaches by addressing three questions for exact observations on chordal graphs: whether the measurements admit a hyperbolic realization, which completion the construction selects, and when this completion recovers the original configuration.

\medskip
{\it \textbf{1.} Realizability of the observations.}
\medskip

\noindent Let $G=(V,E)$ be a finite connected chordal graph with $n:=|V|\geq2$, and suppose that a distance $\lambda_{ij}>0$ is observed for each $\{i,j\}\in E$. Define the partial matrix $A$ by $a_{ii}:=1$, $a_{ij}:=\cosh\lambda_{ij}$ for $\{i,j\}\in E$. For each maximal clique $C$, set $\lambda_{ii}:=0$, choose an anchor $c\in C$, and form
\[
P_C^{(c)}
:=
\bigl(
\cosh\lambda_{ic}\cosh\lambda_{jc}
-\cosh\lambda_{ij}
\bigr)_{i,j\in C\setminus\{c\}}.
\]
Every entry of this matrix is determined by the observations. By Theorem~\ref{T:LG-ker-char}, the clique matrix $A_{C\times C}$ is Lorentz-Gram if and only if $P_C^{(c)}\lgeq 0$. Theorem~\ref{T:chordal-characterization} therefore gives an exact local criterion: the observations admit points $(u_i)_{i\in V}\subseteq\Lob(\R^{n-1})$ satisfying $d_{\Lob}(u_i,u_j)=\lambda_{ij}$ on every observed edge if and only if these anchored matrices are positive semidefinite on all maximal cliques. Thus global consistency can be checked using only fully observed subsets. Failure of a clique test also identifies a specific obstruction: if $z^tP_C^{(c)}z<0$ for some vector $z$, then the observations on $C$ cannot be realized in any hyperbolic space.

\medskip
{\it \textbf{2.} Selection of the canonical completion.}
\medskip

\noindent When the observations are realizable, the completion need not be unique. Under the nonsingularity hypotheses of Theorem~\ref{T:clique-product-completion}, the clique-product construction selects a completion explicitly. These hypotheses can also be checked through the anchored matrices: placing the anchor first gives
\[
A_{C\times C}\sim(1)\oplus\bigl(-P_C^{(c)}\bigr),
\]
so a realizable clique matrix is nonsingular precisely when $P_C^{(c)}\succ0$. Suppose that this condition holds on every maximal clique. Theorems~\ref{T:clique-product-completion} and~\ref{T:chordal-canonicity} then give the nonsingular completion $B_A$ and its realization by Algorithm~\ref{alg:chordal-orthogonal-innovation}. The completed distances
\[
\widehat d_{ij}:=\ach\bigl((B_A)_{ij}\bigr),
\qquad i,j\in V,
\]
agree exactly with the observations. Each missing distance is determined by transfers through observed separator blocks. For example, if adjacent maximal cliques $C,D$ meet in $S$, then for $x\in C\setminus S$ and $y\in D\setminus S$,
\[
\widehat d_{xy}
=
\ach\!\left(
A_{\{x\}\times S}A_{S\times S}^{-1}A_{S\times\{y\}}
\right).
\]
The general formula composes these transfers along the clique-tree path between the endpoint cliques. The selected completion is independent of the clique-tree and of the simplicial elimination ordering. By Theorem~\ref{T:chordal-canonicity}, it is characterized among nonsingular Lorentz-Gram completions by
\[
(B_A^{-1})_{ij}=0
\qquad
\text{whenever }i\neq j\text{ and }\{i,j\}\notin E.
\]
This gives a precise structural interpretation of the reconstruction rule. Its realization has hyperbolic span of dimension $n-1$; a prescribed smaller target dimension requires additional rank conditions.

\medskip
{\it \textbf{3.} Recovery of the original configuration.}
\medskip

\noindent To interpret the completed distances as recovered measurements, one must specify when the observations identify the unknown configuration. Chordality alone does not ensure this. For example, on the path $1-2-3$, fix $\lambda_{12}=a>0$ and $\lambda_{23}=b>0$. The points
\[
u_2=(1,0,0),\qquad
u_1=(\cosh a,\sinh a,0),\qquad
u_3=(\cosh b,\sinh b\cos\theta,\sinh b\sin\theta)
\]
realize these observations for every $0<\theta<\pi$, while
\[
\cosh d_{\Lob}(u_1,u_3)
=
\cosh a\cosh b-\sinh a\sinh b\cos\theta.
\]
Their Gram matrices are nonsingular, but the unobserved distance varies with $\theta$. The canonical completion selects $\theta=\pi/2$. Now suppose that the observations arise from an unknown nonsingular Lorentz-Gram matrix $K=(k_{ij})_{i,j\in V}$, so that $\lambda_{ij}=\ach(k_{ij})$ for $\{i,j\}\in E$. Every clique principal submatrix of $K$ is then a nonsingular Lorentz-Gram matrix. The uniqueness characterization in Theorem~\ref{T:chordal-canonicity} gives
\[
B_A=K
\quad\Longleftrightarrow\quad
(K^{-1})_{ij}=0
\quad\text{for every }i\neq j
\text{ with }\{i,j\}\notin E.
\]
Under this inverse-sparsity assumption, the reconstruction recovers all unobserved distances exactly. Equality of the Gram matrices also identifies the original and reconstructed configurations up to a hyperbolic isometry between their hyperbolic spans.

The number of observations required by this model can be small. If every maximal clique has at most $\omega$ vertices, a simplicial elimination ordering gives each vertex at most $\omega-1$ earlier neighbours, and hence $|E|\leq n(\omega-1)$. For fixed $\omega$, the number of measured distances is therefore linear in $n$. This is a statement about the observations required; explicitly listing all completed distances still requires quadratically many entries.

A complementary recovery mechanism is provided by landmarks. Keller-Ressel and Nargang \cite[Theorem~3.1]{KellerResselNargang2023} show that \texttt{L-hydra} recovers a configuration in $\Lob(\R^d)$ from landmark-to-landmark and landmark-to-point distances when the landmark vectors span $\R\oplus\R^d$. Choose $d+1$ linearly independent landmarks, indexed by $L$. Their observations determine $A_{L\times L}$ and every column $g_v:=A_{L\times\{v\}}$. Writing $u_v=\sum_{s\in L}(c_v)_s u_s$ gives $A_{L\times L}c_v=g_v$, and therefore
\[
k_{ij}=g_i^tA_{L\times L}^{-1}g_j,
\qquad
d_{\Lob}(u_i,u_j)=\ach(k_{ij}).
\]

This observation pattern is chordal, with maximal cliques $L\cup\{v\}$ for the nonlandmark vertices $v$. When $n>d+1$, these clique matrices have rank $d+1$ and are singular, placing them within the general completion setting of Theorem~\ref{T:chordal-characterization}. The transfer formula is forced by the fact that every point lies in the landmark span, so no orthogonal component remains undetermined. The required number of measured distances is $\binom{d+1}{2}+(d+1)(n-d-1)$.

\subsection{Hierarchical and phylogenetic data}
\label{S:hierarchical-phylogenetic-data}

Hyperbolic representations have been used to infer phylogenetic trees from sequence data \cite{Wilson} and to estimate distances for phylogenetic placement and tree updates \cite{Jiang}. The product-distance completion has a direct connection with these applications: it represents an additive tree metric by hyperbolic distances through an invertible scalar transformation. This connects the construction with classical reconstruction from leaf distances and permits explicit recovery guarantees under perturbations.

\medskip
{\it \textbf{1.} Exact representation of additive tree distances.}
\medskip

\noindent Let $T=(V,E)$ be a finite tree with $n:=|V|\geq2$ and positive branch lengths $(w_e)_{e\in E}$. Write $\delta(i,j):=\sum_{e\in P(i,j)}w_e$ for $i,j\in V$, for its path metric. Specify $a_{ii}:=1$ and $a_{ij}:=e^{w_e}$ on each edge $e=\{i,j\}$. Theorem~\ref{T:tree-path-product-completion} gives the Lorentz-Gram completion defined by $b_{ij}
=\prod_{e\in P(i,j)}e^{w_e}=e^{\delta(i,j)}$. Consequently, there are points $(u_i)_{i\in V}\subseteq\Lob(\R^{n-1})$ whose hyperbolic distances $D(i,j):=d_{\Lob}(u_i,u_j)$ satisfy 

\begin{equation}\label{E:phylogenetic-distance-transform} 
D(i,j)=\ach\!\left(e^{\delta(i,j)}\right),
\qquad
\delta(i,j)=\log\cosh D(i,j).
\end{equation}
The hyperbolic length of an edge is therefore $\lambda_e=\ach(e^{w_e})$, and conversely $w_e=\log\cosh\lambda_e$.

The transformation $t\mapsto\ach(e^t)$ was proposed by Matsumoto, Mimori, and Fukunaga \cite[Materials and methods, equations~(2)--(3)]{MatsumotoMimoriFukunaga2021} for phylogenetic embeddings, using the hyperbolic law of cosines at a right angle. Theorem~\ref{T:tree-path-product-completion} supplies an explicit simultaneous realization of all these transformed distances on a finite tree. As noted in Remark~\ref{R:product-distance-tree-metrics}, taking entrywise logarithms of the completion recovers the original additive metric. This exact realization allows the dimension to grow with $n$; it does not assert exact representability in a fixed low-dimensional hyperbolic space.

\medskip
{\it \textbf{2.} Recovery from leaf distances.}
\medskip

\noindent Let $X$ be the set of all leaves of $T$, and suppose that $|X|\geq3$ and every internal vertex has degree at least three. Assume that the hyperbolic distances $D(i,j)$ are known for all $i,j\in X$, while the internal vertices and tree topology are unknown. Equation~\eqref{E:phylogenetic-distance-transform} determines the additive leaf metric $\delta|_{X\times X}$. The uniqueness of the reduced tree representation of an additive metric \cite{Buneman} then determines $T$ and all its branch lengths $w_e$, up to an isomorphism fixing the leaf labels. If desired, the prescribed hyperbolic edge lengths are recovered as $\lambda_e=\ach(e^{w_e})$. In particular, when $T$ is binary, meaning that every internal vertex has degree three, the neighbor-joining algorithm \cite{SaitouNei1987,Atteson1999} applied to $\log\cosh D$ recovers the tree and its branch lengths exactly. For a more recent account on neighbor-joining algorithm, see \cite{Mihaescu-Levy-Pachter}.

The absence of internal vertices of degree two makes the individual branches identifiable: subdividing an edge leaves all leaf distances unchanged provided the new $w$-lengths sum to the old one. Without this assumption, only the tree obtained by suppressing such vertices and its summed branch lengths are determined. The reconstruction is unrooted; identifying a root requires additional information or assumptions.

The recovered splits also have a direct expression in the Lorentz-Gram entries. For four distinct leaves whose induced quartet has split $ij\mid k\ell$, let $h>0$ be the length of its central edge, measured in the $w$-metric after suppressing vertices of degree two. Cancellation of the four pendant lengths gives
\[
\delta(i,k)+\delta(j,\ell)
=\delta(i,\ell)+\delta(j,k)
=\delta(i,j)+\delta(k,\ell)+2h.
\]
Thus $b_{ij}b_{k\ell}<b_{ik}b_{j\ell}=b_{i\ell}b_{jk}$ and $h=\frac12\log\!\left(\frac{b_{ik}b_{j\ell}}{b_{ij}b_{k\ell}} \right)$. If the induced quartet is unresolved, the three products are equal. Hence the multiplicative four-point condition records both the quartet split and its central length.

\medskip
{\it \textbf{3.} Stability of topology recovery.}
\medskip

\noindent Suppose now that $T$ is binary and that the observed leaf dissimilarities $\widetilde D(i,j)$ are symmetric, nonnegative, and zero on the diagonal. Define
\[
\varepsilon:=\max_{i,j\in X}
\bigl|\widetilde D(i,j)-D(i,j)\bigr|,
\qquad
\widetilde\delta(i,j):=\log\cosh\widetilde D(i,j).
\]
Since $(\log\cosh t)'=\tanh t\in[0,1)$ for $t\geq0$, the mean value theorem yields $\max_{i,j\in X} \bigl|\widetilde\delta(i,j)-\delta(i,j)\bigr| \leq\varepsilon$. Atteson's radius theorem \cite{Atteson1999,Mihaescu-Levy-Pachter} therefore implies that neighbor-joining applied to $\widetilde\delta$ recovers the unrooted topology of $T$ whenever
\begin{equation}\label{E:phylogenetic-recovery-bound}
\varepsilon
<\frac12\min_{e\in E}w_e
=\frac12\min_{e\in E}\log\cosh\lambda_e.
\end{equation}
This is a sufficient condition for topology recovery; it does not assert exact recovery of branch lengths from noisy data. The bound follows by combining the nonexpansiveness of the scalar transformation with the classical neighbor-joining guarantee.

\medskip
{\it \textbf{4.} Rooted hierarchies and the scope of the model.}
\medskip

\noindent There is also an interpretation for equidistant rooted trees. Suppose that a root $r\in V$ is specified and every leaf is at the same $w$-distance $H$ from $r$. If $v$ is the least common ancestor of leaves $i,j$ and $h(v):=\delta(r,v)$, then
\[
\delta(i,j)=2\bigl(H-h(v)\bigr),
\qquad
D(i,j)=\ach\!\left(e^{2(H-h(v))}\right).
\]
The leaf metric $\delta$ is then an ultrametric. Since $t\mapsto\ach(e^t)$ is strictly increasing, $D$ is also an ultrametric and determines the same nested clusters, with transformed merge levels. The equal-depth assumption here concerns the additive branch lengths $w_e$.

These conclusions concern distances arising from the product-distance construction, or observations close enough
to such distances to satisfy \eqref{E:phylogenetic-recovery-bound}. Arbitrary hyperbolic distances need not become an additive tree metric under $\log\cosh$. Accordingly, applying this transformation to a learned embedding has an exact recovery interpretation only when the stated tree model is justified. The construction provides an exact geometric representation and a conditional recovery guarantee; a statistical guarantee for inference from sequence data requires an additional model relating those data to the observed distances.

\subsection{Further directions} \ 

\medskip
{\it \textbf{1.} Algorithms for bounded treewidth.} The clique-tree construction in Theorem~\ref{T:clique-product-completion} suggests efficient implementations when the maximal clique size is bounded. The canonical completion requires inversions only of clique and separator matrices, so its local algebraic operations involve matrices whose dimensions are controlled by the treewidth $\max_{C\in \ct}|C|-1$ rather than by the total number of vertices $|V|$. This distinction is particularly relevant when the completion is retained in its clique-tree representation: explicitly forming all missing entries necessarily requires quadratic output size. A detailed complexity analysis of the resulting algorithms lies beyond the scope of the present work.

\medskip
{\it \textbf{2.} Consistency diagnostics.} The chordality theorem also has an interpretation for consistency testing of partial hyperbolic distance data. For chordal specification graphs, global consistency is completely detected by the local clique conditions. In contrast, on a nonchordal graph all clique data may be individually realizable while the complete collection is globally inconsistent; the chordless-cycle construction in the proof of Theorem~\ref{T:chordal-characterization} gives an elementary instance of such an obstruction. This suggests using induced cycles and related minimal obstructions as diagnostic structures for inconsistent or anomalous sparse measurements.

% \medskip
% {\it \textbf{3.} Distributed hyperbolic reconstruction.} The clique-product construction may also be viewed as a negative-curvature counterpart of reconstruction from Euclidean distance measurements. For chordal measurement graphs, compatible local configurations can be assembled through their clique separators, and the missing cross-clique distances are determined by matrix-valued transfers along the clique-tree. This suggests distributed localization and reconstruction procedures for networks modeled in hyperbolic space, in which local computations and separator information replace a single global nonlinear embedding problem.

\medskip
{\it \textbf{3.} Distributed hyperbolic reconstruction.} A distributed reconstruction procedure divides a global problem into smaller local problems. Here the local subsystems are the maximal cliques of the measurement graph, whose edges represent the observed distances. For a chordal graph, these cliques form a clique-tree, and adjacent cliques $C$ and $D$ communicate through their common separator $S=C\cap D$. The clique-product construction determines the missing cross-clique entries by matrix-valued transfers through $A_{S\times S}^{-1}$, composed along paths in the clique-tree. This suggests reconstructing the clique configurations locally and exchanging only separator data, rather than solving a single global nonlinear embedding problem. Developing and analyzing such a distributed algorithm remains an interesting direction.

\subsection{Open questions}

We collect a few natural questions arising from the present article.

\begin{quest}[Canonical completion for singular data]
Theorem~\ref{T:chordal-characterization} permits singular maximal clique matrices, whereas the canonical, inverse-sparsity, and variational characterizations obtained later require nonsingularity. Is there a distinguished completion for singular
admissible data? If so, can it be characterized intrinsically, and is it obtained as a regularization-independent limit of nonsingular canonical completions?    
\end{quest}

\begin{quest}[The space of Lorentz-Gram completions]
Let $G$ be chordal and suppose that its maximal clique principal submatrices are nonsingular Lorentz-Gram matrices. Can the full family of Lorentz-Gram completions be parametrized explicitly in terms of the relative positions of the orthogonal components introduced across the clique separators? Can this parameterization be made intrinsic, and hence independent of the chosen clique-tree?    
\end{quest}

\begin{quest}[Completion on a fixed nonchordal graph]
Theorem~\ref{T:chordal-characterization} characterizes the graphs for which clique-wise admissibility guarantees completion for every choice of partial data. For a fixed nonchordal graph $G$, however, some admissible partial matrices remain completable. Can these matrices be characterized explicitly? In particular, can global compatibility be detected by conditions associated with the induced cycles of $G$, or with a controlled family of its chordal extensions?
\end{quest}

\begin{quest}[Infinite chordal graphs]
To what extent do the existence, canonicality, inverse-sparsity, and variational results of this paper extend to infinite chordal graphs? What additional assumptions on the graph and the prescribed Lorentz-Gram data are necessary for such extensions?
\end{quest}

\begin{quest}[Stability of the local-to-global construction]
Suppose that the prescribed clique matrices are close to Lorentz-Gram matrices and that the separator matrices are uniformly well conditioned. Can one bound the perturbation required to obtain globally realizable data in terms of the local defects? Likewise, can the sensitivity of the clique-product completion be controlled in terms of the separator condition numbers and the geometry of the clique-tree?
\end{quest}

\subsection*{AI declaration} The original proofs in this paper were checked with the assistance of ChatGPT and Claude. We also benefited from these models for orientation through applications and general organization of the article. The authors take full responsibility for all content.

\subsection*{Acknowledgments}

M.P. was partially supported by a Simons Foundation Collaboration Grant. P.K.V. was supported by the Centre de recherches mathématiques and the CRM-Laval Postdoctoral Fellowship at Université Laval.

\end{document}